\documentclass[10pt,a4paper]{article}
\usepackage{amssymb}
\usepackage{graphicx}
\usepackage{amsmath,amsthm}
\usepackage{amsfonts,amssymb}
\usepackage{url,paralist}
\usepackage{anysize}
\usepackage[utf8]{inputenc}
\usepackage[arrow,curve,matrix,tips,2cell]{xy}  
  \SelectTips{eu}{10} \UseTips
  \UseAllTwocells
  \usepackage{tikz,calc,etex}
  \usetikzlibrary{matrix, arrows, calc,backgrounds,fit,shapes}
\usepackage{hyperref}
\usepackage{color}
\usepackage{enumerate}

\makeatletter
\newcommand\getdepthofnode[2]{%
    \pgfextracty{#1}{\pgfpointanchor{#2}{base}}%
    \pgfextracty{\pgf@ya}{\pgfpointanchor{#2}{south}}% \pgf@xa is a length defined by PGF for temporary storage. No need to create a new temporary length.
    \addtolength{#1}{-\pgf@ya}%
}
\newcommand\getheightofnode[2]{%
    \pgfextracty{#1}{\pgfpointanchor{#2}{north}}%
    \pgfextracty{\pgf@ya}{\pgfpointanchor{#2}{base}}% \pgf@xa is a length defined by PGF for temporary storage. No need to create a new temporary length.
    \addtolength{#1}{-\pgf@ya}%
}
\newcommand\getwidthofnode[2]{%
    \pgfextractx{#1}{\pgfpointanchor{#2}{east}}%
    \pgfextractx{\pgf@xa}{\pgfpointanchor{#2}{west}}% \pgf@xa is a length defined by PGF for temporary storage. No need to create a new temporary length.
    \addtolength{#1}{-\pgf@xa}%
}
\makeatother

\newcommand{\tikzhline}[5]{%
\begingroup
\locdimen\Depth
\locdimen\MaxDepth
\locdimen\WidthStart
\locdimen\MaxWidthStart
\locdimen\WidthEnd
\locdimen\MaxWidthEnd
\setlength{\MaxDepth}{0pt}
\foreach \col in {1,...,\the\pgfmatrixcurrentcolumn}{%
\getdepthofnode{\Depth}{#1-#2-\col}
\setlength{\MaxDepth}{\maxof{\MaxDepth}{\Depth}}
\global\MaxDepth=\MaxDepth
}
\setlength{\MaxWidthStart}{0pt}
\foreach \row in {1,...,\the\pgfmatrixcurrentrow}{%
\getwidthofnode{\WidthStart}{#1-\row-#3}
\setlength{\MaxWidthStart}{\maxof{\MaxWidthStart}{\WidthStart}}
\global\MaxWidthStart=\MaxWidthStart
}
\setlength{\MaxWidthEnd}{0pt}
\foreach \row in {1,...,\the\pgfmatrixcurrentrow}{%
\getwidthofnode{\WidthEnd}{#1-\row-#4}
\setlength{\MaxWidthEnd}{\maxof{\MaxWidthEnd}{\WidthEnd}}
\global\MaxWidthEnd=\MaxWidthEnd
} 
\draw[#5] ($(#1-#2-#3.base)-(\MaxWidthStart/2,\MaxDepth+\pgfmatrixrowsep/2)$)--($(#1-#2-#4.base)-(-\MaxWidthEnd/2,\MaxDepth+\pgfmatrixrowsep/2)$);
\endgroup
}

\newcommand{\tikzvline}[5]{%
\begingroup
\locdimen\Width
\locdimen\MaxWidth
\locdimen\HeightStart
\locdimen\MaxHeightStart
\locdimen\DepthEnd
\locdimen\MaxDepthEnd
\setlength{\MaxWidth}{0pt}
\foreach \row in {1,...,\the\pgfmatrixcurrentrow}{%
\getwidthofnode{\Width}{#1-\row-#4}
\setlength{\MaxWidth}{\maxof{\MaxWidth}{\Width}}
\global\MaxWidth=\MaxWidth
}
\setlength{\MaxHeightStart}{0pt}
\foreach \col in {1,...,\the\pgfmatrixcurrentcolumn}{%
\getheightofnode{\HeightStart}{#1-#2-\col}
\setlength{\MaxHeightStart}{\maxof{\MaxHeightStart}{\HeightStart}}
\global\MaxHeightStart=\MaxHeightStart
}
\setlength{\MaxDepthEnd}{0pt}
\foreach \col in {1,...,\the\pgfmatrixcurrentcolumn}{%
\getdepthofnode{\DepthEnd}{#1-#3-\col}
\setlength{\MaxDepthEnd}{\maxof{\MaxDepthEnd}{\DepthEnd}}
\global\MaxDepthEnd=\MaxDepthEnd
} 
\draw[#5] ($(#1-#2-#4.base)+(\MaxWidth/2+\pgfmatrixcolumnsep/2,\MaxHeightStart)$)--($(#1-#3-#4.base)+(\MaxWidth/2+\pgfmatrixcolumnsep/2,-\MaxDepthEnd)$);
\endgroup
}

\newtheorem{theorem}{Theorem}[section]

\newtheorem{lemma}[theorem]{Lemma}
\newtheorem{corollary}[theorem]{Corollary}
\newtheorem{conjecture}[theorem]{Conjecture}
\theoremstyle{definition}
\newtheorem{definition}[theorem]{Definition}
\theoremstyle{remark}

\newtheorem{remark}[theorem]{Remark}

\newcommand{\RR}{\mathbb{R}}

\newcommand{\QQ}{\mathbb{Q}}
\newcommand{\ZZ}{\mathbb{Z}}
\newcommand{\FF}{\mathbb{F}}
\newcommand{\BB}{\mathrm{B}}
\newcommand{\EE}{\mathrm{E}}
\newcommand{\ArA}{\mathcal{A}}
\newcommand{\ArB}{\mathcal{B}}
\newcommand{\DiA}{\mathfrak{X}}
\newcommand{\hcolim}{\mathrm{hocolim}}
\newcommand{\colim}{\mathrm{colim}}
\newcommand\Sym{\mathfrak S}

\newcommand{\To}{\longrightarrow} %long function arrow
\newcommand{\sgn}{\textnormal{sign}} %sign of a permutation or a real number
\newcommand{\Index}[1]{\textnormal{Index}_{#1}} %cohomological ideal-valued index_G
\newcommand{\Impl}{\Longrightarrow} %long implication arrow
\newcommand{\codim}{\textnormal{codim}} %codimension
\newcommand{\id}{\textnormal{id}} %identity  

\newcommand{\smallToggle}{\small}

\newcommand{\cat}{\operatorname{cat}}
\newcommand{\secat}{\operatorname{secat}}
\newcommand{\ind}{\operatorname{ind}}
\newcommand{\pr}{\operatorname{pr}}
\newcommand{\res}{\operatorname{res}}
\newcommand{\trf}{\operatorname{trf}}
\newcommand{\wgt}{\operatorname{wgt}}
\newcommand{\conve}{\operatorname{Conv}}
\begin{document}
% -------------------------------------------------------------------%

% -------------------------------------------------------------------%
\title%[] % <---short title
{Equivariant Topology of Configuration Spaces} % <--- long title
% -------------------------------------------------------------------%

% -------------------------------------------------------------------%
\author{Pavle V. M. Blagojevi\'{c}\thanks{The research leading to
    these results has received funding from the European Research Council under
    the European Union's Seventh Framework Programme (FP7/2007-2013) /
    ERC Grant agreement no.~247029-SDModels. Also supported by the grant ON 174008 of the Serbian
    Ministry of Education and Science.} \\
  \smallToggle Matemati\v cki Institut SANU\\[-3pt]
  \smallToggle Knez Mihailova 36\\[-3pt]
  \smallToggle 11001 Beograd, Serbia\\[-3pt]
  \smallToggle \url{pavleb@mi.sanu.ac.rs}\\[-3pt]
  \smallToggle and\\[-3pt]
  \smallToggle Institut f\" ur Mathematik, Freie Universit\"at Berlin\\[-3pt]
  \smallToggle Arnimallee 2\\[-3pt]
  \smallToggle 14195 Berlin, Germany\\[-3pt]
  \smallToggle \url{blagojevic@math.fu-berlin.de}
  \and 
  \setcounter{footnote}{6} 
  Wolfgang L\" uck\thanks{Research supported by a Leibniz award of the German Research Association DFG.}\\
  \smallToggle Mathematisches Institut der Universit\"at Bonn\\[-3pt]
  \smallToggle Endenicher Allee 60\\[-3pt]
  \smallToggle 53115 Bonn, Germany\\[-3pt]
  \smallToggle \url{wolfgang.lueck@him.uni-bonn.de} \and
  \setcounter{footnote}{6} 
  G\"unter M. Ziegler$^*$\thanks{$^*$The research leading to
    these results has received funding from the European Research Council under
    the European Union's Seventh Framework Programme (FP7/2007-2013) /
    ERC Grant agreement no.~247029-SDModels.} \\[-3pt]
  \smallToggle Institut f\" ur Mathematik, Freie Universit\"at Berlin\\[-3pt]
  \smallToggle Arnimallee 2\\[-3pt]
  \smallToggle 14195 Berlin, Germany\\[-3pt]
  \smallToggle \url{ziegler@math.fu-berlin.de}}
  \date{\today}
% -------------------------------------------------------------------%

\maketitle

% -------------------------------------------------------------------%
\begin{abstract}\noindent
  We study the Fadell--Husseini index of the configuration space $F(\RR^d,n)$
  with respect to various subgroups of the symmetric group $\Sym_n$.  For $p$
  prime and $k\ge 1$, we compute $\Index{\ZZ/p}(F(\RR^d,p);\FF_p)$ and partially describe $\Index{(\ZZ/p)^k}(F(\RR^d,p^k);\FF_p)$.   
  In this process we obtain results of independent interest, including: 
  (1) an extended equivariant Goresky--MacPherson formula, 
  (2) a complete description of the top homology of the partition lattice $\Pi_p$ as an $\FF_p[\ZZ_p]$-module, and 
  (3) a generalized Dold theorem for elementary abelian groups.
 
  The results on the Fadell--Husseini index yield a new proof of the
  Nandakumar \& Ramana Rao conjecture for primes.  For $n=p^k$ a
  prime power, we compute the Lusternik--Schnirelmann category
  $\cat(F(\RR^d,n)/\Sym_n)=(d-1)(n-1)$.
  Moreover, we extend coincidence results related to the Borsuk--Ulam theorem, 
  as obtained by Cohen \& Connett, Cohen \& Lusk, and Karasev \& Volovikov.
  \\[2mm]
\end{abstract}

% -------------------------------------------------------------------%

%%%%%%%%%%%%%%%%%%%%%%%%%%%%%%%%%%%%%%%%%%%%%%%%%%%%%%%%%%%%%%%%%%%%%%%%%%%%%%%%%%%%%
%%%%%%%%%%%%%%%%%%%%%%%%%%%%%%%%%%%%%%%%%%%%%%%%%%%%%%%%%%%%%%%%%%%%%%%%%%%%%%%%%%%%%
% -----------------------------------------------------------------------------------%
\section{Introduction and statement of main results}
\label{Sec:Intro}
% -----------------------------------------------------------------------------------%
%%%%%%%%%%%%%%%%%%%%%%%%%%%%%%%%%%%%%%%%%%%%%%%%%%%%%%%%%%%%%%%%%%%%%%%%%%%%%%%%%%%%%
%%%%%%%%%%%%%%%%%%%%%%%%%%%%%%%%%%%%%%%%%%%%%%%%%%%%%%%%%%%%%%%%%%%%%%%%%%%%%%%%%%%%%

%%%%%%%%%%%%%%%%%%%%%%%%%%%%%%%%%%%%%%%%%%%%%%%%%%%%%%%%%%%%%%%%%%%%%%%%%%%%%%%%%%%%%

\subsection{Configuration spaces}
The \emph{configuration space} of $n$ labeled points in the topological space
$X$ is the space
\[
F(X,n)=\{(x_1,\ldots,x_n)\in X^n~:~x_i\neq x_j\text{ for all }i\neq j\}\subset X^n.
\]
The symmetric group $\Sym_n$ naturally acts on $F(X,n)$ by permuting the points $x_1,\ldots ,x_n$.

We refer to F.~Cohen \cite{Cohen} and Fadell \& Husseini \cite{FadellHusseini:book} for background on configuration spaces as well as for references to the external literature in this context.

%%%%%%%%%%%%%%%%%%%%%%%%%%%%%%%%%%%%%%%%%%%%%%%%%%%%%%%%%%%%%%%%%%%%%%%%%%%%%%%%%%%%%

\subsection{The Fadell--Husseini index}
\label{subsec:Fadell--Huesseini_index_intro}

In this paper the focus is on the Fadell--Husseini index of the configuration
space $F(\RR^d,n)$ with respect to different subgroups $G$ of the symmetric group $\Sym_n$.

Let $G$ be a finite group acting on the space $Y$, and let $R$ be a commutative
ring with unit.  The \emph{Fadell--Husseini index} of $Y$ with respect to the
group $G$ and coefficients $R$ is the kernel ideal of the map in equivariant
cohomology induced by the $G$-equivariant map $p_Y \colon Y\to \mathrm{pt}$:
\[
\Index{G}(Y;R) := \ker\bigl(p_Y^*\colon H^*_{G}(\mathrm{pt},R)\longrightarrow H^*_{G}(Y,R)\bigr)
= \ker\bigl(H^*(\BB G,R)\longrightarrow H^*(\EE G\times _{G}Y,R)\bigr).
\]
The main property of the Fadell--Husseini index is that it yields a necessary
condition for the existence of a $G$-equivariant map $Y \to Z$, namely, that $\Index{G}(Z;R)
\subseteq \Index{G}(Y;R)$ must hold.

To study the Fadell--Husseini index of the configuration space $F(\RR^d,n)$, we
have to understand the Serre spectral sequence associated to the fibration
\[
F(\RR^d,n)\longrightarrow \EE G\times_{G}F(\RR^d,n)\longrightarrow \BB G
\]
whose $E_2$-term is given by
\[
E_2^{r,s}=H^r(\BB G;\mathcal{H}^s(F(\RR^d,n);R))\cong H^r(G;H^s(F(\RR^d,n);R)).
\]
Here $\mathcal{H}^*$ denotes the cohomology with local coefficients where the local coefficient system is given by the action of $\pi_1(\BB G)\cong G$ on the cohomology $H^*(F(\RR^d,n);R)$.

\noindent In order to compute this spectral sequence we need to determine
\begin{compactitem} %[$\circ$]
\item the $E_2$-term of the spectral sequence. 
For this we need to determine the $R[G]$-module structure on the cohomology $H^*(F(\RR^d,n);R)$, see Sections~\ref{Sec:GM} and~\ref{Sec:CohomologyOfCS-as-module};

\item the rows of the $E_2$-term of the spectral sequence as $H^*(G;R)$-modules, see Sections~\ref{Sec:FH-Index-I};

\item the differentials of this spectral sequence as $H^*(G;R)$-morphisms, see Section~\ref{Sec:DiffSSSeq} and ~\ref{Sec:FH-Index-I}.
\end{compactitem}
Utilizing all these data, we derive in Section~\ref{Sec:FH-Index-I} the following results for a prime $p$:
\begin{compactenum}[\rm (1)]
\item the complete description of the Fadell--Husseini index of the
  configuration space $F(\RR^d,p)$ with respect to the group $\ZZ/p\leq\Sym_p$
  and coefficients $\FF_p$, in Theorem~\ref{Th:IndexF(X,p)};
\item a partial estimate of the Fadell--Husseini index of the configuration
  space $F(\RR^d,p^k)$ with respect to the regularly embedded subgroup $(\ZZ/p)^k$ of the symmetric group $\Sym_{p^k}$ and
  coefficients $\FF_p$, in   Theorem~\ref{Th:Estimate-IndexF(X,p_upper_k)};
\item for $n = p^k$ a prime power, the existence of a non-zero element in the
  difference
  \[
  H^{(d-1)(n-1)}_{\Sym_n^{(p)}}(\mathrm{pt};\ZZ){\setminus}\Index{\Sym_n^{(p)}}(F(\RR^d,n);\ZZ).
  \]
\end{compactenum}
This last result is not obtained via spectral sequence calculations.  Instead,
using our results from~\cite[Section~4]{B-Z-New}, we identify the non-zero
obstruction element for the existence of an $\Sym_n^{(p)}$-equivariant map
$F(\RR^d,n)\to S(W_n^{\oplus(d-1)})$ with the appropriate Euler class of the vector bundle
\[
W_n^{\oplus(d-1)} \longrightarrow F(\RR^d,n)\times_{\Sym_n^{(p)}} W_n^{\oplus(d-1)}\longrightarrow
F(\RR^d,n)/\Sym_n^{(p)},
\]
where $\Sym_n^{(p)}$ is a $p$-Sylow subgroup and $W_n$ a specific orthogonal
representation.

%%%%%%%%%%%%%%%%%%%%%%%%%%%%%%%%%%%%%%%%%%%%%%%%%%%%%%%%%%%%%%%%%%%%%%%%%%%%%%%%%%%%%

\subsection{Further results}
\label{subsec:Further_Results}

As a by-product of the Fadell--Husseini index calculations, we reprove some
known facts and obtain new results that are of independent interest.  For
example, we get:
\begin{compactenum}[\rm (1)]
\item an equivariant Goresky--MacPherson formula, in Theorem~\ref{Th:EqGM};
\item the $R[\Sym_n]$-module structure on the cohomology of the configuration
  space $H^*(F(\RR^d,n);R)$, in Theorem~\ref{Th:ModuleStructureOnCohomology};
\item an extended generalization of Dold's theorem for elementary abelian
  groups, in Theorem~\ref{Th:Dold};
\item for a prime $p$ the $\FF_p[\ZZ/p]$-module structure on the top homology of
  the proper part of the partition lattice $\Pi_p$, in Corollary~\ref{cor:RepresPi_p}.
\end{compactenum}

%%%%%%%%%%%%%%%%%%%%%%%%%%%%%%%%%%%%%%%%%%%%%%%%%%%%%%%%%%%%%%%%%%%%%%%%%%%%%%%%%%%%%

\subsection{Applications}
\label{subsec:Applications}

We are interested in the following two conjectures, one from convex geometry and the other from algebraic topology.

\begin{compactitem} %[$\circ$]
\item The Nandakumar \& Ramana--Rao conjecture: For any planar
  convex body $K$ and any natural number $n>1$ there exists a partition of the
  plane into $n$ convex pieces $P_1,\ldots ,P_n$ such that
  \[
  \mathrm{area}(P_1\cap K)=\cdots =\mathrm{area}(P_n\cap K) \quad \text{and}
  \quad \mathrm{perimeter}(P_1\cap K)=\cdots =\mathrm{perimeter}(P_n\cap K).
  \]

\item The Lusternik--Schnirelmann category of the configuration space
  $F(\RR^d,n)/\Sym_n$ of unordered pairwise distinct points in $\RR^d$ is equal
  to the cohomological dimension of the configuration space $F(\RR^d,n)$, i.e.,
  \[
  \cat(F(\RR^d,n)/\Sym_n)=(d-1)(n-1).
  \]
  This has been conjectured by Roth~\cite[Conjecture~1.3]{Roth}. (See also Karasev \cite[Lemma~6 and Theorem~9]{Karasev:genus-and-category}
  for partial results.)
\end{compactitem}
The results on the Fadell--Husseini index are used in Section~\ref{Sec:App} as the main ingredient to give a new proof of the Nandakumar \& Ramana--Rao
conjecture for $n = p$ and give a proof of the second conjecture for $n = p^k$, where $p$ is a prime.
The partial calculation of the Fadell--Husseini index in the case $n = p^k$, Section~\ref{Sec:FH-Index-II}, allows us to extend and improve Borsuk--Ulam type coincidence results by Cohen \& Connett~\cite{Cohen-Connett}, Cohen \& Lusk~\cite{Cohen-Lusk}, and Karasev \& Volovikov~\cite{Karasev-Volovikov}.

\subsection{Connections with classical results on configuration spaces}
\label{subsec:Connections}

In this section we point out connections and overlaps of our results with classical results by F.~Cohen, P.~May and F.~Cohen \& L.~Taylor.

\subsubsection{Cohomology of the configuration space as a module}

The configuration space $F(X,n)$ admits a free action of the symmetric group $\Sym_n$. Consequently, both homology $H_*(F(X,n);R)$ and cohomology $H^*(F(X,n);R)$ of the configuration space, with coefficients in any ring $R$, have the natural structure of an $R[G]$-module for any subgroup $G$ of $\Sym_n$.

Frederick Cohen, in his landmark paper from 1976  \cite[Section 7]{Cohen}, described the $\FF[\Sym_n]$-module structure on $H_*(F(\RR^d,n);\FF)$ and $H^*(F(\RR^d,n);\FF)$, where $\FF$ denotes an arbitrary field, as follows.
The basis for (co)homology of the configuration space that arose from the work of Fadell \& Neuwirth $\{\beta_{i,j}\}$ was modified to a basis $\{\alpha_{i,j}\}$ that allowed the author to describe the action in a concise way \cite[Lemma 7.1]{Cohen}.
The action of an arbitrary transposition $\tau_r=(r,r+1)$ on this basis was given in \cite[Proposition 7.2, Corollary 7.4]{Cohen}.
A more systematic description of the module structure on the cohomology $H^*(F(\RR^d,n);R)$, in the language of the representation theory, was given by F.~Cohen \& Taylor in 1993 \cite[Sections 3 and 5]{Cohen-Taylor}.  
Further on, a description of the top $(d-1)(n-1)$-th (co)homology of the configuration space $F(\RR^d,n)$ as a $\ZZ[\Sym_n]$-module in the language of Lie algebras was given by F.~Cohen in \cite[Theorem 6.1]{Cohen95}.

A paper by Ossa from 1996 presented similar results: It described the action of transpositions on the generators of $H^*(F(\RR^d,n);\ZZ)$ and  decomposed $H^*(F(\RR^d,n);\QQ)$ into a direct sum of induced representations \cite[Section 2]{Ossa}.

\subsubsection{Cohomology of the quotient configuration spaces}

The cohomology of the quotient configuration space $F(X,n)/G$ for a subgroup $G$ of the symmetric group $\Sym_n$, with coefficients in the ring $R$, can be studied via the Serre spectral sequence of the fibration
\[
F(X,n)\longrightarrow \EE G\times_{G}F(X,n)\longrightarrow \BB G,
\]
with the $E_2$-term given by
\[
E_2^{r,s}=H^r(\BB G;\mathcal{H}^s(F(X,n);R))\cong H^r(G;H^s(F(X,n);R)).
\]
The key ingredient in  the computation of this spectral sequence is the $R[G]$-module structure on the coefficients $H^*(F(X,n);R)$.

In the case when $p$ is an odd prime, using the previously described action of the transpositions on the cohomology   $H^*(F(\RR^d,p);\FF_p\otimes\sgn^{\varepsilon})$, $\varepsilon=0,1$ (twisted, or not, with the sign representation) F.~Cohen in \cite[Sections 5 and 8 to 11]{Cohen} analyzed the spectral sequences associated to the fibrations
\[
F(\RR^d,p)\longrightarrow \EE \ZZ/p\times_{\ZZ/p}F(\RR^d,p)\longrightarrow \BB \ZZ/p
\quad
\text{and}
\quad
F(\RR^d,p)\longrightarrow \EE \Sym_p\times_{\Sym_p}F(\RR^d,p)\longrightarrow \BB \Sym_p.
\]
The corresponding morphism of spectral sequences induced by the restriction turns out to be a monomorphism between $E_2$-terms.
One of many results derived in this framework is the following isomorphism of algebras \cite[Theorem 5.2]{Cohen}:
\begin{equation}
\label{revision:1}
H^*(F(\RR^d,p)/\Sym_p;\FF_p)\cong H^{\leq (d-1)(p-1)}(\Sym_p;\FF_p),
\end{equation}
which holds for odd integers $d\ge3$.
An analogous result in the case of the twisted coefficients $\FF_p\otimes\sgn$ was given in \cite[Theorem 5.3]{Cohen}.

The paper of Ossa \cite[Section 3]{Ossa} gave a brief account of these results.

\subsubsection{Fadell--Husseini index of the configuration space}

The Fadell--Husseini index of the configuration space $F(\RR^d,n)$, with respect to a subgroup $G$ of the symmetric group $\Sym_n$ and coefficients in the ring $R$, is the kernel ideal
\[
\Index{G}(F(\RR^d,n);R) := \ker\bigl(H^*_{G}(\mathrm{pt},R)\rightarrow H^*_{G}(F(\RR^d,n),R)\bigr).
\]

From the results of F.~Cohen in \cite{Cohen}, assuming that $p$ is an odd prime,  the following informations about the Fadell--Husseini index of the configuration space can be deduced:

\smallskip
\begin{compactenum}
\item[\rm (A)]  The Vanishing theorem \cite[Theorem 8.2]{Cohen} can be used to derive that
\[
\Index{\ZZ/p}(F(\RR^d,p);\FF_p)=H^{\geq (d-1)(p-1)+1}(\ZZ/p;\FF_p).
\]
\item[\rm (B)]  The description \eqref{revision:1} of the cohomology $H^*(F(\RR^d,p)/\Sym_p;\FF_p)$ given in \cite[Theorem 5.2]{Cohen} implies that for an odd integer $d\ge3$
\[
\Index{\Sym_p}(F(\RR^d,p);\FF_p)=H^{\geq (d-1)(p-1)+1}(\Sym_p;\FF_p).
\]
\end{compactenum}

\noindent
Further on, let us consider the vector bundle $\xi_{d,n}$ given by
  \[
  \RR^n\longrightarrow F(\RR^d,n)\times_{\Sym_n}\RR^n\longrightarrow F(\RR^d,n)/\Sym_n,
  \]
where $\Sym_n$ acts on $\RR^n$ by permuting the coordinates.
Then:

\smallskip
\begin{compactenum}
\item[\rm (C)] F.~Cohen \& Handel, using the little cube operad proved in \cite[Lemma 3.2]{Cohen-Handel} that in the case when $n$ is a power of $2$ 
\[
w_{n-1}(\xi_{2,n})\notin \Index{\Sym_n}(F(\RR^2,n);\FF_2).
\]
Here $w_{n-1}(\xi_{2,n})$ denotes the $(n-1)$-th Stiefel--Whitney class of the vector bundle $\xi_{2,n}$.
\item[\rm (D)]  In the case when both $d$ and $n$ are powers of $2$, Chisholm extended the result of F.~Cohen \& Handel and proved in \cite[proof of Lemma 3]{Chisholm} that
\[
w_{n-1}^{d-1}(\xi_{d,n})\notin \Index{\Sym_n}(F(\RR^d,n);\FF_2).
\]
\item[\rm (E)] With some additional work, using the result of Gromov in \cite[Non-vanishing lemma 5.1 and Remark after it]{Gromov}, one can obtain that in the case when $n$ is a power of $2$ and $d\geq2$:
\[
w_{n-1}^{d-1}(\xi_{d,n})\notin \Index{\Sym_n}(F(\RR^d,n);\FF_2).
\]
\end{compactenum}

\subsubsection*{Acknowledgments}

We are grateful to Frederick Cohen, Roman Karasev, Peter Landweber, Jim Stasheff, and Volkmar Welker for valuable discussions, comments, and pointers to the literature. Thanks to a referee for excellent comments and recommendations.

%%%%%%%%%%%%%%%%%%%%%%%%%%%%%%%%%%%%%%%%%%%%%%%%%%%%%%%%%%%%%%%%%%%%%%%%%%%%%%%%%%%%%
%%%%%%%%%%%%%%%%%%%%%%%%%%%%%%%%%%%%%%%%%%%%%%%%%%%%%%%%%%%%%%%%%%%%%%%%%%%%%%%%%%%%%
% -----------------------------------------------------------------------------------%
\section{An Equivariant Goresky--MacPherson formula}
\label{Sec:GM}
% -----------------------------------------------------------------------------------%
%%%%%%%%%%%%%%%%%%%%%%%%%%%%%%%%%%%%%%%%%%%%%%%%%%%%%%%%%%%%%%%%%%%%%%%%%%%%%%%%%%%%%
%%%%%%%%%%%%%%%%%%%%%%%%%%%%%%%%%%%%%%%%%%%%%%%%%%%%%%%%%%%%%%%%%%%%%%%%%%%%%%%%%%%%%

% -----------------------------------------------------------------------------------%

The main result of this section is the following theorem that is a generalization of a result by Sundaram 
\& Welker \cite[Theorem 2.5, page 1397]{SunWel}; explanations and the proof will follow below.

\begin{theorem}[Equivariant Goresky--MacPherson formula]
  \label{Th:EqGM}
  Let $\rho\colon G\to \mathrm{O}(d)$ be an orthogonal action of a finite group
  $G$ on the Euclidean space $E\cong\RR^d$.  Consider a $G$-invariant
  arrangement of linear subspaces $\ArA =\{V_1,\ldots,V_k\}$ in $E$.
  Assume that\\[2mm]
  {\bf (R)} The coefficient ring $R$ is a principal ideal domain and for every
  $V\in L_{\ArA}^{>\hat{0}}$ the homology groups $H_{*}(\Delta(\hat{0},V);R)$
  are free $R$-modules; and \\[2mm]
  {\bf (C)} The arrangement $\ArA$ is a $c$-arrangement for some integer $c>1$.
  \\[2mm]
  Then:
  \begin{compactenum}[\rm (i)]
  \item \label{Th:EqGM-Formula-1} For the homology of the link of the arrangement $\ArA$ there is an isomorphism of $R[G]$-modules
    \begin{equation*}
      H_{i}(D_{\ArA};R)\cong\bigoplus_{r+s=i}\; \bigoplus_{V\in L_{\ArA}^{>\hat{0}}/G} 
      \ind^G_{G_V}\tilde{H}_{r-1}(\Delta(\hat{0},V);R)\otimes_R \tilde{H}_s(S(V);R) ,
    \end{equation*}
    where $S(V)$ denotes the unit sphere in the linear subspace $V$.

  \item \label{Th:EqGM-Formula-2} For the cohomology of the complement of the arrangement $\ArA$ there is an isomorphism of $R[G]$-modules
    \begin{equation*}
      H^{i}(M_{\ArA};R)\cong\mathcal{R}\otimes\bigoplus_{r+s=d-i-2}\; \bigoplus_{V\in L_{\ArA}^{>\hat{0}}/G} 
      \ind^G_{G_V}\tilde{H}_{r-1}(\Delta(\hat{0},V);R)\otimes_R \tilde{H}_s(S(V);R) ,
    \end{equation*}
    where $\mathcal{R}$ is the $R[G]$-module whose underlying $R$-module is $R$
    and for which $g\in G$ acts by $g\cdot r:=\det_{R}(\rho (g))r$.
  \end{compactenum}
\end{theorem}
\noindent Here we use the convention that $\tilde{H}_{-1}(\emptyset;R)=R$.

If we would like to drop the condition (C) on the arrangement and still have the same description of the $R[G]$-module structure of the cohomology of the complement  we need to strengthen the condition on the coefficients. 
\begin{corollary}
\label{Cor:EqGM}
  Let $\rho\colon G\to \mathrm{O}(d)$ be an orthogonal action of a finite group
  $G$ on the Euclidean space $E\cong\RR^d$.  Consider a $G$-invariant
  arrangement of linear subspaces $\ArA =\{V_1,\ldots,V_k\}$ in $E$.
  Assume that the coefficient ring $R$ is a field of characteristic prime to the order of $G$.
  Then:
  \begin{compactenum}[\rm (i)]
  \item \label{Th:EqGM-Formula-1cor} For the homology of the link of the arrangement $\ArA$ there is an isomorphism of $R[G]$-modules
    \begin{equation*}
      H_{i}(D_{\ArA};R)\cong\bigoplus_{r+s=i}\; \bigoplus_{V\in L_{\ArA}^{>\hat{0}}/G} 
      \ind^G_{G_V}\tilde{H}_{r-1}(\Delta(\hat{0},V);R)\otimes_R \tilde{H}_s(S(V);R) ,
    \end{equation*}
    where $S(V)$ denotes the unit sphere in $V$.

  \item \label{Th:EqGM-Formula-2cor} For the cohomology of the complement of the arrangement $\ArA$ there is an isomorphism of $R[G]$-modules
    \begin{equation*}
      H^{i}(M_{\ArA};R)\cong\mathcal{R}\otimes\bigoplus_{r+s=d-i-2}\; \bigoplus_{V\in L_{\ArA}^{>\hat{0}}/G} 
      \ind^G_{G_V}\tilde{H}_{r-1}(\Delta(\hat{0},V);R)\otimes_R \tilde{H}_s(S(V);R) ,
    \end{equation*}
    where $\mathcal{R}$ is the $R[G]$-module whose underlying $R$-module is $R$
    and for which $g\in G$ acts by $g\cdot r:=\det_{R}(\rho (g))r$.
  \end{compactenum}
\end{corollary}

%%%%%%%%%%%%%%%%%%%%%%%%%%%%%%%%%%%%%%%%%%%%%%%%%%%%%%%%%%%%%%%%%%%%%%%%%%%%%%%%%%%%%

\subsection{Arrangements}
\label{subsec:Arrangements}

Let $d>0$ be a fixed integer and $E=\RR^d$ the Euclidean space.  An
\textit{arrangement} in $E$ is any finite collection $\ArA$ of linear subspaces
of $E$.  To any arrangement $\ArA =\{V_1,\ldots,V_k\}$ we associate:

\begin{compactitem} %[$\bullet$]

\item the \textit{union} of the arrangement $\ArA$ to be the topological space
  $U_{\ArA}:=V_1\cup\cdots\cup V_k$;
\item the \textit{complement} of the arrangement $\ArA$ to be the topological
  space $M_{\ArA}:=E{\setminus} (V_1\cup\cdots\cup V_k)=E{\setminus} U_{\ArA}$;
\item the \textit{link} of arrangement $\ArA$ to be the topological space
  $D_{\ArA}:=S(E)\cap U_{\ArA}$, where $S(E)\approx S^{d-1}$ denotes the unit
  sphere in $E$ with the center in the origin;
\item the \textit{intersection lattice} of the arrangement $\ArA$ to be the
  partially ordered set $L_{\ArA}$ of all intersections of elements of the
  arrangement $\ArA$ partially ordered by reversed (!) inclusion and augmented
  with the space $E$ as the minimum of $L_{\ArA}$, i.e., $\hat{0}:=E$ and $\hat{1}=V_1\cap\cdots\cap V_k$.

\end{compactitem}

The formula of Goresky and MacPherson~\cite[Theorem~III.1.3.A]{G-M},~\cite[Theorem~2.2]{ZZ} describes the (co)homology of the complement of the arrangement in terms of combinatorial data, namely the homology of lower intervals of the intersection lattice:
\begin{equation}
  \label{eq:GMFormula}
  \tilde{H}^i(M_{\ArA};R)\cong\bigoplus_{V\in L_{\ArA}^{>\hat{0}}} \tilde{H}_{\codim (V)-i-2}(\Delta(\hat{0},V);R)
\end{equation}
where the coefficients are taken in a commutative ring with unit $R$.  Here
$\Delta(\hat{0},V)$ stands for the order complex of the open interval
$(\hat{0},V)$ of the lattice $L_{\ArA}$.  The Goresky--MacPherson isomorphism
\eqref{eq:GMFormula} factors in the following way:
\begin{eqnarray}
  \tilde{H}^i(M_{\ArA};R) & \cong &  \tilde{H}^i(S(E){\setminus} D_{\ArA};R)         
  \label{fact_(1)}
  \\
  & \cong &  \tilde{H}_{\dim(E)-2-i}(D_{\ArA};R)                               
  \label{fact_(2)}
  \\
  & \cong &  \bigoplus_{V\in L_{\ArA}^{>\hat{0}}} \tilde{H}_{\codim (V)-i-2}(\Delta(\hat{0},V);R).
  \label{fact_(3)}
\end{eqnarray}
The first isomorphism is a consequence of the radial deformation retraction of
$M_{\ArA}$ onto $S(E){\setminus} D_{\ArA}$, while the second one is the
Alexander duality isomorphism.  The final isomorphism was obtained in the work
of Goresky and MacPherson~\cite{G-M} as an application of stratified Morse
theory.  
On the other hand, in \cite[Theorem~2.4]{ZZ} the homotopy type of the link $D_{\ArA}$ of
the arrangement was determined:
\begin{equation}
  \label{eq:Homotopy-Link}
  D_{\ArA}\simeq \bigvee_{V\in L_{\ArA}^{>\hat{0}}} \Delta(\hat{0},V) * S^{\dim (V)-1}.
\end{equation}
This also implies the last isomorphism \eqref{fact_(3)} of the Goresky--MacPherson isomorphism
factorization.

%%%%%%%%%%%%%%%%%%%%%%%%%%%%%%%%%%%%%%%%%%%%%%%%%%%%%%%%%%%%%%%%%%%%%%%%%%%%%%%%%%%%%

\subsection{Equivariant arrangements}
\label{subsec:equivariant_Arrangements}

Now consider an orthogonal action of a finite group $G$ on the Euclidean space
$E$ via a fixed homomorphism $\rho :G\to\mathrm{O}(d)$.  An arrangement $\ArA
=\{V_1,\ldots,V_k\}$ is \textit{$G$-invariant} if for every $V\in\ArA$ and every
$g\in G$ we have $g\cdot V\in\ArA$. Thus, the union of the arrangement
$U_{\ArA}$ and its complement $M_{\ArA}$ are $G$-invariant subspaces of $E$.
Moreover, since the action of $G$ is orthogonal, the link of the arrangement
$D_{\ArA}$ is also a $G$-invariant subspace of the $G$-invariant sphere $S(E)$.
The action of the group $G$ on the arrangement $\ArA$ also induces an action on
the cohomology of its complement. We want to describe the $R[G]$-module
structure of the cohomology ring $H^*(M_{\ArA};R)$.  To isolate the main
difficulty let us analyze the factorization of the Goresky--MacPherson
isomorphism for a $G$-invariant arrangement $\ArA$.

The action of the group $G$ on $E$ is orthogonal and so the first isomorphism is
induced via a $G$-equivariant radial deformation retraction $M_{\ArA}\to
S(E){\setminus}D_{\ArA}$.  Consequently, the first isomorphism~\eqref{fact_(1)} is an
isomorphism of $R[G]$-modules.

The Alexander duality map is a $G$-equivariant map up to an ``orientation character.''  The
duality map is given by $\alpha\mapsto \alpha\cap \mathcal{O}$ where
$\mathcal{O}$ is the fundamental class of the sphere $S(E)$ and ``$\cap$'' is
the usual ``cap''-pairing relating homology with cohomology.  Therefore, in
order to transform the second isomorphism into an isomorphism of $R[G]$-modules
we have to take into account the associated orientation character.  This means
that we need to know how the fundamental class $\mathcal{O}$ is transformed by
the action of the group $G$.  Since the action of $G$ on $E$ is orthogonal we
have that $g\cdot\mathcal{O}=\det_{R}(\rho(g))\mathcal{O}$, for each $g$ in $G$.
Here $\det_{R}$ is evaluated in the ring $R$.  In many cases of interest for
this paper the orientation character will be trivial and so the second
isomorphism~\eqref{fact_(2)} will also be an isomorphism of $R[G]$-modules.

It remains to deal with an equivariant version of the third
isomorphism~\eqref{fact_(3)}.  For that, following the setup of Sundaram and
Welker~\cite[Section~2]{SunWel}, we adapt the diagram approach presented
in~\cite{ZZ}.

%%%%%%%%%%%%%%%%%%%%%%%%%%%%%%%%%%%%%%%%%%%%%%%%%%%%%%%%%%%%%%%%%%%%%%%%%%%%%%%%%%%%%

\subsection{The diagram approach}
\label{subsec:diagrams_approach}

Let us fix an orthogonal action of the group $G$ on the Euclidean space
$E=\RR^d$, $d>0$, and a $G$-invariant arrangement $\ArA =\{V_1,\ldots,V_k\}$.

Any partially ordered set $(P,\leq)$ can be considered as a small category with
the objects coinciding with the elements of the poset, $\mathrm{Ob}_{P}:=P$, and
a unique morphism $p\to q$ whenever $q\leq p$ in $P$.  We abuse notation by
making no distinction between a poset and its induced small category.  Moreover,
for every subposet $Q$ of $P$, there is a natural inclusion poset map, or
functor, $\mathfrak{i}^P_Q:Q\to P$.

The arrangement $\ArA$ induces a covariant functor, or a \textit{diagram of
  spaces}, $\DiA:L_{\ArA}^{>\hat{0}}\to\mathrm{Top}$ in the following way:
\begin{compactitem}[$\circ$]
\item $\DiA_{V}:=S(E)\cap V\approx S^{\dim (V)-1}$, for every $V\in
  L_{\ArA}^{>\hat{0}}$, and
\item $\DiA_{V\subseteq W}:(\DiA_{V}=S(E)\cap V)\To (\DiA_{W}=S(E)\cap W)$ is
  the inclusion map, for every relation $V\subseteq W$ in $L_{\ArA}^{>\hat{0}}$.
\end{compactitem}
Here $L_{\ArA}^{>\hat{0}}$ denotes the small category induced by the
intersection lattice, and $\mathrm{Top}$ the category of topological spaces.
For a detailed account of the notions introduced and some applications
consult~\cite{WZZ} or~\cite{ZZ}.

The orthogonal action of the group $G$ induces an additional structure on the
diagram $\DiA$.  The intersection lattice $L_{\ArA}$ becomes a $G$-set.  Indeed,
for any $g\in G$:
\begin{eqnarray*}
  W=V_{i_1}\cap\cdots\cap V_{i_r}\in L_{\ArA} &\Impl & g\cdot W=(g\cdot V_{i_1})\cap\cdots\cap (g\cdot V_{i_r})\in L_{\ArA},\\
  V\subseteq W\text{~in~} L_{\ArA}            &\Impl  & g\cdot V\subseteq g\cdot W\text{~in~} L_{\ArA}.
\end{eqnarray*}
Moreover, for every relation $V\subseteq W$ in $L_{\ArA}^{>\hat{0}}$ and every
group element $g\in G$ the following diagram commutes:
\[
\xymatrix@!C= 4em{\DiA_{V} \ar[r]^{g \cdot} \ar[d]_{\DiA_{V\subseteq W}} &
  \DiA_{g\cdot V} \ar[d]^{\DiA_{g\cdot V\subseteq g\cdot W}}
  \\
  \DiA_{W} \ar[r]_{g \cdot} & \DiA_{g\cdot W} }
\]

We introduce two topological spaces associated to the diagram of an arrangement.
The following definitions can be directly generalized for any diagram of spaces
over a small category $\mathcal{C}$, that is, a covariant functor
$\mathcal{C}\to \mathrm{Top}$.
\begin{compactitem} %[$\circ$]

\item The \textit{colimit} of a diagram
  $\DiA:L_{\ArA}^{>\hat{0}}\to\mathrm{Top}$ is defined to be the quotient space
  \begin{equation*}
    \colim_{L_{\ArA}^{>\hat{0}}}\DiA:=\coprod_{V\in L_{\ArA}^{>\hat{0}}}\DiA_{V}/\thicksim ,
  \end{equation*}
  where ``$\thicksim$'' is generated by all relations of the form
  \[
  x\sim y \Leftrightarrow (\exists V\subseteq W \; \text{in} \;
  L_{\ArA}^{>\hat{0}})~x\in V,\; y\in W,\; \DiA_{V\subseteq W}(x)=y.
  \]
  In the case when all the maps in the diagram are
  inclusions, as it is in the case of the diagram associated to an arrangement, the
  colimit of the diagram coincides with the union.
  Thus the colimit of the diagram $\DiA:L_{\ArA}^{>\hat{0}}\to\mathrm{Top}$ is the link of the arrangement $\ArA$, i.e.,
  $\colim_{L_{\ArA}^{>\hat{0}}}\DiA=D_{\ArA}$.

\item The \textit{homotopy colimit} of the diagram
  $\DiA:L_{\ArA}^{>\hat{0}}\to\mathrm{Top}$ is defined as:
  \begin{equation}
    \label{eq:DefHomcolim}
    \hcolim_{L_{\ArA}^{>\hat{0}}}\DiA:=\coprod_{V\in L_{\ArA}^{>\hat{0}}}\Delta(\hat{0},V]\times\DiA_{V}/\backsim
  \end{equation}
  where the equivalence relation ``$\backsim$'' is defined as follows.  For
  simplicity let use denote by
  \[
  X:=\coprod_{V\in L_{\ArA}^{>\hat{0}}}\Delta(\hat{0},V]\times\DiA_{V} \quad
  \text{and} \quad Y:=\coprod_{V\subseteq W\text{~in~}
    L_{\ArA}^{>\hat{0}}}\Delta(\hat{0},W]\times\DiA_{V}.
  \]
  Consider the maps $\alpha :Y\to X$ and $\beta :Y\to X$ given by the component
  maps:
  \[\begin{array}{lcll}
    \alpha_{W\supseteq V} \colon \Delta (\hat{0},W]\times \DiA_{V} & \longrightarrow & \Delta (\hat{0},W]\times \DiA_{W}, 
    & \quad
    (p,a)\longmapsto (p,\DiA_{V\subseteq W}(a));
    \\
    \beta_{W\supseteq V} \colon\Delta (\hat{0},W]\times \DiA_{V} & \longrightarrow &\Delta (\hat{0},V]\times \DiA_{V}, 
    &\quad
    (p,a)\longmapsto (\Delta\big(\mathfrak{i}^{(\hat{0},V]}_{(\hat{0},W]}\big)(p),a).
  \end{array}
  \]
  for every relation $V\subseteq W$ in $L_{\ArA}^{>\hat{0}}$.  Here $\Delta\big(\mathfrak{i}^{(\hat{0},V]}_{(\hat{0},W]}\big)$ denotes the map
  between order complexes induced by the inclusion $\mathfrak{i}^{(\hat{0},V]}_{(\hat{0},W]}$ 
  of posets.  Now the equivalence relation ``$\backsim$'' is given by
  $\alpha(p,x)\backsim\beta(p,x)$, for every $(p,x)\in Y$.
\end{compactitem}
The projection to the second factor, in the definition of homotopy colimit
\eqref{eq:DefHomcolim}, induces a natural map
\begin{equation*} 
  \hcolim_{L_{\ArA}^{>\hat{0}}}\DiA \To \colim_{L_{\ArA}^{>\hat{0}}}\DiA
\end{equation*}
called the \textit{projection map}.  The central property of this map is that under certain conditions for general diagrams over small categories it induces a
homotopy equivalence~\cite[Proposition~3.1]{WZZ}
\[
\hcolim_{L_{\ArA}^{>\hat{0}}}\DiA \simeq \colim_{L_{\ArA}^{>\hat{0}}}\DiA.
\]
Thus, instead of studying the link of an arrangement, one can then consider the homotopy colimit of the diagram induced by the arrangement.

\smallskip

The group $G$ induces additional structure on the intersection lattice of the arrangement $\ArA$ and on the diagram $\DiA$.  
Consequently, the action of the group $G$ can be defined on
both $\colim_{L_{\ArA}^{>\hat{0}}}\DiA$ and $\hcolim_{L_{\ArA}^{>\hat{0}}}\DiA$
such that:
\begin{compactitem} %[$\circ$]
\item the $G$-action on $\colim_{L_{\ArA}^{>\hat{0}}}\DiA$ coincides with the
  $G$-action induced on the link $D_{\ArA}$, and
\item the projection map $\hcolim_{L_{\ArA}^{>\hat{0}}}\DiA \To
  \colim_{L_{\ArA}^{>\hat{0}}}\DiA$ is a $G$-homotopy
  equivalence~\cite[Lemma~2.1]{SunWel}.
\end{compactitem}
Therefore, in order to understand the $G$-module structure on the homology of
the link $D_{\ArA}$ we will study the $G$-module structure on the homology of
the homotopy colimit $\hcolim_{L_{\ArA}^{>\hat{0}}}\DiA$ of the arrangement
$\ArA$.  This will be done using the spectral sequence converging to the
homology of the homotopy colimit of a diagram introduced by
Segal~\cite[Proposition~5.1]{Segal}, see also~\cite[Theorem~4.7]{Davis-Lueck(1998)},~\cite[Theorem~3.5]{ZZ}.  In the
present situation, the additional structure imposed by the group action can be
retrieved by a careful study of the spectral sequence convergence.

%%%%%%%%%%%%%%%%%%%%%%%%%%%%%%%%%%%%%%%%%%%%%%%%%%%%%%%%%%%%%%%%%%%%%%%%%%%%%%%%%%%%%

\subsection{A spectral sequence argument}
\label{subsec:a_spectral_sequence_argument}

Let $R$ be a commutative ring with unit.  For a simplicial complex $K$ and
$r\geq 0$ an integer, let $K^{(r)}$ stand for the $r$-skeleton subcomplex of
$K$.

Let us consider the family of $G$-invariant subspaces $\{X_r:r\geq 0\}$ of the
homotopy colimit of the diagram $\DiA$ defined by
\begin{equation*}
  X_r:=\coprod_{V\in L_{\ArA}^{>\hat{0}}}\Delta(\hat{0},V]^{(r)}\times\DiA_{V}/\backsim_{r}
\end{equation*}
where the equivalence relation ``$\backsim_{r}$'' is the restriction of the
relation ``$\backsim$.''  The following sequence of inclusions of $G$-invariant
subspaces
\begin{equation}
  \label{eq:Fil}
  X_0\subseteq X_1 \subseteq \cdots \subseteq X_d=\hcolim_{L_{\ArA}^{>\hat{0}}}\DiA
\end{equation}
defines a $G$-invariant filtration of the homotopy colimit.

The homology spectral sequence  associated to the filtration~\eqref{eq:Fil}, has the $E^1$-term  
\[
E^1_{r,s}=H_{r+s}(X_r,X_{r-1};R),
\]
and for the differential
\[
\partial^1\colon  E^1_{r,s}=H_{r+s}(X_r,X_{r-1};R)\longrightarrow
E^1_{r-1,s}=H_{r+s-1}(X_{r-1},X_{r-2};R)
\]
the boundary map of the long exact sequence of the triple
$(X_r,X_{r-1},X_{r-2})$.

The $G$-action on the filtration~\eqref{eq:Fil} implies the following
$R[G]$-module decomposition of the $E^1$-term:
\[
E^1_{r,s}=H_{r+s}(X_r,X_{r-1};R)\cong\bigoplus_{(V_0<\cdots <V_r)\in (\Delta
  (L_{\ArA}^{>\hat{0}}))^{(r)}/G} \ind^G_{G_{V_r}}\tilde{H}_s(\DiA_{V_r};R)
\]
where $G_{V_r}$ denotes the subgroup of $G$ that stabilizes the element $V_r\in
L_{\ArA}^{>\hat{0}}$.  Consider an $(r+s)$-chain $(V_0<\cdots <V_r)\times c$ of
the space $(V_0<\cdots <V_r)\times \DiA_{V_r}$ where $c$ is a cycle.  The
evaluation of the differential $\partial^1$ on this chain is given by
\begin{eqnarray}
  \label{eq:D1-1}
  \partial^1((V_0<\cdots <V_r)\times c) 
  & = & \sum_{i=1}^{r-1}(-1)^i(V_0<\cdots <\widehat{V_i} <\cdots<V_r)\times c + \\
  &   & (-1)^r (V_0<\cdots <V_{r-1})\times (\DiA_{V_r\subseteq V_{r-1}})_{\#}(c).
  \nonumber
\end{eqnarray}
Here $(\DiA_{V_r\subseteq V_{r-1}})_{\#}$ denotes the map on the chain level
induced by the inclusion map $\DiA_{V_r\subseteq V_{r-1}}\colon
\DiA_{V_r}\to\DiA_{V_{r-1}}$.  Since all the maps $\DiA_{V\subseteq W}$ are
positive codimension inclusions of spheres, all the maps in homology
\[
(\DiA_{V_r\subseteq V_{r-1}})_{*}\colon
\tilde{H}_s(\DiA_{V_r};R)\longrightarrow\tilde{H}_s(\DiA_{V_{r-1}};R)
\]
vanish.  Therefore, the differential $\partial^1$ is determined by the
expression in line~\eqref{eq:D1-1}, i.e.,\ without losing generality we can
assume that
\begin{eqnarray*}
  \partial^1((V_0<\cdots <V_r)\times c) & = & \sum_{i=1}^{r-1}(-1)^i(V_0<\cdots <\hat{V_i} <\cdots<V_r)\times c.
\end{eqnarray*}
Recalling the notion of the Whitney homology of a poset~\cite[Section~5,
pages~120-122]{Bjo}~\cite[Section~1, pages~227-229]{Sun}, we conclude that the
differential $\partial^1$ is the boundary operator of the Whitney homology of
the intersection lattice with coefficients in $R[G_V]$-modules
$H_*(\DiA_{V};R)$.  Moreover, using~\cite[Theorem~5.1, page~121]{Bjo}
or~\cite[Theorem~1.2, page~229]{Sun}, we have the following description of the
$E^2$-term:
\begin{equation*}
  E^2_{r,s}\cong \bigoplus_{V\in \Delta (L_{\ArA}^{>\hat{0}})/G} \ind^G_{G_V}\tilde{H}_{r-1}(\Delta(\hat{0},V);\tilde{H}_s(\DiA_{V_r};R)).
\end{equation*}
In order to avoid confusion when comparing references, we note that
Bj\"orner's definition of the Whitney homology in~\cite{Bjo} has dimension shift
$+1$ with respect to the Whitney homology as defined by Sundaram in~\cite{Sun}.

Recall that we assume the condition ${\bf (R)}$ that the coefficient ring $R$ is
a principal ideal domain and additionally for every $V\in L_{\ArA}^{>\hat{0}}$
the homology groups $H_{*}(\Delta(\hat{0},V);R)$ are free $R$-modules.  The
universal coefficient theorem implies that there is an isomorphism of
$R[G]$-modules
\begin{equation*}
  E^2_{r,s}\cong \bigoplus_{V\in L_{\ArA}^{>\hat{0}}/G} \ind^G_{G_V}\tilde{H}_{r-1}(\Delta(\hat{0},V);R)\otimes_R \tilde{H}_s(\DiA_V;R).
\end{equation*}
The spectral sequence we consider converges to the homology of the link
$H_{*}(D_{\ArA};R)$.  On the other hand, the homotopy type of the link is given
by~\eqref{eq:Homotopy-Link}, and consequently its homology as an $R$-module is
known.  The condition ${\bf (R)}$ implies the equality
\[
\sum_{r+s=n}\mathrm{rank}_R E^2_{r,s}=\mathrm{rank}_R H_{n}(D_{\ArA};R),
\]
where $\mathrm{rank}_R$ denotes the rank of a module over the principal ideal
domain $R$.  Therefore the spectral sequence collapses at the $E^2$-term, i.e.,
$E^2_{*,*}=E^{\infty}_{*,*}$. So we conclude

\begin{lemma} \label{lem:equivariant_filtration} There is an ascending
  filtration of $R[G]$-modules
  \[
  F_{-1,n+1} = \{0\} \subseteq F_{0,n} \subseteq F_{1,n-1} \subseteq \cdots
  \subseteq F_{n-1,1} \subseteq F_{n,0} = H_{n}(D_{\ArA};R),
  \]
  such that we have isomorphisms of $R[G]$-modules
  \[
  F_{r,s}/F_{r-1,s+1} \cong E^{\infty}_{r,s} = E^{2}_{r,s} \cong \bigoplus_{V\in
    L_{\ArA}^{>\hat{0}}/G}
  \ind^G_{G_V}\tilde{H}_{r-1}(\Delta(\hat{0},V);R)\otimes_R
  \tilde{H}_s(\DiA_V;R).
  \]
\end{lemma}

%%%%%%%%%%%%%%%%%%%%%%%%%%%%%%%%%%%%%%%%%%%%%%%%%%%%%%%%%%%%%%%%%%%%%%%%%%%%%%%%%%%%%

\subsection{$c$-arrangements}
\label{subsec:c-arrangements}

If we forget the $G$-action and use the condition ${\bf (R)}$, we obtain an
isomorphism of $R$-modules
\begin{equation}
  \label{eq:Isomor-Vsp-Ab}
  H_{n}(D_{\ArA};R)\cong\bigoplus_{r+s=n}\bigoplus_{V\in L_{\ArA}^{>\hat{0}}/G} 
  \ind^G_{G_V}\tilde{H}_{r-1}(\Delta(\hat{0},V);R)\otimes_R \tilde{H}_s(\DiA_V;R).
\end{equation}
An important observation is that this isomorphism does not automatically become an isomorphism of $R[G]$-modules. 
This is the case if $R$ is a field of characteristic prime to the order of $G$, and we conclude the proof of Corollary~\ref{Cor:EqGM}.
  
We are also interested in the modular case, i.e., the characteristic of $R$ will divide the order of $G$.  
In order to ensure that~\eqref{eq:Isomor-Vsp-Ab} is an $R[G]$-isomorphism, we make an
extra assumption on the arrangement, namely the condition ${\bf (C)}$ that the
arrangement is a $c$-arrangement for a some $c >1$ in the following sense.

\begin{definition}[$c$-arrangement]\label{def:c-arrangement}
  The arrangement $\ArA$ is called a \emph{$c$-arrangement} for some integer $c>1$,
  if $\codim_{\RR}V_i=c$, for all $i\in\{1,\ldots,k\}$, and for every pair of
  elements $V\subseteq W$ in the intersection lattice
  $c~|~\codim_{\RR}(V\subseteq W)$.
\end{definition}

In~\cite[Proposition in Section~III.4.1]{G-M} it was proved that the intersection
lattice $L_{\ArA}$ of every $c$-arrangement $\ArA$ is also a geometric lattice.
In particular, this means that for each open interval $(V,W)$ in $L_{\ArA}$
\[
\tilde{H}_i(\Delta (V,W);R)=0 \quad \text{for all} \quad i\neq \dim \Delta (V,W) \quad
\text{and any coefficients} \; R.
\]
If this additional assumption holds, as the filtration~\eqref{eq:Fil} grows from
$X_{m-1}$ to $X_m$, the contribution to the homology of the homotopy colimit
comes from all $V\in L_{\ArA}$ such that $\codim V=c(m+1)$ and it appears at the
position $(m, \dim V-1)=(\tfrac{1}{c}\codim V-1,\dim V-1)$ in dimension
\[
\tfrac{1}{c}\codim V+\dim V-2=d-2-(1-\tfrac{1}{c})\codim
V=\tfrac{d}{c}-2+(1-\tfrac{1}{c})\dim V.
\]
Since $c>1$, the contribution always comes in dimensions where there was no previous contribution from smaller elements of the filtration.  
Therefore, under the assumptions {\bf (R)} and {\bf (C)}, we get for every $n$ that in the filtration appearing in
Lemma~\ref{lem:equivariant_filtration} there is at most one index $r$ for which $F_{r,n-r} \not= F_{r-1,n+1-r}$. 
Hence the isomorphism~\eqref{eq:Isomor-Vsp-Ab} is an $R[G]$-isomorphism by
Lemma~\ref{lem:equivariant_filtration}. This finishes the proof of
assertion~\eqref{Th:EqGM-Formula-1} of Theorem~\ref{Th:EqGM}.

As we have seen, the Alexander duality isomorphism is an equivariant map up to
the orientation character.  Therefore assertion~\eqref{Th:EqGM-Formula-2} of
Theorem~\ref{Th:EqGM} follows from assertion~\eqref{Th:EqGM-Formula-1}. This
finishes the proof of Theorem~\ref{Th:EqGM}.

%%%%%%%%%%%%%%%%%%%%%%%%%%%%%%%%%%%%%%%%%%%%%%%%%%%%%%%%%%%%%%%%%%%%%%%%%%%%%%%%%%%%%
%%%%%%%%%%%%%%%%%%%%%%%%%%%%%%%%%%%%%%%%%%%%%%%%%%%%%%%%%%%%%%%%%%%%%%%%%%%%%%%%%%%%%
% -----------------------------------------------------------------------------------%
\section{Cohomology of the configuration space as an $R[\Sym_n]$-module}
\label{Sec:CohomologyOfCS-as-module}
% -----------------------------------------------------------------------------------%
%%%%%%%%%%%%%%%%%%%%%%%%%%%%%%%%%%%%%%%%%%%%%%%%%%%%%%%%%%%%%%%%%%%%%%%%%%%%%%%%%%%%%
%%%%%%%%%%%%%%%%%%%%%%%%%%%%%%%%%%%%%%%%%%%%%%%%%%%%%%%%%%%%%%%%%%%%%%%%%%%%%%%%%%%%%

In this section we consider the configuration space $F(\RR^d,n)$ as the complement of an $\Sym_n$-invariant arrangement.  
Here $\Sym_n$ denotes the group of permutations on $n$ letters.  
Using the Equivariant Goresky--MacPherson formula from Theorem~\ref{Th:EqGM}~\eqref{Th:EqGM-Formula-2} we describe the $R[\Sym_n]$-module structure on the cohomology of the configuration space $F(\RR^d,n)$ with coefficients in an appropriate ring $R$.

Different descriptions of the $\Sym_n$-equivariant structure on cohomology of the configuration space $F(\RR^d,n)$ can be found, in the landmark paper \cite[Section 7]{Cohen} of Fred Cohen, book by Fadell \& Husseini \cite[Part II, Chapter V]{FadellHusseini:book}, the paper of F.~Cohen \& Taylor \cite{Cohen-Taylor}
and in the paper by Arone \cite[Proposition 2.1 and Lemma 2.2]{Arone2006}.

%%%%%%%%%%%%%%%%%%%%%%%%%%%%%%%%%%%%%%%%%%%%%%%%%%%%%%%%%%%%%%%%%%%%%%%%%%%%%%%%%%%%%

\subsection{Configuration spaces}
\label{subsec:Configuration_spaces}

For a topological space $X$, the configuration space of $n$ distinct points is
defined to be
\[
F(X,n):=\{ (x_1,\ldots,x_n)\in X^n : x_i\neq x_j\text{ for all } i < j\} \subset
X^n.
\]
The symmetric group $\Sym_n$ acts on $X^n$ and consequently on $F(X,n)$ by
permuting the factors in the product $X^n$.  For $i < j$, let us denote:
\[
L_{i,j}(X):=\{ (x_1,\ldots,x_n)\in X^n : x_i=x_j\}
\]
and
\[
\mathcal{B}_n(X) :=\{ \sigma\cdot L_{1,2}(X) : \sigma\in\Sym_n\}=\{ L_{i,j}(X) : 1\leq i<j\leq n\}.
\]
Then the configuration space can be viewed as the complement of the subspace arrangement $\mathcal{B}_n$ 
\[
F(X,n)=X^n {\setminus}\bigcup\mathcal{B}_n(X)=X^n
{\setminus}\bigcup_{\sigma\in\Sym_n}\sigma\cdot L_{1,2}(X)=X^n{\setminus}\bigcup_{1\leq
  i<j\leq n}L_{i,j}(X).
\]
When $X$ is the Euclidean space $\RR^d$, the configuration space $F(\RR^d,n)$ is a
complement of the linear subspace arrangement
$\mathcal{B}_{n,d}:=\mathcal{B}_n(\RR^d)$.

%%%%%%%%%%%%%%%%%%%%%%%%%%%%%%%%%%%%%%%%%%%%%%%%%%%%%%%%%%%%%%%%%%%%%%%%%%%%%%%%%%%%%

\subsection{Partitions}
\label{subsec:partitions}

Let $\Pi_n$ denote the lattice of all partitions of the set
$[n]:=\{1,\ldots,n\}$ ordered by refinement (induced by inclusion of the
blocks).  The minimum of $\Pi_n$ is the partition
$\hat{0}:=\{\{1\},\{2\},\ldots,\{n\}\}$ given by all singletons and the maximum
is $\hat{1}:=\{[n]\}$.  Thus a typical element of $\Pi_n$ is a
partition $\{P_1,\ldots,P_k\}$ of the set $[n]$ where $P_1\cup\cdots\cup
P_k=[n]$ and $P_i\cap P_j=\emptyset$ for all $i < j$.  We assume that
there are no empty sets in the presentation of a partition.  The number of
(non-empty) blocks of the partition $\pi$ is also called its \textit{size},
denoted by $\mathrm{size}(\pi)$.

The poset $\Pi_n{\setminus}\{\hat{0}\}$ is isomorphic to the intersection poset
of the arrangement $\mathcal{B}_{n,d}$ for any $d\geq 1$.  The correspondence
between the elements of these posets is given by
\begin{eqnarray*}
\{P_1,\ldots,P_k\}                       
  \longleftrightarrow                         
  V_{\{P_1,\ldots,P_k\},d}  
& := & 
\left\{(x_1,\ldots,x_n)\in (\RR^d)^n:  x_i=x_j \text{ if }i,j\in P_r\text{ for some }1\leq r\leq k \right\} 
\\
& = & 
\bigcap_{\substack{r\in\{1,\ldots,k\} \\ i, j\in P_r; \;  i < j}} \; L_{i,j}(\RR^d).     
\end{eqnarray*}
In what follows, we do not distinguish between the partition
$\{P_1,\ldots,P_k\}$ and its associated linear subspace
$V_{\{P_1,\ldots,P_k\},d}$.

\smallskip

When we speak about topological properties of a lattice $L$, we always have the
order complex of its proper part $\bar{L}:=L{\setminus}\{\hat{0},\hat{1}\}$ in
mind,~\cite[Section~4, page~287]{BjoWel}.  The homotopy type of the partition
lattice $\Pi_n$ is known:
\begin{equation}
  \label{eq:HomotopyTypeofPi_n}
  \Delta(\Pi_n{\setminus}\{\hat{0},\hat{1}\})=\Delta(\bar{\Pi}_n)\simeq \bigvee_{(n-1)!} S^{n-3}.
\end{equation}
Consult for example~\cite[Theorem~1.5(b), page~279 or Proposition~4.1, page~288]{BjoWel}.

When necessary, we use notation $\Pi_X$ to denote the partition lattice of the
finite set $X$.  The set $X$ is also called the \textit{ground set} for the
partition lattice $\Pi_X$.  In particular $\Pi_n=\Pi_{[n]}$ and
$\Pi_X\cong_{poset}\Pi_{|X|}$.

%%%%%%%%%%%%%%%%%%%%%%%%%%%%%%%%%%%%%%%%%%%%%%%%%%%%%%%%%%%%%%%%%%%%%%%%%%%%%%%%%%%%%

\subsection{Lower intervals}
\label{subsec:lower_intervals}

Now we describe the topology of the lower intervals of the partition lattice $\Pi_n$.  
Consider a fixed partition $\pi:=\{P_1,\ldots,P_j\}$ of size $j$ of
$[n]$ and denote by:
\begin{compactitem}
\item $a_i(\pi)$ the size of the block $P_i$, $i\in\{1,\ldots,j\}$;
\item $b_i(\pi)$ the number of blocks of size $i$, $i\in\{1,\ldots,n\}$.
\end{compactitem}
Following Stanley~\cite[page~317]{Stanley} we obtain that
\[ [\hat{0},\pi]=[\hat{0},V_{\pi,d}]\cong
\Pi_{P_1}\times\cdots\times\Pi_{P_j}\cong\Pi_{a_1(\pi)}\times\cdots\times\Pi_{a_j(\pi)}\cong\Pi_{1}^{b_1(\pi)}\times\cdots\times\Pi_{n}^{b_n(\pi)}.
\]
Here for a poset $P$ its $0$-power $P^0$ is the poset with only one element,
i.e., minimum and maximum of $P$ coincide.  
Furthermore, from Walker~\cite[Theorem~6.1(d)]{JW}, with $\bar{\Pi}_1=\bar{\Pi}_2=\emptyset$ and ``$K*\emptyset =K$'', it follows that
\begin{eqnarray}
  \label{eq:HomotopyTypeofLowerConePi_n}
  \Delta(\hat{0},V_{\pi})                       & \simeq &
  \Sigma^{k-1}\left(\Delta(\bar{\Pi}_{P_1})*\cdots *\Delta(\bar{\Pi}_{P_j})\right)\\
  & \approx      &
  \Sigma^{k-1}\left(\Delta(\bar{\Pi}_{a_1(\pi)})*\cdots *\Delta(\bar{\Pi}_{a_j(\pi)})\right)\nonumber\\
  & \approx      &
  \Sigma^{k-1}\left(\Delta(\bar{\Pi}_{1})^{*b_1(\pi)}*\cdots *\Delta(\bar{\Pi}_{n})^{*b_n(\pi)}\right).\nonumber
\end{eqnarray}
Here $\Sigma$ denotes the suspension.

The stabilizing subgroup of the partition $\pi$ is the subgroup
\[
(\Sym_n)_{\pi}:=(\Sym_1\wr\Sym_{b_1(\pi)})\times\cdots\times
(\Sym_n\wr\Sym_{b_n(\pi)})
\]
of the symmetric group $\Sym_n$. 
For $b_r=0$, the group $\Sym_{b_r}\wr\Sym_{b_r(\pi)}$ is trivial. 
Only the maximal and minimal element, the partitions $\hat{1}=\{[n]\}$ and $\hat{0}=\{\{1\},\ldots,\{n\}\}$, are stabilized by the complete symmetric group $\Sym_n$. 
Now $(\Sym_n)_{\pi}$ acts on the lower interval $[\hat{0},V_{\pi,d}]$ and consequently on its homology.
Using Wachs~\cite[Theorem~5.1.5, page~588]{Wachs}, we obtain the
$R[(\Sym_n)_{\pi}]$-module structure on the homology of the lower interval
$[\hat{0},V_{\pi,d}]$:
\[
\tilde{H}_{r}(\Delta(\hat{0},V_{\pi,d});R)\cong
\bigoplus_{i_1+\cdots+i_j+2j-2=r}\tilde{H}_{i_1}(\Delta(\bar{\Pi}_{a_1(\pi)});R)\otimes\cdots\otimes\tilde{H}_{i_j}(\Delta(\bar{\Pi}_{a_j(\pi)});R).
\]
Moreover, this homology is non-trivial if and only if $r=n-j-2$, and
\begin{eqnarray}
  \tilde{H}_{n-j-2}(\Delta(\hat{0},V_{\pi,d});R)              &\cong &
  \tilde{H}_{a_1(\pi)-3}(\Delta(\bar{\Pi}_{P_1});R)\otimes\cdots\otimes\tilde{H}_{a_j(\pi)-3}(\Delta(\bar{\Pi}_{P_j});R)\nonumber\\
  &\cong &
  \tilde{H}_{a_1(\pi)-3}(\Delta(\bar{\Pi}_{a_1(\pi)});R)\otimes\cdots\otimes\tilde{H}_{a_j(\pi)-3}(\Delta(\bar{\Pi}_{a_j(\pi)});R)\nonumber\\
  &\cong &
  \tilde{H}_{1-3}(\Delta(\bar{\Pi}_{1});R)^{\otimes b_1(\pi)}\otimes\cdots\otimes\tilde{H}_{n-3}(\Delta(\bar{\Pi}_{n});R)^{\otimes b_n(\pi)}\nonumber\\
  \label{eq:HomologyLowerCone-as-module}
  &\cong &
  \tilde{H}_{-2}(\Delta(\bar{\Pi}_{1});R)^{\otimes b_1(\pi)}\otimes\cdots\otimes\tilde{H}_{n-3}(\Delta(\bar{\Pi}_{n});R)^{\otimes b_n(\pi)}
\end{eqnarray}
assuming that $\tilde{H}_{-2}(\emptyset ;R)=\tilde{H}_{-1}(\emptyset ;R)=R$.
The group $(\Sym_n)_{\pi}$ acts componentwise, meaning that each factor
$\tilde{H}_{j-3}(\Delta(\bar{\Pi}_{j});R)^{\otimes b_j(\pi)}$ in the tensor
product is an $R[\Sym_j\wr\Sym_{b_j(\pi)}]$-module in a natural way.

%%%%%%%%%%%%%%%%%%%%%%%%%%%%%%%%%%%%%%%%%%%%%%%%%%%%%%%%%%%%%%%%%%%%%%%%%%%%%%%%%%%%%

\subsection{The cohomology of the configuration space as $R[\Sym_n]$-module}
\label{subsec:The_cohomology_of_the_configuration_space}

\begin{theorem}
  \label{Th:ModuleStructureOnCohomology}
  Let $n>1$ and $d>1$ are integers. Let $R$ be a principal ideal domain.  Then
  \begin{center}
    $H^i(F(\RR^d,n);R)\neq 0$ if and only if $i=(d-1)(n-j)$ for some
    $j\in\{1,\ldots,n\}$.
  \end{center}
  For $j\in\{1,\ldots,n-1\}$ there is an isomorphism of $R[\Sym_n]$-modules

\begin{multline*}
  H^{(d-1)(n-j)}(F(\RR^d,n);R) \cong
  \\
  \mathcal{R}\otimes \bigoplus_{\substack{\pi\in (\Pi_n \setminus
      \{\hat{0}\})/\Sym_n\\ \mathrm{size}(\pi)=j}}
  \ind^{\Sym_n}_{(\Sym_n)_{\pi}}\tilde{H}_{-2}(\Delta(\bar{\Pi}_{1});R)^{\otimes
    b_1(\pi)}
  \otimes\cdots\otimes\tilde{H}_{n-3}(\Delta(\bar{\Pi}_{n});R)^{\otimes
    b_n(\pi)}\otimes\mathcal{R}_{V_{\pi,d}}
\end{multline*}
where $\mathcal{R}$ is the $R[\Sym_n]$-module whose underlying $R$-module is $R$
and for which $g\in\Sym_n$ acts by $g\cdot r:=\det_{R}(g)r$, and
$\mathcal{R}_{V_{\pi,d}}$ is the $R[(\Sym_n)_{\pi}]$-module whose underlying
$R$-module is $R$ and for which $g\in(\Sym_n)_{\pi}$ acts by $g\cdot
r:=\det_{R}(g|_{V_{\pi,d}})r$.
\end{theorem}
\begin{proof}
  The homotopy equivalences~\eqref{eq:HomotopyTypeofPi_n}
  and~\eqref{eq:HomotopyTypeofLowerConePi_n} imply that all the assumptions of
  the Equivariant Goresky--MacPherson formula~\ref{Th:EqGM} hold.  Now apply
  Theorem~\ref{Th:EqGM}~\eqref{Th:EqGM-Formula-2} and
  formula~\eqref{eq:HomologyLowerCone-as-module}.
\end{proof}

\begin{remark}
Observe that for every subgroup $G$ of $\Sym_n$ a similar theorem can be stated that describes the cohomology of the configuration space $F(\RR^d,n)$ with the coefficients in the ring $R$ as an $R[G]$-module.
The proof would be identical.
\end{remark}

\begin{remark}
The structure of the cohomology of the configuration space $F(\RR^d,n)$ described in Theorem ~\ref{Th:ModuleStructureOnCohomology}, in the case $R=\ZZ,$ can be directly related to the results of F.~Cohen \& Taylor given in \cite[Section 3, Theorem 3.9]{Cohen-Taylor} via the following dictionary:
{\small
\begin{center}
    \begin{tabular}{ | l | l | l|}
    \hline
    Partition of $[n]$ & $\pi$ & $\mathcal{P}$ \\ \hline
    Number of blocks in the partition& $k$& $\pi_0(\mathcal{P})$\\ \hline
    A direct summand in  $H^*(F(\RR^d,n);\ZZ)$& $\tilde{H}_{n-k-2}(\Delta(\hat{0},V_{\pi,d});\ZZ)\otimes \mathcal{R}\otimes\mathcal{R}_{V_{\pi,d}}$    &$T(\mathcal{P},d)$\\ \hline
     Stabilizing subgroup of the partition & $(\Sym_n)_{\pi}$ & $\Sigma(\mathcal{P})$ \\ \hline
     A direct summand in a module decomposition& $\mathcal{R}\otimes\ind^{\Sym_n}_{(\Sym_n)_{\pi}}\tilde{H}_{n-k-2}(\Delta(\hat{0},V_{\pi,d});\ZZ)\otimes\mathcal{R}_{V_{\pi,d}}$  &
     $T(\mathcal{L}_{\mu},d)|^{\Sigma_n}$\\ \hline
    \end{tabular}
\end{center}
}

\end{remark}

\begin{remark}
Consider the natural inclusion of the partition lattices $i_{\Pi_n}^{\Pi_{n+1}}\colon\Pi_n\to\Pi_{n+1}$ given by
\[
\pi:=(P_1,\ldots,P_j)\longmapsto \pi':=(P_1,\ldots,P_j,\{n+1\})
\]
where $(P_1,\ldots,P_j)$ is a partition of $[n]$. 
Let $1\le j\leq n$ and $\Pi_n^{\{j\}}$ denotes the anti-chain of all partitions of size $j$.
Then $i_{\Pi_n}^{\Pi_{n+1}}$ includes the anti-chain $\Pi_n^{\{j\}}$ into the anti-chain $\Pi_{n+1}^{\{j+1\}}$.

If $\Sym_n\to\Sym_{n+1}$ is a monomorphism of groups induced from the set inclusion $[n]\to[n+1]$, then $i_{\Pi_n}^{\Pi_{n+1}}$ is an $\Sym_n$-equivariant map.
Hence, for every $\pi\in\Pi_n$ the inclusion $i_{\Pi_n}^{\Pi_{n+1}}$ induces an $\Sym_n$-equivariant poset isomorphism between the open intervals $(0,V_{\pi})\to (0,V_{\pi'})$.

Now, Theorem~\ref{Th:ModuleStructureOnCohomology} implies that the inclusion $i_{\Pi_n}^{\Pi_{n+1}}$ induces a monomorphism of $R[\Sym_n]$-modules:
\[
\Phi_{n,\ell}\colon H^{(d-1)\ell}(F(\RR^d,n);R)\rightarrow H^{(d-1)\ell}(F(\RR^d,n+1);R),
\] 
when $d\geq 2$, $n\geq 1$ and $1\leq\ell\leq n-1$,

Observe that for  $2(j+1)>n+1$ each partition in $\Pi_{n+1}^{\{j+1\}}$  must have at least one  part of length $1$.
Therefore, the $\Sym_{n+1}$-orbit of the image $i_{\Pi_n}^{\Pi_{n+1}}(\Pi_n^{\{j\}})$ of the anti-chain $\Pi_n^{\{j\}}$ equals $\Pi_{n+1}^{\{j+1\}}$.
Consequently,  the $\Sym_{n+1}$-orbit of the image $\Phi_{n,\ell}(H^{(d-1)\ell}(F(\RR^d,n);R))$ equals $H^{(d-1)\ell}(F(\RR^d,n+1);R)$ when $l<\tfrac{n+1}{2}$.

These two properties, derived from Theorem~\ref{Th:ModuleStructureOnCohomology}, are manifestations of ``Representation stability'' phenomena of Church and Farb. 
More precisely, these are ``Injectivity'' and ``Surjectivity'' properties in the language of \cite[Definition 1.1]{Church} and \cite[Definition 2.3]{Church-Farb}.
\end{remark}

%%%%%%%%%%%%%%%%%%%%%%%%%%%%%%%%%%%%%%%%%%%%%%%%%%%%%%%%%%%%

\subsection{The special case $G = (\ZZ/p)^k$}
\label{subsec:The_special_case_G_is_(Z/p)k}

Finally, we discuss the special case when $n=p^k$ is a power of a prime $p$.
Moreover, let $G\cong (\ZZ/p)^k$ be a subgroup of $\Sym_n$ given by the regular
embedding $\mathrm{(reg)} \colon G \to \Sym_n$,~\cite[Example~2.7  on page~100]{Adem-Milgram}.
The regular embedding is given by the left translation action of $(\ZZ/p)^k$ on itself. 
To each element $g\in (\ZZ/p)^k$ we associate permutation $L_g\colon  (\ZZ/p)^k\to  (\ZZ/p)^k$ from $\mathrm{Sym}((\ZZ/p)^k)\cong\Sym_{p^k}$ given by $L_g(x)=g+x$.
In this special case we are interested in partitions of
$[n]$ that are stabilized by the whole group $G$.

Let $H$ be a non-zero subgroup of $G$ and $m:=|G/H|$.  Then $H$ acts on $[n]$.
Let $O_1,\ldots,O_m$ be the orbits of the $H$-action on $[n]$ and $\pi_H$ be the
element of $\Pi_n$ corresponding to the partition $\{O_1,\ldots,O_m\}$.  If
$\pi_H$ is considered as an element of the intersection lattice of an
arrangement, it will be denoted by $V_{H,d}$, where $(\RR^d)^n$ is assumed to be the 
ambient space of the corresponding arrangement, i.e.,
\begin{equation}
  \label{eq:V_H}
  V_{H,d}
  :=
  \left\{(x_1,\ldots,x_n)\in (\RR^d)^n : x_i=x_j \; \text{if } i,j\in O_r\text{ for some } 1\leq r\leq m \right\} = ((\RR^d)^n)^H.
\end{equation}
All the blocks in the partition $\{O_1,\ldots,O_m\}$ are of the same size $|H|$. 
A partition $\pi$ of $[n]$ is stabilized by $G$ if and only if there is a non-zero subgroup $H$ of $G$ such that $\pi=\pi_H$.

As we have already seen
\begin{equation*} {[\mathbf{0},\pi_H]}
  \cong_{H\text{-poset}}
  \Pi_{O_1}\times\cdots\times\Pi_{O_m}
  \cong_{H\text{-poset}} 
  \underbrace{\Pi_{H}\times\cdots\times\Pi_{H}}_{m\;\text{times}}
  \cong_{\text{poset}}
  \underbrace{\Pi_{|H|}\times\cdots\times\Pi_{|H|}}_{m\; \text{times}},
\end{equation*}
and consequently
\begin{equation*}
  \Delta(\mathbf{0},\pi_H)
  \simeq 
  \Sigma^{m-1}
  \Big(
  \underbrace{\Delta(\bar{\Pi}_{|H|})*\cdots *\Delta(\bar{\Pi}_{|H|})}_{m\;\text{times}}
  \Big).
\end{equation*}
Therefore, the $i$-th reduced homology of $\Delta(\mathbf{0},\pi_H)$ is
non-trivial if and only if $i=n-m-2$, and
\begin{equation}
  \label{eq:TensorInduced}
  \tilde{H}_{n-m-2}(\Delta(\mathbf{0},\pi_H);R)
  \cong
  \tilde{H}_{|H|-3}(\Delta(\bar{\Pi}_{|H|});R)^{\otimes G/H}
  \cong
  \tilde{H}_{|H|-3}(\Delta(\bar{\Pi}_{H});R)^{\otimes G/H}
\end{equation}
is the ``tensor induced'' $R[G]$-module obtained from the $R[H]$-module
$\tilde{H}_{|H|-3}(\Delta(\bar{\Pi}_{|H|});R)$, see~\cite[Chapter~5.1, page~45]{Evens}.  
Here the $H$-action on $\tilde{H}_{|H|-3}(\Delta(\bar{\Pi}_{|H|});R)$ is induced from the action of $H$ on the partition lattice $\Pi_H$.

%%%%%%%%%%%%%%%%%%%%%%%%%%%%%%%%%%%%%%%%%%%%%%%%%%%%%%%%%%%%%%%%%%%%%%%%%%%%%%%%%%%%%
%%%%%%%%%%%%%%%%%%%%%%%%%%%%%%%%%%%%%%%%%%%%%%%%%%%%%%%%%%%%%%%%%%%%%%%%%%%%%%%%%%%%%
% -----------------------------------------------------------------------------------%
\section{Differentials in the Serre spectral sequence of the Borel construction}
\label{Sec:DiffSSSeq}
% -----------------------------------------------------------------------------------%
%%%%%%%%%%%%%%%%%%%%%%%%%%%%%%%%%%%%%%%%%%%%%%%%%%%%%%%%%%%%%%%%%%%%%%%%%%%%%%%%%%%%%
%%%%%%%%%%%%%%%%%%%%%%%%%%%%%%%%%%%%%%%%%%%%%%%%%%%%%%%%%%%%%%%%%%%%%%%%%%%%%%%%%%%%%

Fix a prime $p$. Let $\mathfrak{I}_G$ denote the family of all $\FF_p[G]$-modules that are
$\FF_p[G]$-isomorphic to finite direct sums of $\FF_p[G]$-modules of the shape
$\ind^G_H N$ for some subgroup $H \subseteq G, H \not= G$ and some $\FF_p[H]$-module~$N$.
In particular we assume that $0\in\mathfrak{I}_G$.
In our case coinduced and induced $\FF_p[G]$-modules
coincide,~\cite[Proposition~III.5.9 page~70]{Brown}, therefore we do not
distinguish between them. Define a wider family of $\FF_p[G]$-modules $\mathfrak{FI}_G$
to consist of those  $\FF_p[G]$-modules $M$ for which there exists an integer $r>0$
and a (finite) filtration of $M$ by $\FF_p[G]$-modules
\[
0=M_0\subseteq M_1 \subseteq\cdots\subseteq M_r = M
\]
such that all quotients $M_{i}/M_{i-1}$ belong to $\mathfrak{I}_G$. 

The main result of this section treats the Serre spectral sequence for the Borel construction of the $G$-space $X$, i.e., for the fibration $X\longrightarrow \EE G\times_{G}X\longrightarrow\BB G$:

\begin{theorem}
  \label{Th:DiffSSSeq-EAb}
  Let $G=\left(\ZZ/p\right) ^k$ be an elementary abelian group and $n$ be a
  natural number.  Let $X$ be a connected $G$-space such that $H^i(X;\FF_p)\in
  \mathfrak{FI}_G$ for every $i\in \{1,\ldots,n\}$.  Then for every $r\geq 0$
  and every $s\in\{2,\ldots,n+1\}$ the differential
  \[
  \partial_{s}:E_{s}^{r,s-1}(\EE G\times _{G}X)\longrightarrow E_{s}^{r+s,0}(\EE G\times_{G}X)
  \]
  vanishes. Consequently,
  \[
  \Index{G}(X;\FF_p)\subseteq H^{\geq n+2}_{G}(\mathrm{pt};\FF_p).
  \]
\end{theorem}

An immediate consequence of the previous theorem is an extension of the
generalized Dold theorem for elementary abelian groups from~\cite[Theorem~16, page~1934]{BlDiMcC}.

\begin{theorem}[Generalized Dold theorem]
  \label{Th:Dold} Let $G=\left(\ZZ/p\right) ^k$ be an elementary abelian group
  and let $n$ be a natural number.  Let $X$ and $Y$ be connected
  $G$-spaces. Suppose that $H^i(X;\FF_p)\in\mathfrak{FI}_{G}$ for every $1\leq
  i\leq n$, and $\pi^{*}_X\colon H^{j}_G(\mathrm{pt};\FF_p )\rightarrow
  H^{j}_G(Y;\FF_p)$ is not injective for some $1\leq j\leq n+1$.

  Then there is no $G$-equivariant map $X\to Y$.
\end{theorem}

The proof of Theorem~\ref{Th:DiffSSSeq-EAb} needs some preparation.

The cohomology of the group $G=\left(\ZZ/p\right) ^k$ with coefficients in the
field $\FF_p$ is given by:
\begin{equation*}
  \begin{array}{llll}
    H^*(\left( \ZZ/2\right) ^k;\FF_2)        & =    & \FF_2[t_1,\ldots,t_k],                                 & \deg t_j=1,                     \\
    H^*(\left( \ZZ/p\right) ^k;\FF_p)        & =    & \FF_p[t_1,\ldots,t_k]\otimes \Lambda [e_1,\ldots,e_k], & \deg t_j=2,\deg e_i=1,\text{ for }p>2,
  \end{array}
\end{equation*}
where $\Lambda[e_1,\ldots,e_k]$ denotes the exterior algebra generated by elements $e_1,\ldots,e_k$.

The cohomology algebra $H^*(G;\FF_p)$ contains the maximal multiplicative set
\begin{equation*}
  S_{G}:=\text{(polynomial part of }H^*(G;\FF_p)\text{)}{\setminus}\{0\}=\left\{
    \begin{array}{lll}
      \FF_2[t_1,\ldots,t_k]{\setminus}\{0\},         &     & \text{for} \;G=\left( \ZZ/2\right) ^k, \\
      \FF_p[t_1,\ldots,t_k]{\setminus}\{0\},         &     & \text{for}\; G=\left( \ZZ/p\right) ^k\;\text{and}\;p>2.
    \end{array}
  \right.
\end{equation*}
The central property \cite[Proof of Proposition~1, page~45]{Hsiang:cohomology}~\cite[Lemma~15]{BlDiMcC} of the multiplicative set $S_G$ and
the class of elementary abelian groups is that
\begin{equation}
  \label{eq:Intersection_Res_and_SG}
  \bigcap _{H\in \mathrm{Sub}_G}\ker \left(\res_H^G\colon H^*(G;\FF_p)\To H^*(H;\FF_p)\right)\cap S_G\neq \emptyset;
\end{equation}
here $\mathrm{Sub}_{G}$ stands for the collection of all proper subgroups of the group $G$, i.e., all the subgroups different from $G$.

\medskip

We call an endomorphism $f \colon M \to M$ of a graded abelian group \emph{nilpotent of degree $\le d$} if $f^d = 0$, and \emph{nilpotent}
if it is nilpotent of degree $\le d$ for some natural number $d$.
\newpage

\begin{lemma}
\label{lem;properties_of_mathfrac_and_nil}
\qquad
  \begin{compactenum}[\rm (i)]
  \item \label{lem;properties_of_mathfrac_and_nil_NI_and_extensions} 
If $0 \to L \xrightarrow{i} M \xrightarrow{p}  N \to 0$ is an exact sequence of $\FF_p[G]$-modules and $L$ and $N$ belong to $\mathfrak{FI}_G$, then $M$ also belongs to $\mathfrak{FI}_G$.
  \item \label{lem;properties_of_mathfrac_and_nil:nilpotnent_and-extensions}
Consider the following diagram of graded abelian groups, where the horizontal maps are degree preserving, while the vertical maps do not necessarily preserve the degree:
\[
\xymatrix{ M_0 \ar[r]^{i} \ar[d]^{f_0} & M_1 \ar[r]^{p} \ar[d]^{f_1} & M_2
  \ar[d]^{f_2}
  \\
  M_0 \ar[r]^{i} & M_1 \ar[r]^{p} & M_2 }
\]
Suppose that the rows are degreewise exact at $M_1$, and that $f_0$ and $f_2$ are nilpotent of degree $\le d_0$ resp. $\le d_2$.
Then $f_1$ is nilpotent of degree $\le (d_0+d_2)$.
\item \label{lem;properties_of_mathfrac_and_nil:submodules_and_quotient_modules}
  Let $f \colon N \to N$ be an endomorphism of the graded abelian group $N$
  which is nilpotent of degree $\le d$.  If $M \subseteq N$ is a graded abelian
  subgroup with $f(M) \subseteq M$, then $f|_{M} \colon M \to M$ and the induced
  map $\overline{f} \colon N/M \to N/M$ are nilpotent of degree $\le d$.
\end{compactenum}
\end{lemma}
\begin{proof}~(\ref{lem;properties_of_mathfrac_and_nil_NI_and_extensions}) By
  assumption we can choose filtrations $0=L_0\subseteq L_1
  \subseteq\cdots\subseteq L_r = L$ and $0=N_0\subseteq N_1
  \subseteq\cdots\subseteq N_s = N$ such that each quotient belongs to
  $\mathfrak{I}_G$.  Define $M_k = i(L_k)$ for $k = 0,1,2 \ldots, r$ and $M_k =
  p^{-1}(N_{k-r})$ for $k = (r+1), \ldots, (r+s)$. Then we obtain a filtration
  $0=M_0\subseteq M_1 \subseteq\cdots\subseteq M_{r+s} = M$ such that each
  quotient belongs to $\mathfrak{I}_G$. Hence $M$ belongs to $\mathfrak{FI}_G$.
  \\[1mm]~(\ref{lem;properties_of_mathfrac_and_nil:nilpotnent_and-extensions})
  Consider $x \in M_1$. Then $p \circ f_1^{d_2}(x) = (f_2)^{d_2} \circ p(x) =
  0$. Hence there exists $y \in N_0$ with $i(y) = f_1^{d_2}(x)$. We conclude that
  $f_1^{d_0 + d_2}(x) = f_1^{d_0} \circ f_1^{d_2}(x) = f_1 ^{d_2} \circ i(y) = i
  \circ f_0^{d_0}(y) = 0$.  Hence $f_1$ is nilpotent of degree $\le (d_0 +
  d_2)$.
  \\[1mm]~(\ref{lem;properties_of_mathfrac_and_nil:submodules_and_quotient_modules})
  This is obvious.
\end{proof}

\begin{lemma}
  \label{lem:nilpotent_operation}
  Let $M$ be an $\FF_p[G]$-module in $\mathfrak{FI}_{G}$. 
  Consider an element $\xi \in H^*(G;\FF_p)$ which is mapped to zero under the restriction map $\res_{H}^G\colon H^*(G;\FF_p) \to H^*(H;\FF_p)$ for every 
  $H \in \mathrm{Sub}_G$. 
  Then multiplication with $\xi$ is a nilpotent map $\xi \colon H^*(G;M) \to H^*(G;M)$.
\end{lemma}
\begin{proof}
  We use induction over $n$ for which there exists an $\FF_p[G]$-filtration $0 =
  M_0 \subseteq M_1 \subseteq \cdots \subseteq M_n = M$ such that each quotient
  belongs to $\mathfrak{I}_G$. The induction beginning $n = 0$ is trivial. In
  the induction step from $n$ to $n+1$ we consider the exact a sequence $0 \longrightarrow
  M_n \longrightarrow M \longrightarrow M/M_n \to 0$.  It induces a sequence $H^*(G;M_n) \longrightarrow H^*(G;M)
  \longrightarrow H^*(G;M/M_n)$ which is exact at $H^*(G;M)$. By induction hypothesis $\xi
  \colon H^*(G;M_n) \to H^*(G;M_n)$ is nilpotent.  
  In order to show that $\xi  \colon H^*(G;M) \to H^*(G;M) $ is nilpotent, using 
  Lemma~\ref{lem;properties_of_mathfrac_and_nil}\,(\ref{lem;properties_of_mathfrac_and_nil:nilpotnent_and-extensions}),
  it is enough to prove that $\xi \colon H^*(G;M/M_n) \to H^*(G;M/M_n)$ is nilpotent. 
  Since $M/M_n$ belongs to $\mathfrak{I}_{G}$, it suffices to show that $\xi \colon
    H^*(G;\mathrm{\ind}_H^G N) \to H^*(G;\ind_H^G N)$ is trivial 
  for any subgroup $H\in \mathrm{Sub}_G$
  and any $\ZZ[H]$-module $N$. 
  By Shapiro's
  Lemma \cite[Proposition 6.2, page 73]{Brown} there is an isomorphism $\alpha \colon H^*(G;\mathrm{\ind}_H^G N)
  \xrightarrow{\cong} H^*(H;N)$ such that $\alpha(\xi \cdot x) =
  \mathrm{res}^G_H(\xi) \cdot \alpha(x)$ holds for all $x \in
  H^*(G;\mathrm{\ind}_H^G N)$. Since $\mathrm{res}^G_H(\xi) = 0$ by assumption,
  the claim follows.
\end{proof}

\begin{proof}[Proof of Theorem~\ref{Th:DiffSSSeq-EAb}]
  Because of~\eqref{eq:Intersection_Res_and_SG} we can choose $0\neq \xi\in
  S_{G}\cap \bigcap_{\alpha \in \Lambda}\ker (\res_{H_{\alpha}}^G)$.  The Serre
  spectral sequence comes with a module structure over the graded ring
  $H^*(G;\FF_p)$.  The $E_2$-term looks like
  $E_{2}^{*,s}=H^{*}(G;H^{s}(X;\FF_p))$.  Hence multiplication with $\xi$
  induces a nilpotent map $\xi \colon E_{2}^{*,s} \to E_{2}^{*,s}$ for $1\leq s\leq n$ by
  Lemma~\ref{lem:nilpotent_operation}. 
  Since all differentials are $H^*(G;\FF_p)$-maps, we conclude from
  Lemma~\ref{lem;properties_of_mathfrac_and_nil}\,(\ref{lem;properties_of_mathfrac_and_nil:submodules_and_quotient_modules})
  that there is a natural number $d$ such that the map $\xi^d \colon E_{r}^{*,s} \to
  E_{r}^{*,s}$ given by multiplication with $\xi^d$ is trivial for all $1\leq s\leq n$ and $r\geq 2$.
  Notice for the sequel that for every $x \in E_{2}^{r,s-1}$ we have
  $\partial_s(\xi^d \cdot x) = \xi^d \cdot \partial_s(x)$, and hence $\xi^d
  \cdot \partial_s(x) = 0$.

  Now we show by induction that $\partial_{s}\colon E_{s}^{r,s-1}(\EE G\times
  _{G}X)\longrightarrow E_{s}^{r+s,0}(\EE G\times_{G}X)$ vanishes for
  $s\in\{2,\ldots,n+1\}$.  The induction beginning $s = 2$ follows from the fact
  that the map $\xi^d \colon E_2^{*,0} = H^*(G;\FF_p) \to E_2^{*,0} =
  H^*(G;\FF_p)$ is injective since $\xi$ belongs to $S_{G}$.  Finally we explain
  the induction step from $s-1$ to $s \ge 3$. By induction hypothesis all
  differentials landing in the $0$-th row in the $E_i$-term are trivial for $i
  \le s-1$. Hence $E_s^{*,0} = E_2^{*,0} = H^*(G;R)$.  Now the same argument
  as above shows that $\partial_{s} \colon E_{s}^{r,s-1} \to E_s^{r+s,0}$ is trivial.
\end{proof}

%%%%%%%%%%%%%%%%%%%%%%%%%%%%%%%%%%%%%%%%%%%%%%%%%%%%%%%%%%%%%%%%%%%%%%%%%%%%%%%%%%%%%
%%%%%%%%%%%%%%%%%%%%%%%%%%%%%%%%%%%%%%%%%%%%%%%%%%%%%%%%%%%%%%%%%%%%%%%%%%%%%%%%%%%%%
% -----------------------------------------------------------------------------------%
\section{Equivariant obstruction theory, Euler classes and Lusternik--Schnirelmann category}
\label{Sec:EqOb_Euler_LS-cat}
% -----------------------------------------------------------------------------------%
%%%%%%%%%%%%%%%%%%%%%%%%%%%%%%%%%%%%%%%%%%%%%%%%%%%%%%%%%%%%%%%%%%%%%%%%%%%%%%%%%%%%%
%%%%%%%%%%%%%%%%%%%%%%%%%%%%%%%%%%%%%%%%%%%%%%%%%%%%%%%%%%%%%%%%%%%%%%%%%%%%%%%%%%%%%

\subsection{Equivariant primary obstructions and Euler classes}
\label{subsec:Equivariant_Primary_Obstructions_and_Euler_Classes}
Let $G$ be a finite group and $X$ be a free $G$-CW complex.
Consider an orthogonal $G$-representation $W$. 
Denote by $\xi$ the associated vector bundle $X \times_G W\longrightarrow X/G$ and by $S(\xi)$ the related sphere bundle $X \times_G S(W) \longrightarrow X/G$.
\begin{lemma}\label{sections_and_maps}
  There is a one-to-one correspondence between sections of the sphere bundle $S(\xi) \colon X \times_G S(W) \longrightarrow X/G$ and $G$--equivariant maps from $X\to S(W)$.
\end{lemma}
\begin{proof}
  Given a $G$--equivariant map $f \colon X \to S(W)$, we obtain a section of $S(\xi)$ by
  taking the $G$-quotient of the map $X \to X \times S(W),\; x \mapsto
  (x,f(x))$.
   
  Let $s \colon X/G \to X \times_G S(W)$ be a section of $S(\xi)$.  Consider the
  following diagram whose squares are pullbacks
  \[
  \xymatrix{s^*p^*X \ar[r]^{\overline{s}} \ar[d] & p^*X \ar[r]^{\overline{p}}
    \ar[d] & X\ar[d]
    \\
    X/G \ar[r]^s & X \times_G S(W) \ar[r]^p & X/G .}
  \]
  There is a canonical isomorphism of $G$-coverings over $X \times_G S(W)$ from $X
  \times S(W) \to X \times_G S(W)$ to $p^*X \to X \times_G S(W)$.  Since $p \circ
  s = \id$, we obtain a preferred isomorphism of $G$-coverings over $X/G$ from
  $X \to X/G$ to $s^*p^*X \to X/G$.  Hence the $G$-equivariant map $\overline{s}$ can be
  identified with a $G$-equivariant map $X \to X \times S(W)$.  Its composition with the
  projection $X \times S(W) \to S(W)$ yields a $G$-equivariant map $X \to S(W)$.
\end{proof}

Now suppose that $X$ is a $d$-dimensional connected free $G$-$CW$-complex for $d
= \dim(W)$.  Let $w \colon G \to \{\pm 1\}$ be the orientation homomorphism of
$W$, i.e., $w(g)$ is $1$ if $g$ acts orientation preserving and is $-1$
otherwise. The first Stiefel-Whitney class $w_1(\xi)\in H^1(X/G;\ZZ/2)$ is given
by the composition $w_1(\xi)\colon \pi_1(X/G) \xrightarrow{\partial} G
\xrightarrow{w} \{\pm 1\}$, where $\partial$ is the classifying map associated to
the $G$-covering $X \to X/G$. 

There is the notion of the \emph{Euler class}
\begin{eqnarray*}
  & e(\xi) \in H^d(X/G;\mathcal{Z})&
\end{eqnarray*}  
where $H^d(X/G;\mathcal{Z})$ is the cohomology of $X/G$ with coefficients in the
local coefficient system $\mathcal{Z}$, which assigns to $y \in X/G$ the
$(d-1)$-st homotopy group of the fiber of $S(\xi)$ over $y$. It is defined as the
primary obstruction to the existence of a section of $S(\xi)$ and is a
characteristic class.  See for instance~\cite{Greenblatt}, where further
references, e.g.,~\cite{Steenrod(1951)} and~\cite{Thom(1952)}, are given.  The
orientable case, i.e., $w_1$ is trivial, is treated in~\cite[\S~9 and \S~12]{M-S}.

There is a natural identification
\[
H^d(X/G;\mathcal{Z}) \cong H^d(X/G;\ZZ^{w_1(\xi)})
\]
where $H^d(X/G;\ZZ^{w_1(\xi)})$ is 
$H^d\bigl(\hom_{\ZZ[ \pi_1(X/G)]}(C_*(\widetilde{X/G}),\ZZ^{w_1(\xi)})\bigr)$ for the
$\ZZ[\pi_1(X/G)]$-module $\ZZ^{w_1(\xi)}$ whose underlying abelian group is
$\ZZ$ and for which $g$ acts by multiplication with $w_1(X/G)(g)$. Moreover,
there is a natural identification
\[
H^d(X/G;\ZZ^{w_1(\xi)}) \cong H^d_G(X;\ZZ^w) := H^d\bigl(\hom_{\ZZ
  G}(C_*(X),\ZZ^w)\bigr).
\]
The primary equivariant obstruction for the existence of a $G$-equivariant map $X \to S(W)$ is
an element (see~\cite[page~120]{tDieck})
\[
\gamma^G(X,S(W)) \in H^d_G(X;\ZZ^w).
\]
The reduction of coefficients $\ZZ^w$ to $\ZZ/2$ defines the (mod\,2)-primary obstruction
\[
\gamma^G_{\ZZ/2}(X,S(W)) \in H^d_G(X;\ZZ/2)
\]
in a natural way.

The proof of the next lemma consists of unravelling the definitions of the
primary obstruction for the existence of a section of $S(\xi)$ and the
equivariant primary obstruction for the existence of a $G$-equivariant map $X \to S(W)$
using Lemma~\ref{sections_and_maps} and the fact that cells in $X/G$ correspond
to equivariant cells in $X$.

\begin{lemma}
  \label{lem:Euler_and_o}
  Let $X$ be a $d$-dimensional connected free $G$-$CW$-complex for $d =
  \dim(W)$. Then the composite of  natural isomorphisms
  \[H^d(X/G;\mathcal{Z}) \xrightarrow{\cong} H^d(X/G;\ZZ^{w_1(\xi)})
  \xrightarrow{\cong} H^d_G(X;\ZZ^w)
  \]
  maps the Euler class $e(\xi)$ to the equivariant primary obstruction  $\gamma^G(X;S(W))$.
\end{lemma}

The reduction of twisted coefficients to $\ZZ/2$ yields a similar identification of the (mod 2)-primary obstruction with the appropriate Stiefel--Whitney class.
\begin{lemma}
  \label{lem:SW_and_o}
  Let $X$ be a $d$-dimensional connected free $G$-$CW$-complex for $d =
  \dim(W)$. Then the natural isomorphism
  \[H^d(X/G;\ZZ/2) \xrightarrow{\cong} H^d_G(X;\ZZ/2)
  \]
  maps the top Stiefel--Whitney class $w_d(\xi)$ to the equivariant ${\rm (mod~2)}$-primary obstruction $\gamma^G_{\ZZ/2}(X;S(W))$.
\end{lemma}
%%%%%%%%%%%%%%%%%%%%%%%%%%%%%%%%%%%%%%%%%%%%%%%%%%%%%%%%%%%%%%%%%%%%%%%%%%%%%%%%%%%%%

%%%%%%%%%%%%%%%%%%%%%%%%%%%%%%%%%%%%%%%%%%%%%%%%%%%%%%%%%%%%%%%%%%%%%%%%%%%%%%%%%%%%%

\subsection{Restriction and transfer}
\label{subsec:Restriction_and_Transfer}

\begin{lemma}\label{lem:transfer}
  Let $G$ be a finite group and $H \subseteq G$ be a subgroup. 
  Consider a $\ZZ[G]$-chain complex $C_*=(C_n,c_n)$ and a $\ZZ[G]$-module $M$.
  Denote by $\res$ the restriction from $G$ to $H$.

  Then for any $n$ there is a restriction homomorphism
  \[
  \res \colon H^n\bigl(\hom_{\ZZ[G]}(C_*,M)\bigr) \to
  H^n\bigl(\hom_{\ZZ[H]}(\res C_*,\res M)\bigr),
  \]
  and a transfer homomorphism
  \[
  \trf \colon H^n\bigl(\hom_{\ZZ[H]}(\res C_*,\res M)\bigr) \to
  H^n\bigl(\hom_{\ZZ[G]}(C_*,M)\bigr),
  \]
with the property that
  \[
  \trf \circ \res = [G:H] \cdot \id.
  \]
\end{lemma}
\begin{proof}
  Choose a map of sets $s \colon G/H \to G$ such that the composite $\pr\circ s$ with
  the projection $\pr \colon G \to G/H$ is the identity. Given a $\ZZ[H]$-map
  $\phi \colon \res C_n \to\res M$, define a $\ZZ[G]$-map
  \[
  \trf_n(\phi) \colon C_n \to M, \quad x \mapsto \sum_{gH \in G/H} s(gH) \cdot
  \phi(s(gH)^{-1} \cdot x).
  \]
  This definition is actually independent of the choice of $s$, as the following
  calculation for another section $s'$ shows
  \begin{eqnarray*}
    s'(gH) \cdot \phi(s'(gH)^{-1} x)
    & = & 
    s(gH) \cdot s(gH)^{-1} \cdot s'(gH) \cdot \phi(s'(gH)^{-1} \cdot x)
    \\
    & = & 
    s(gH) \cdot \phi\bigl( s(gH)^{-1} \cdot s'(gH) \cdot s'(gH)^{-1} \cdot x\bigr)
    \\
    & = & 
    s(gH) \cdot \phi( s(gH)^{-1} x).
  \end{eqnarray*}
  The collection of the maps $\trf_n$ yields a $\ZZ[H]$-cochain map $\trf_* \colon
  \hom_{\ZZ[H]}(\res C_*, \res M) \to \hom_{\ZZ[G]}(C_*, M)$ by the following
  calculation for $\phi \in \hom_{\ZZ[H]}(\res C_*, \res M)$ and $x \in C_{n+1}$
  using the fact that the differentials of $C_*$ are $G$-equivariant
  \begin{eqnarray*}
    \trf_n(\phi) (c_{n+1}(x))
    & = & 
    \sum_{gH \in G/H} s(gH) \cdot \phi\bigl(s(gH)^{-1} \cdot c_{n+1}(x)\bigr)
    \\
    & = & 
    \sum_{gH \in G/H} s(gH) \cdot (\phi\circ c_{n+1})\bigl(s(gH)^{-1} \cdot x\bigr)
    \\
    & = & 
    \trf_n(\phi_n \circ c_{n+1}).
  \end{eqnarray*}
  The desired transfer map $\trf \colon H^n\bigl(\hom_{\ZZ[H]}(\res C_*,\res
  M)\bigr) \to H^n\bigl(\hom_{\ZZ[G]}(C_*,M)\bigr)$ is obtained by applying
  cohomology to the cochain map $\trf_*$.  
  
  In order to prove that $\trf \circ \res =[G:H]$, it suffices to show that for any $\ZZ[G]$-map $\phi \colon C_n \to M$
  \begin{eqnarray*}
    \trf_n \circ \res(\phi)
    & = & 
    \sum_{gH \in G/H} s(gH) \cdot \phi\bigl(s(gH)^{-1} \cdot x\bigr)
    \\ 
    & = & 
    \sum_{gH \in G/H} s(gH) \cdot s(gH)^{-1} \cdot  \phi(x)
    \\ 
    & = & 
    \sum_{gH \in G/H}  \phi(x)
    \\ & = & [G:H] \cdot \phi.
  \end{eqnarray*}
  This finishes the proof of Lemma~\ref{lem:transfer}.
\end{proof}

%%%%%%%%%%%%%%%%%%%%%%%%%%%%%%%%%%%%%%%%%%%%%%%%%%%%%%%%%%%%%%%%%%%%%%%%%%%%%%%%%%%%%

\subsection{Lusternik--Schnirelmann category}
\label{subsec:Lusternik-Schnirelmann_Category}

The \emph{Lusternik--Schnirelmann category} $\cat(X)$ of a space $X$ is the least
integer $n$ for which $X$ can be covered by $n+1$ open subsets $U_1, U_2,
\ldots, U_{n+1}$ such that the inclusions $U_i \to X$ are nullhomotopic.

\noindent 
The \emph{sectional category} $\secat(p)$ of the fibration $F\to E\overset{p}{\to}B$ is the minimal integer $n$ for which $B$ can be covered by $n+1$ open subsets $U_1, U_2,\ldots, U_{n+1}$ such that each restriction fibration $F\to p^{-1}(U_i)\to U_i$ admits a section $s_i\colon U_i\to p^{-1}(U_i)$.
Originally, the notion of sectional category was introduced by Schwarz in \cite{Schwarz} under the name {\em genus}. 

A few key properties of the Lusternik--Schnirelmann and sectional category that we use are stated in the next lemma.

\begin{lemma}
\label{lem:LS-1}
\qquad
\begin{compactenum}[\rm (1)]
\item If $X$ is homotopy equivalent to $Y$, then $\cat(X)=\cat(Y)$.
\item If $p:X\to Y$ is a covering, then $\cat(X)\leq\cat(Y)$.
\item If $X$ is an $(n-1)$-connected $CW$-complex, then $\cat(X)\leq\tfrac{1}{n}\dim(X)$.
\item If $F\to E\overset{p}{\to}B$ is a fibration, then $\secat(p)\leq\cat(B)$.
\end{compactenum}
\end{lemma}

Let $X$ be a topological space and $R$ be a commutative ring with unit.
The \emph{category weight} of the element $u\in H^*(X;R)$ is
\[
\wgt(u):=\left\{
\begin{array}{ll}
    \max \{k:p^*_{k-1}(u)=0 \} ,  &  \text{ if the maximum exists,}\\
    \infty  ,                     &  \text{ otherwise.}
\end{array}
\right.
\]
Here $p_{k-1}\colon G_{k-1}\to X$ denotes the $(k-1)$st Ganea fibration~\cite{Ganea}.
This definition of the category weight is due to Rudyak~\cite{Rud} and Strom~\cite{Strom}.
For more details consult~\cite[Section~2.7, page~62; Section~8.3, page~240]{CLOT}.
Properties of the category weight that we use are collected in the lemma that follows~\cite[Proposition~8.22, pages~242--243, page 259]{CLOT}, \cite[Proposition 2.2(3)]{Roth}.

\begin{lemma}
\label{lem:LS-2}
Let $R$ be a commutative ring with unit.
\begin{compactenum}[\rm (1)]
\item If $0\neq u\in H^{\ell}(X;R)$, then $\wgt(u)\leq\cat(X)$.
\item Let $f:X\to Y$ be a continuous map and $u\in H^{\ell}(Y;R)$.
      If $0\neq f^*(u)\in H^{\ell}(X;R)$, then $\wgt(u)\leq\wgt(f^*(u))$.
\item Let $G$ be a finite group. 
      If $0\neq u\in H^{\ell}(\BB G;R)$, then $\ell\leq\wgt(u)$.
\item Let $F\to E\overset{p}{\to}B$ be a fibration that is a pullback of a fibration  $F\to \hat{E}\overset{\hat{p}}{\to}\hat{B}$ along the map $f\colon B\to\hat{B}$.
If $\hat{E}$ is contractible, $0\neq u\in H^{\ell}(\hat{B};R)$ and $f^*(u)\neq0$, then $\wgt(u)\leq\secat(p)$.
\end{compactenum}
\end{lemma}

The main result of this section is the following theorem.

\begin{theorem}[Lusternik--Schnirelmann category]
\label{the:estimate_for_cat}
Let $G$ be a finite group and let $p$ be a prime.  Let $X$ be a free
$d$-dimensional connected $G$-$CW$-complex and let $W$ be a $d$-dimensional
orthogonal $G$-representation with unit sphere $S(W)$.  Suppose that a
$p$-Sylow subgroup $G^{(p)}$ acts orientation preserving on $S(W)$.  (This is
automatically satisfied if $p$ is odd.)  Suppose that there exists no $G$-equivariant map $X
\to S(W)$ and that every torsion element in $H_G^d(X;\pi_{d-1}(S(W)))$ has
$p$-power order.
\begin{compactenum}[\rm (1)]
\item The Lusternik--Schnirelmann category $\cat(X/G)$ of the quotient space is
\[
\cat(X/G)= d.
\]

\item Let $Y$ be a free $G$-space. 
If there exists a $G$-equivariant map $h\colon X\to Y$, then the Lusternik--Schnirelmann category $\cat(Y/G)$ of the quotient space satisfies
\[
\cat(Y/G)\geq d.
\]
\end{compactenum}

\end{theorem}
\begin{proof}
Since there is no $G$-equivariant map $X \to S(W)$, the equivariant primary obstruction 
\[
\gamma^G(X,S(W)) \in H^d_G(X;\pi_{d-1}(S(W)))
\]
is non-trivial. 
Here $\pi_{d-1}(S(W))\cong_{{\rm Ab}}\ZZ$ is considered as a $G$-module that need not be trivial.
On the other hand, by assumption, $\pi_{d-1}(S(W))$ is a trivial $G^{(p)}$-module. 
The restriction homomorphism 
\[
\res_{G^{(p)}}^G \colon H^d_{G}(X;\pi_{d-1}(S(W))) \to H^d_{G^{(p)}}(X;\ZZ)
\]
maps $\gamma^G(X,S(W))$ to $\gamma^{G^{(p)}}(X,S(W))$. 
Since  every torsion element in $H_G^d(X;\pi_{d-1}(S(W)))$ has $p$-power order the composition 
\[
\trf_{G^{(p)}}^G \circ \res_{G^{(p)}}^G = [G:G^{(p)}] \cdot \id\colon H^d_{G}(X;\pi_{d-1}(S(W))) \to H^d_{G}(X;\pi_{d-1}(S(W)))
\]
is an injection.
Thus, $\res_{G^{(p)}}^G \colon H^d_G(X;\pi_{d-1}(S(W))) \to H^d_{G^{(p)}}(X;\ZZ)$ 
is injective and $\gamma^{G^{(p)}}(X,S(W))\in H^d_{G^{(p)}}(X;\ZZ)$ is non-trivial.
 
\textrm{(1)} 
Let $\xi$ be the vector bundle $X \times_{G^{(p)}} W\longrightarrow X/G^{(p)}$ and $\eta$ be the vector bundle $\EE G^{(p)} \times_{G^{(p)}} W\longrightarrow \BB G^{(p)}$. 
Choose a classifying map $f \colon X/{G^{(p)}} \to \BB G^{(p)}$ for the $G^{(p)}$-covering $X \to X/G^{(p)}$. 
Then $\xi$ is isomorphic to $f^*\eta$, and there is the following pullback diagram:
\[
\xymatrix@!C=9em
{
X \times_{G^{(p)}} W \ar[r]  \ar[d]^{\xi}&\EE G^{(p)} \times_{G^{(p)}}W\ar[d]^{\eta}\\
X/G^{(p)} \ar[r]^{f}                               &\BB G^{(p)} .
}
\]
By Lemma~\ref{lem:Euler_and_o} the Euler class $e(\xi) \in H^d(X/G^{(p)};\ZZ)$ is non-trivial. 
The naturality property of Euler classes implies that $H^d(f) \colon H^d(\BB G^{(p)};\ZZ)\to H^d(X/G^{(p)};\ZZ)$ maps $e(\eta)$ to $e(\xi)$. 
Since $e(\xi)$ is non-trivial, using Lemma~\ref{lem:LS-2}, we have that
\[
d \leq \wgt(e(\eta)) \leq \wgt(e(f^* \eta)) = \wgt(e(\xi)) \leq \cat(X/G^{(p)}).
\]
Furthermore, the quotient map $X/G^{(p)}\to X/G$ is a covering and therefore by Lemma~\ref{lem:LS-1} we get
\[
\cat(X/G^{(p)}) \leq \cat(X/G) \leq \dim(X/G) = d.
\]

\textrm{(2)} Let $\hat{h}\colon X/G^{(p)}\to Y/G^{(p)}$ be the quotient map. 
Consider $\chi$ the bundle $Y \times_{G^{(p)}} W\longrightarrow Y/G^{(p)}$ and  $g \colon Y/{G^{(p)}} \to \BB G^{(p)}$ a classifying map for the $G^{(p)}$-covering $Y \to Y/G^{(p)}$.
Then $\chi$ is isomorphic to $g^*\eta$.
The composition $l:=g\circ\hat{h}$ is homotopic to the classifying map $f$.
Consequently, $\xi$ is isomorphic to $l^*\eta$.
The relationship between these bundles can be illustrated by the following diagram:
\[
\xymatrix@!C=9em
{
X \times_{G^{(p)}} W\ar[r]\ar[d]^{\xi} & Y \times_{G^{(p)}} W\ar[r] \ar[d]^{\chi} & \EE G^{(p)} \times_{G^{(p)}}W\ar[d]^{\eta}\\
X/G^{(p)}\ar[r]^-{\hat{h}} &  Y/G^{(p)} \ar[r]^-{g}            & \BB G^{(p)}.
}
\]
The naturality property of Euler classes implies that
\[
\xymatrix{
e(\eta)\ar[r]^{H^d(g)} & e(\chi)\ar[r]^{H^d(\hat{h})} & e(\xi).
         }
\]
We have seen that $e(\xi)\neq 0$.
Therefore, $e(\chi)\neq 0$ and so by Lemmas~\ref{lem:LS-1} and~\ref{lem:LS-2} 
\[
d \leq \wgt(e(\eta)) \le \wgt(e(g^* \eta)) = \wgt(e(\chi))  \le \cat(Y/G^{(p)}) \le \cat(Y/G).
\]
This concludes the proof of the theorem.
\end{proof}

%%%%%%%%%%%%%%%%%%%%%%%%%%%%%%%%%%%%%%%%%%%%%%%%%%%%%%%%%%%%%%%%%%%%%%%%%%%%%%%%%%%%%
%%%%%%%%%%%%%%%%%%%%%%%%%%%%%%%%%%%%%%%%%%%%%%%%%%%%%%%%%%%%%%%%%%%%%%%%%%%%%%%%%%%%%
%-----------------------------------------------------------------------------------%
\section{Fadell--Husseini index of the configuration space}
\label{Sec:FH-Index-I}
%-----------------------------------------------------------------------------------%
%%%%%%%%%%%%%%%%%%%%%%%%%%%%%%%%%%%%%%%%%%%%%%%%%%%%%%%%%%%%%%%%%%%%%%%%%%%%%%%%%%%%%
%%%%%%%%%%%%%%%%%%%%%%%%%%%%%%%%%%%%%%%%%%%%%%%%%%%%%%%%%%%%%%%%%%%%%%%%%%%%%%%%%%%%%

Now we start our study of the Fadell--Husseini ideal-valued index of the configuration space $F(\RR^d,n)$ with respect to different subgroups of the symmetric group:
\begin{compactitem}
\item $\Index{\ZZ/p}(F(\RR^d,p);\FF_p)\subseteq H^*(\ZZ/p;\FF_p)$ 
for $p$ prime, $d>1$ and $\ZZ/p$ acting on $F(\RR^d,p)$ by a cyclic shift,
\item $\Index{(\ZZ/p)^k}(F(\RR^d,p^k);\FF_p)\subseteq H^*((\ZZ/p)^k;\FF_p)$ 
for $p$ prime, $d>1$ and $(\ZZ/p)^k$ acting on $F(\RR^d,p^k)$ as a subgroup of 
$\Sym_{p^k}$ that acts  on $[p^k]=\{1,\ldots , p^k\}$ via the regular embedding $\mathrm{(reg)} \colon (\ZZ/p)^k\to\Sym_{p^k}$. 
\end{compactitem}
In this section we will obtain a complete answer in case when $n$ is a prime.
 
To study the index of the configuration space $F(\RR^d,n)$ with respect to some subgroup $G$ of the symmetric group $\Sym_n$ we will consider the Serre spectral sequence of the fibration
\[
F(\RR^d,n)  \longrightarrow \EE G\times_{G}F(\RR^d,n) \longrightarrow \BB G .
\]

%-----------------------------------------------------------------------------------%
\subsection{Fadell--Husseini index, definition and a few basic properties}
\label{subsec:FH-indexDef}
%-----------------------------------------------------------------------------------%
First we collect some basic properties of the Fadell--Husseini index
that will be used below.  For more details and for proofs of the listed properties
consult~\cite{FH},~\cite{Vol-1} and~\cite{B-Z}.

Let $G$ be a finite group and $R$ be a commutative ring with unit. For a $G$-space
$X$ and a ring $R$, we define the \emph{Fadell--Husseini index} of $X$ to be the
kernel ideal of the map in equivariant cohomology induced by the $G$-equivariant
map $p_X \colon X\to \mathrm{pt}$:

\begin{eqnarray*}
\Index{G}(X;R)  & :=  & \ker\Big( p_X^*\colon H^*_{G}(\mathrm{pt};R)\longrightarrow H^*_{G}(X;R)\Big)\\
                &  =  & \ker\Big(H^*(G;R)\longrightarrow H^*(\EE G\times _{G}X;R)\Big).
\end{eqnarray*}

The Serre spectral sequence of the fibration $X\longrightarrow \EE G\times_{G}X\longrightarrow \BB G$ gives the presentation of the homomorphism 
$p_{X}^* \colon H^*(G;R)\rightarrow H^*(\EE G\times _{G}X;R)$ as the composition
\[
H^*(G;R)\longrightarrow E_{2}^{*,0}\longrightarrow E_{3}^{*,0}\longrightarrow E_{4}^{*,0}
\longrightarrow \cdots\longrightarrow E_{\infty}^{*,0}\subseteq H^*(\EE G\times _{G}X;R).
\]
The \textit{$k$-th partial Fadell--Husseini index} of $X$ is defined by
\begin{eqnarray*}
\Index{G}^1(X;R)      &  :=  &  \{0\},               \\
\Index{G}^r(X;R)      &  :=  &\ker \big(H^*(\BB G;R)\rightarrow E_{r}^{* ,0}\big),\qquad r\geq 2.
\end{eqnarray*}
The partial Fadell--Husseini indexes filter the Fadell--Husseini index
\[
\Index{G}^{1}(X;R)\subseteq \Index{G}^{2}(X;R)\subseteq \cdots\subseteq \Index{G}(X;R),
\]
with $\bigcup_{r\in\mathbb{N}}\Index{G}^{r}(X;R)=\Index{G}(X;R)$.
The (partial) Fadell--Husseini indexes satisfy the following properties:
\begin{compactitem}

\item \textit{Monotonicity}: If there is a $G$-equivariant map $X\to Y$ then
\[
\Index{G}^r(X;R) \supseteq \Index{G}^r(Y;R) \quad \text{and} \quad  \Index{G}(X;R) \supseteq \Index{G}(Y;R).
\]

\item \textit{Additivity}: If $(X_1\cup X_2,X_1,X_2)$ is an excisive triple of $G$-spaces, then
\[
\Index{G}^{r_1}(X_1;R)\cdot\Index{G}^{r_2}(X_2;R)\subseteq\Index{G}^{r_1 + r_2}(X_1\cup X_2;R),
\]
and
\[
\Index{G}(X_1;R)\cdot\Index{G}(X_2;R)\subseteq\Index{G}(X_1\cup X_2;R).
\]
\item The \textit{General Borsuk--Ulam--Bourgin--Yang theorem:} 
If there is a $G$-equivariant map $f:X\to Y$ and a closed $G$-invariant subspace $Z\subseteq Y$ then
\[
\Index{G}^{r_1}(f^{-1}(Z);R)\cdot\Index{G}^{r_2}(Y{\setminus}Z;R)\subseteq \Index{G}^{r_1 + r_2}(X;R),
\]
and
\[
\Index{G}(f^{-1}(Z);R)\cdot\Index{G}(Y{\setminus}Z;R)\subseteq \Index{G}(X;R).
\]

\end{compactitem}

%-----------------------------------------------------------------------------------%
\subsection{Calculation of $\Index{\ZZ/p}(F(\RR^d,p);\FF_p)$}
%-----------------------------------------------------------------------------------%
The result we present in this part can also be deduced from the Vanishing Theorem of F.~Cohen~\cite[Theorem~8.2, page~268]{Cohen}.

Let us denote the cohomology of the group $\ZZ/p$ with coefficients in~$\FF_p$ for a prime $p$, as in Section~\ref{Sec:DiffSSSeq}, by
\begin{equation*}
\begin{array}{llll}
H^*(\ZZ/2;\FF_2)        & =    & \FF_2[t],                                  & \quad\deg t=1,                    \\
H^*(\ZZ/p;\FF_p)        & =    & \FF_p[t]\otimes \Lambda [e],               &\quad \deg t=2,\deg e=1\text{ and $p$ odd.}
\end{array}
\end{equation*}

\begin{theorem}
\label{Th:IndexF(X,p)}
Let $p$ be a prime and $d>1$.
Then,
\begin{equation*}
\Index{\ZZ/p}(F(\RR^d,p);\FF_p)=H^{\geq (d-1)(p-1)+1}(\ZZ/p;\FF_p)=
\left\{
\begin{array}{ll}
\langle t^{(d-1)(p-1)+1}\rangle , & \text{~for~}p=2, \\
\langle e\,t^{\tfrac{(d-1)(p-1)}{2}}, t^{\tfrac{(d-1)(p-1)}{2}+1}\rangle , & \text{~for~}p\text{~odd}.
\end{array}
\right.
\end{equation*}
\end{theorem}

\begin{proof}
  \textbf{(1)} The Equivariant Goresky--MacPherson formula,
  Theorem~\ref{Th:EqGM}~\eqref{Th:EqGM-Formula-2}, applied to the configuration
  space with acting group $\ZZ/p$ and with coefficients $\FF_p$ implies that the
  non-zero cohomology as an $\FF_p[\ZZ/p]$-module has the following description
 \[
H^{(d-1)(p-j)}(F(\RR^d,p);\FF_p)\cong \bigoplus_{\substack{\pi\in (\Pi_p{\setminus}{\hat{0}})/(\ZZ/p),\\\mathrm{size}(\pi)=j}}
\ind^{\ZZ/p}_{(\ZZ/p)_{\pi}}\tilde{H}_{-2}(\Delta(\bar{\Pi}_{1});\FF_p)^{\otimes b_1(\pi)}
\otimes\cdots\otimes\tilde{H}_{p-3}(\Delta(\bar{\Pi}_{p});\FF_p)^{\otimes b_p(\pi)}
\]
for $1\leq j \leq p-1$.

In this case the ``orientation'' representations from 
Theorem~\ref{Th:ModuleStructureOnCohomology} are trivial due to the choice of the acting
group and related coefficients.  The maximal partition
$\hat{1}=\{[p]\}$, of size $1$, is the only partition stabilized by
the whole group $\ZZ/p$.  The contribution of the maximal partition to the
cohomology of the configuration space appears in dimension $(d-1)(p-1)$.
Moreover, there is an isomorphism of $\FF_p[\ZZ/p]$-modules
\[
H^{(d-1)(p-1)}(F(\RR^d,p);\FF_p)\cong \tilde{H}_{p-3}(\Delta(\bar{\Pi}_{p});\FF_p).
\]
Since $H^i(F(\RR^d,p);\FF_p)\in\mathfrak{I}_{\ZZ/p}$ for $1\leq i\leq (d-1)(p-1)-1$, 
Theorem~\ref{Th:DiffSSSeq-EAb} can be applied to the Serre spectral sequence of the fibration
\[
F(\RR^d,p)\To \EE\ZZ/p\times_{\ZZ/p}F(\RR^d,p)\To \BB\ZZ/p.
\]
We obtain that for every $r\geq 0$ and every $s\in\{2,\ldots,(d-1)(p-1)\}$ the differential
\[
\partial_{s}:E_{s}^{r,s-1}(\EE\ZZ/p\times _{\ZZ/p}F(\RR^d,p))\longrightarrow E_{s}^{r+s,0}(\EE\ZZ/p\times_{\ZZ/p}F(\RR^d,p))
\]
vanishes. Hence $H^s(G;\FF_p) = E^s_{r,0} = E^{\infty}_{s,0}$ for $s \le (d-1)(p-1)$. This implies
\[
\Index{\ZZ/p}(F(\RR^d,p);\FF_p)\subseteq H^{\geq (d-1)(p-1)+1}(\ZZ/p;\FF_p).
\]

The $E_2$-term of this spectral sequence can be described further.
In this case the family $\mathfrak{I}_{\ZZ/p}$ is just the family of all free $\FF_p[\ZZ/p]$-modules.
Therefore, $H^{\ell}(\ZZ/p;M)=0$ for all $M\in\mathfrak{I}_{\ZZ/p}$ and $l>0$.
Consequently
\begin{equation}
\label{eq:E_2-term-zero}
E^{r,s}_2=H^r(\ZZ/p;H^s(F(\RR^d,p);\FF_p))=0,
\end{equation}
for all $r>1$ and $0<s<(d-1)(p-1)$.

\medskip

\noindent\textbf{(2)}
The configuration space $F(\RR^d,p)$ is a free $\ZZ/p$-space, so $F(\RR^d,p)/(\ZZ/p)\simeq E\ZZ/p\times_{\ZZ/p}F(\RR^d,p)$. 
The spectral sequence we consider converges to
\[
H^*(\EE\ZZ/p\times_{\ZZ/p}F(\RR^d,p);\FF_p)\cong H^*(F(\RR^d,p)/(\ZZ/p);\FF_p).
\]
Since $H^i(F(\RR^d,p)/(\ZZ/p);\FF_p)=0$ for at least $i>\dim F(\RR^d,p)=dp$
and $H^i(F(\RR^d,p);R)\neq 0$ if and only if $i=(d-1)(p-k)$ for some $k\in\{1,\ldots,p\}$ (see Theorem~\ref{Th:ModuleStructureOnCohomology}),
we know that the differential 
\[
\partial_{(d-1)(p-1)+1}:E_{(d-1)(p-1)+1}^{r,(d-1)(p-1)}(\EE \ZZ/p\times _{\ZZ/p}F(\RR^d,p))
\longrightarrow E_{(d-1)(p-1)+1}^{r+(d-1)(p-1)+1,0}(\EE \ZZ/p\times_{\ZZ/p}F(\RR^d,p))
\]
is NOT zero for some $r\geq 0$.

As we have already seen in~\eqref{eq:HomotopyTypeofPi_n}, the cohomology group
\[
H^{(d-1)(p-1)}(F(\RR^d,p);\ZZ)\cong \tilde{H}_{p-3}(\Delta(\bar{\Pi}_{p});\ZZ)\cong_{Ab} \ZZ^{(p-1)!}
\]
has no $p$-torsion.  
Therefore, a result of Allday, Hanke and Puppe in~\cite[Theorem~2, page~3276 or Proposition~4, page~3281]{A-H-P} can be applied
to our situation providing the following decomposition of $\FF_p[\ZZ/p]$-modules when $p>2$:
\begin{equation}
\label{eq:ModuleDecom-1}
H^{(d-1)(p-1)}(F(\RR^d,p);\FF_p)\cong \tilde{H}_{p-3}(\Delta(\bar{\Pi}_{p});\FF_p)\cong T \oplus F \oplus A .
\end{equation}
Here $T$ is a trivial $\FF_p[\ZZ/p]$-module, $F$ is a free $\FF_p[\ZZ/p]$-module and $A$ stands for a direct sum of kernels
of the augmentation map $K:=\ker(\epsilon \colon \FF_p[\ZZ/p]\To\FF_p)$. 

Since all differentials $\partial_{\ell}$ are $H^*(\ZZ/p;\FF_p)$-module maps then the description of the $E_2$-term given in \eqref{eq:E_2-term-zero} implies that $\partial_{\ell}=0$ for all $\ell\in\{1,\ldots ,(d-1)(p-1)\}$.
The only possible non-trivial differential is $\partial_{(d-1)(p-1)+1}$.
Consequently,
\[
E^{r,s}_2=E^{r,s}_{\ell}
\]
for all $r,s\in\ZZ$ and $\ell\in\{1,\ldots ,(d-1)(p-1)\}$.
In particular, 
\[
E^{*,(d-1)(p-1)}_{(d-1)(p-1)+1}
=
E^{*,(d-1)(p-1)}_{2}
=
H^*(\ZZ/p;T\oplus F\oplus A)
=
H^*(\ZZ/p;T)\oplus H^*(\ZZ/p;F)\oplus H^*(\ZZ/p;A).
\]
The $H^*(\ZZ/p;\FF_p)$-modules $H^*(\ZZ/p;T)$ and $H^*(\ZZ/p;F)$ have generators in dimension $0$, and $H^*(\ZZ/p;\FF_p)$-module $H^*(\ZZ/p;A)$ has generators in dimensions $0$ and $1$.
Therefore, the generators of the $E^{*,(d-1)(p-1)}_{(d-1)(p-1)+1}$-row as an $H^*(\ZZ/p;\FF_p)$-module appear in the $E^{0,(d-1)(p-1)}_{(d-1)(p-1)+1}$-term 
if $p=2$, and in the $E^{0,(d-1)(p-1)}_{(d-1)(p-1)+1}$ and $E^{1,(d-1)(p-1)}_{(d-1)(p-1)+1}$-term if $p$ is an odd prime. 

Since the $E_\infty$ term should not have any non-trivial entries above the $dp$ diagonal, then the differential $\partial_{(d-1)(p-1)+1}$ has to be non-trivial.
The differential $\partial_{(d-1)(p-1)+1}$ is completely determined by its values on the generators that appear at the position $(0,(d-1)(p-1))$ for $p=2$, and in the $(0,(d-1)(p-1))$ and $(1,(d-1)(p-1))$ positions for $p>2$.
Thus decomposition~\eqref{eq:ModuleDecom-1} for $p>2$ can be made more precise:
\begin{equation}
\label{eq:ModuleDecom-2}
H^{(d-1)(p-1)}(F(\RR^d,p);\FF_p)\cong \tilde{H}_{p-3}(\Delta(\bar{\Pi}_{p});\FF_p)
\cong \FF_p[\ZZ/p]^{\tfrac{(p-1)!-p+1}{p}}\oplus K.
\end{equation}
Indeed, having in mind that the differential $\partial_{(d-1)(p-1)+1}$ always
lands in an $\FF_p$ vector space of dimension one, we have that
\begin{compactitem}
\item if $T\neq 0$ then $E_{\infty}^{2l+1,(d-1)(p-1)}\neq 0$ for all $l>0$, a contradiction;
\item if $A=K^{\oplus a}$ and $a>1$, then $E_{\infty}^{l,(d-1)(p-1)}\neq 0$ for all $l\geq 0$, a contradiction;
\item if $A=0$ and $T = 0$, then $E_{\infty}^{l,0}\neq 0$ for all $l\geq 0$, also a contradiction.
\end{compactitem}
Finally, when $p=2$ the result is obtained directly from the fact that 
$\tilde{H}_{-1}(\Delta(\bar{\Pi}_{2});\FF_2)=\tilde{H}_{-1}(\emptyset;\FF_2)\cong\FF_2$ is a trivial $\FF_2[\ZZ/2]$-module.

\medskip

\noindent\textbf{(3)} Let $p>2$ be an odd prime and $\varepsilon$ denotes a generator of the group $\ZZ/p$.
The exact sequence of $\FF_p[\ZZ/p]$-modules
\[
0\To\FF_p\xrightarrow{1+\varepsilon+\cdots +\varepsilon^{p-1}}\FF_p[\ZZ/p]\To K\To 0
\]
induces a long exact sequence in cohomology of the group $\ZZ/p$,~\cite[Proposition~6.1, pages~71-72]{Brown}.
This exact sequence yields the following  description of $H^*(\ZZ/p;K)$ as an $H^*(\ZZ/p;\FF_p)$-module
\[
H^*(\ZZ/p;K)\cong H^{*+1}(\ZZ/p;\FF_p),
\]
for $*\geq 0$.
Thus as an $H^*(\ZZ/p;\FF_p)$-module $H^*(\ZZ/p;K)$ is generated 
by two elements $a\in H^0(\ZZ/p;K)$ and $b\in H^1(\ZZ/p;K)$ such that
\begin{compactitem}
\item $a$ spans $H^0(\ZZ/p;K)\cong\FF_p$ as a vector space over $\FF_p$,
\item $b$ spans $H^1(\ZZ/p;K)\cong\FF_p$ as a vector space over $\FF_p$,
\item $t^i\cdot a$ spans $H^{2i}(\ZZ/p;K)$ as a vector space over $\FF_p$,
\item $t^i\cdot b$ spans $H^{2i+1}(\ZZ/p;K)$ as a vector space over $\FF_p$,
\item $e\cdot a=0\in H^1(\ZZ/p;K)$ and $t^i\cdot a=e\,t^{i-1}\cdot b\in H^{2i}(\ZZ/p;K)$.
\end{compactitem}
Here ``$\cdot$'' denotes the action of $H^*(\ZZ/p;\FF_p)$.
Now the index of the configuration space is generated as an 
ideal by the $\partial_{(d-1)(p-1)+1}$ images of elements $a$ and $b$, i.e.,
\[
\Index{\ZZ/p}(F(\RR^d,p);\FF_p)=\langle\partial_{(d-1)(p-1)+1}(a),
\partial_{(d-1)(p-1)+1}(b)\rangle=\langle e\,t^{\tfrac{(d-1)(p-1)}{2}}, t^{\tfrac{(d-1)(p-1)}{2}+1}\rangle.
\]
A similar, but simpler, argument implies the result for $p=2$.
\end{proof}

An important consequence of the previous index calculation is a description of
the top homology of the partition lattice $\Pi_p$ as an $\FF_p[\ZZ_p]$-module.

\begin{corollary}
\label{cor:RepresPi_p}
Let $p$ be an odd prime.
Then
\[
\tilde{H}_{p-3}(\Delta(\bar{\Pi}_{p});\FF_p)\cong \FF_p[\ZZ/p]^{\tfrac{(p-1)!-p+1}{p}}\oplus K.
\]
\end{corollary}

%-----------------------------------------------------------------------------------%
\subsection{Estimate of $\Index{(\ZZ/p)^k}(F(\RR^d,p^k);\FF_p)$, $k>1$}
\label{Sec:Estiamtion-1}
%-----------------------------------------------------------------------------------%
Using the notation of Section~\ref{Sec:DiffSSSeq} we have that
\begin{equation}
\label{eq:PresGroupCoho}
\begin{array}{llll}
H^*(\left( \ZZ/2\right) ^k;\FF_2)        & =    & \FF_2[t_1,\ldots,t_k],                                & \deg t_j=1                     \\
H^*(\left( \ZZ/p\right) ^k;\FF_p)        & =    & \FF_p[t_1,\ldots,t_k]\otimes \Lambda [e_1,\ldots,e_k], & \deg t_j=2,\deg e_i=1\text{ and $p$ odd.}
\end{array}
\end{equation}

Recall that we consider the action of the group $(\ZZ/p)^k$ on the configuration
space $F(\RR^d,p^k)$ via the regular embedding
$\mathrm{(reg)}:(\ZZ/p)^k\to\Sym_{p^k}$ as described in~\cite[Example~2.7, page~100]{Adem-Milgram}.
The regular embedding is given by the left translation action of $(\ZZ/p)^k$ on itself. 
To each element $g\in (\ZZ/p)^k$ we associate permutation $L_g\colon  (\ZZ/p)^k\to  (\ZZ/p)^k$ from $\mathrm{Sym}((\ZZ/p)^k)\cong\Sym_{p^k}$ given by $L_g(x)=g+x$.

We prove the following estimate.
(Note that the following theorem and its proof are also valid for $k=1$, but in this case Theorem \ref{Th:IndexF(X,p)}  is a sharper result.)

\begin{theorem}
\label{Th:Estimate-IndexF(X,p_upper_k)}
Let $p$ be a prime, $d>1$ and $k>1$.
Then,
\begin{equation*}
\Index{(\ZZ/p)^k}(F(\RR^d,p^k);\FF_p)\subseteq H^{\geq (d-1)(p^k-p^{k-1})+1}((\ZZ/p)^k;\FF_p).
\end{equation*}
\end{theorem}

\begin{proof}
  Again, the Equivariant Goresky--MacPherson formula,
  Theorem~\ref{Th:EqGM}~\eqref{Th:EqGM-Formula-2}, now applied for the group
  $(\ZZ/p)^k$ and with coefficients $\FF_p$, implies that the positive non-zero
  cohomology of the configuration space as an $\FF_p[(\ZZ/p)^k]$-module can be
  described as
\begin{multline*}
H^{(d-1)(p^k-j)}(F(\RR^d,p^k);\FF_p)\cong 
\\
\bigoplus_{\substack{\pi\in (\Pi_{p^k}{\setminus}{\hat{0}})/(\ZZ/p)^k \\\mathrm{size}(\pi)=j}}
\ind^{(\ZZ/p)^k}_{((\ZZ/p)^k)_{\pi}}\tilde{H}_{-2}(\Delta(\bar{\Pi}_{1});\FF_p)^{\otimes b_1(\pi)}\otimes
 \cdots\otimes \tilde{H}_{p^k-3}(\Delta(\bar{\Pi}_{p^k});\FF_p)^{\otimes b_{p^k}(\pi)}
\end{multline*}
where $1\leq j\leq p^k-1$.
Again, the ``orientation'' representations from 
Theorem~\ref{Th:ModuleStructureOnCohomology} are trivial due to the choice of the acting
group and related coefficients.

\medskip

\noindent The key observation is that
\[
\mathrm{size}(\pi)>p^{k-1} \Longrightarrow  ((\ZZ/p)^k)_{\pi}\neq (\ZZ/p)^k.
\]
Actually more is true, consult Section~\ref{subsec:The_special_case_G_is_(Z/p)k}:
\[
\mathrm{size}(\pi)\text{ not a power of }p \Longrightarrow   ((\ZZ/p)^k)_{\pi}\neq (\ZZ/p)^k.
\]
Therefore, $H^i(F(\RR^d;p^k);\FF_p)\in\mathfrak{I}_{(\ZZ/p)^k}$ for all $1\leq i\leq (d-1)(p^k-p^{k-1})-1$.
As before, Theorem~\ref{Th:DiffSSSeq-EAb} implies that in the Serre spectral sequence of the fibration
\[
F(\RR^d;p^k)\To \EE(\ZZ/p)^k\times_{(\ZZ/p)^k}F(\RR^d,p^k)\To \BB(\ZZ/p)^k
\]
the differentials
\[
\partial_{s}:E_{s}^{r,s-1}(\EE(\ZZ/p)^k\times _{(\ZZ/p)^k}F(\RR^d,p^k))
\longrightarrow E_{s}^{r+s,0}(\EE(\ZZ/p)^k\times_{(\ZZ/p)^k}F(\RR^d,p^k))
\]
vanish for all $s\in\{2,\ldots,(d-1)(p^k-p^{k-1})\}$ and all $r \ge 0$.
Therefore,
\[
E_{2}^{r,0}(\EE(\ZZ/p)^k\times _{(\ZZ/p)^k}F(\RR^d,p^k))=
E_{\infty}^{r,0}(\EE(\ZZ/p)^k\times _{(\ZZ/p)^k}F(\RR^d,p^k))\cong
H^{r}((\ZZ/p)^k;\FF_p)
\]
for all $r\in\{0,\ldots,(d-1)(p^k-p^{k-1})\}$.
Consequently the claim of theorem follows:
\[
\Index{(\ZZ/p)^k}(F(\RR^d,p^k);\FF_p)\subseteq H^{\geq (d-1)(p^k-p^{k-1})+1}((\ZZ/p)^k;\FF_p).
\]
\end{proof}

\noindent In terms of the partial Fadell--Husseini indexes we just proved that
\[
\Index{(\ZZ/p)^k}^{1}(F(\RR^d,p^k);\FF_p)=\cdots =\Index{(\ZZ/p)^k}^{(d-1)(p^k-p^{k-1})+1}(F(\RR^d,p^k);\FF_p)=\{0\}.
\]

The estimate obtained in the previous theorem is an extension of  Cohen's results in~\cite{Cohen}.

%%%%%%%%%%%%%%%%%%%%%%%%%%%%%%%%%%%%%%%%%%%%%%%%%%%%%%%%%%%%%%%%%%%%%%%%%%%%%%%%%%%%%
%%%%%%%%%%%%%%%%%%%%%%%%%%%%%%%%%%%%%%%%%%%%%%%%%%%%%%%%%%%%%%%%%%%%%%%%%%%%%%%%%%%%%
%-----------------------------------------------------------------------------------%
\section{Fadell--Husseini index, II}
\label{Sec:FH-Index-II}
%-----------------------------------------------------------------------------------%
%%%%%%%%%%%%%%%%%%%%%%%%%%%%%%%%%%%%%%%%%%%%%%%%%%%%%%%%%%%%%%%%%%%%%%%%%%%%%%%%%%%%%
%%%%%%%%%%%%%%%%%%%%%%%%%%%%%%%%%%%%%%%%%%%%%%%%%%%%%%%%%%%%%%%%%%%%%%%%%%%%%%%%%%%%%

Let  $G:=(\ZZ/p)^k$ and $N:=(d-1)(p^k-p^{k-1})$.
In this section we study the partial index of the configuration space
\[
\Index{G}^{N+2}(F(\RR^d,p^k);\FF_p)=\ker (H^*(\BB G;\FF_p)\to E^{*,0}_{N+2}(EG\times_{G}F(\RR^d,p^k))).
\]
In particular we prove the following theorem.

\begin{theorem}
\label{Th:IndexN_plus_3}
Let  $G:=(\ZZ/p)^k$ and $N:=(d-1)(p^k-p^{k-1})$. Then
\begin{compactenum}[\rm (i)]
\item $\Index{G}^{N+2}(F(\RR^d,p^k);\FF_p)\neq \{0\}$,
\item $\Index{G}^{N+2}(F(\RR^d,p^k);\FF_p)\cap H^{N+1}(\BB G;\FF_p)\neq \{0\}$, i.e., the homomorphism
\[
 H^{N+1}(\BB G;\FF_p)\to E^{N+1,0}_{N+2}(EG\times_{G}F(\RR^d,p^k))
\]
is not injective.
\end{compactenum}
\end{theorem}

\subsection{$\Index{G}(F(\RR^d,p^k);\FF_p)$ as a $\mathrm{GL}_k(\ZZ/p)$-invariant ideal}

Some properties of the Fadell--Husseini index of the configuration space are proved in this section.
They will not be used in the proof of Theorem \ref{Th:IndexN_plus_3}, but may be interesting and useful for other situations.

\begin{lemma} 
\label{lem:Fadell--Husseini_index_autos}
Let $G$ be any finite group, $X$ a $G$-space and $R$ a commutative ring with
unit.  Consider a group automorphism $\varphi \colon G \to G$ such that there
exists a $\varphi$-equivariant self homotopy equivalence $f \colon X \to X$,
i.e., for every $g\in G$ and $x\in X$
\[
f(g\cdot x)=\varphi(g)\cdot x.
\]
Then the induced map $\varphi^* \colon H^*(\BB G;R) \xrightarrow{\cong} H^*(\BB G;R)$
respects the Fadell--Husseini index and the partial Fadell--Husseini indices, that is,
\[
\varphi^*\bigl(\Index{G}(X;R)\bigr) = \Index{G}(X;R) \quad \text{and} \quad \varphi^*\bigl(\Index{G}^r(X;R)\bigr) = \Index{G}^r(X;R),
\]
for all $r\geq 2$.
\end{lemma}
\begin{proof}
  The group automorphism $\varphi$ induces a $\varphi$-equivariant map 
  $\Phi   \colon \EE G \to \EE G$ which is unique up to a $\varphi$-equivariant
  homotopy.  This follows from the universal property of $\EE G$ since both
  spaces, $\EE G$ and $\varphi^* \EE G$, are models for $\EE G$, where
  $\varphi^* \EE G$ is obtained from $\EE G$ by twisting the group action
  with $\varphi$.  We obtain the following commutative diagram where the
  horizontal maps are $\varphi$-equivariant homotopy equivalences and the
  vertical arrows are the canonical projections:
\[
\label{eq:Diagram-1}
\xymatrix{
\EE G \times X \ar[r]^{\Phi\times f} \ar[d]  & \EE G \times X  \ar[d] \\
\EE G \ar[r]_{\Phi}                          & \EE G
         }
\]
Dividing out the $G$-actions yields the commutative diagram:
\begin{equation}
\label{eq:Diagram-2}
\xymatrix{\EE G \times_G X \ar[r]^{\overline{\Phi \times f}} \ar[d]_{\pr}^{\simeq}  & \EE G \times_G X  \ar[d]^{\pr}_{\simeq} \\
\BB G \ar[r]_{\overline{\Phi}}                                            & \BB G
         }
\end{equation}
where the vertical maps are the canonical projections and the horizontal maps are homotopy equivalences. 
Applying cohomology yields a commutative diagram whose vertical maps are the same and whose horizontal arrows are isomorphisms (and thus bijections):
\[
\xymatrix@!C=12em{
H^*(\EE G \times_G X;R)           & H^*(\EE G \times_G X;R) \ar[l]^{\cong}_{H^*(\overline{\Phi \times f};R)} \\
H^*(\BB G;R)  \ar[u]^{H^*(\pr;R)} & H^*(\BB G;R)\ar[u]_{H^*(\pr;R)} \ar[l]^{\cong}_{H^*(\overline{\Phi};R) = \varphi^*} 
}
\]
This implies that
\[
\varphi^*\bigl(\Index{G}(X;R)\bigr)
 = 
\varphi^*\bigl(\ker(H^*(\pr;R))\bigr)
=
\ker(H^*(\pr;R))
 = 
 \Index{G}(X;R).
\]

The commutative diagram~\eqref{eq:Diagram-2} defines an automorphism of the
Borel construction fibration $X\to \EE G\times_{G}X\to \BB G$.  Consequently it
induces an automorphism of the associated Serre spectral sequence such that the
following diagram commutes:
\[
\xymatrix@!C=12em{
E^{*,0}_{r}(\EE G \times_G X;R)           & E^{*,0}_{r}(\EE G \times_G X;R)    \ar[l]_{E^{*,0}_{r}(\overline{\Phi\times f};R)}       \\
E^{*,0}_{2}(\EE G \times_G X;R)\ar[u]     & E^{*,0}_{2}(\EE G \times_G X;R)    \ar[l]_{E^{*,0}_{2}(\overline{\Phi\times f};R)} \ar[u] \\
H^*(\BB G;R)  \ar[u]                      & H^*(\BB G;R)\ar[u]                 \ar[l]^{\cong}_{H^*(\overline{\Phi};R) = \varphi^*} 
                 }
\]
Therefore,
\[
\varphi^*\bigl(\Index{G}^r(X;R)\bigr) = \Index{G}^r(X;R).
\]
\end{proof}

Recall that the group $G:=(\ZZ/p)^k$ acts on the configuration space $F(\RR^d,p^k)$ via the regular embedding $\mathrm{(reg)}:G\to\Sym_{p^k}$.  
The normalizer of the group $G$ in $\Sym_{p^k}$ is the semi-direct product $N(G)=G\rtimes \mathrm{Aut}(G)\cong G\rtimes\mathrm{GL}_k(\ZZ/p)$, consult~\cite[Example~2.7, page~100]{Adem-Milgram}.  
The group $N(G)$ acts on $G$ via conjugation, $\eta\cdot g:=\eta g \eta^{-1}$, for $\eta\in N(G)$ and $g\in G$.
The first factor in the semi-direct product $N(G)=G\rtimes \mathrm{Aut}(G)$ acts trivially on $G$ and the $N(G)$-action on $G$ factorizes over the projection 
$N(G) \to \mathrm{GL}_k(\ZZ/p)$ to the standard $\mathrm{GL}_k(\ZZ/p)$-action on $G$.  
This action defines an action of $\mathrm{Aut}(G)\cong\mathrm{GL}_k(\ZZ/p)$ on the group cohomology $H^*(G;\FF_p)$ which under the projection above corresponds to the obvious $\mathrm{GL}_k(\ZZ/p)$-action.
\newpage

\begin{theorem}
Let $G:=(\ZZ/p)^k$ and $k>1$. 
Then
\begin{compactenum}[\rm (1)]
\item  $\Index{G}(F(\RR^d,p^k);\FF_p)$ is a $\mathrm{GL}_k(\ZZ/p)$-invariant ideal in $H^*(G;\FF_p)$,
\item  $\Index{G}^r(F(\RR^d,p^k);\FF_p)$ is a $\mathrm{GL}_k(\ZZ/p)$-invariant ideal in $H^*(G;\FF_p)$, for any $r\geq 2$.
\end{compactenum}
\end{theorem}
\begin{proof}
  Let $\eta\in \mathrm{Aut}(G)\subset N(G)\subset \Sym_{p^k}$ be a permutation
  and $\varphi_{\eta}\colon G\to G$ the automorphism $\varphi_{\eta}(g):=\eta g
  \eta^{-1}$.  It defines a $\varphi_{\eta}$-equivariant selfhomotopy
  equivalence $F(\RR^d,p^k) \to F(\RR^d,p^k)$ as the restriction of the
  homomorphism $f\colon (\RR^d)^{p^k}\to (\RR^d)^{p^k}$ given by
\begin{equation}
\label{eq:SelfEq}
f(x_1,\ldots ,x_{p^k}):=(x_{\eta(1)},\ldots ,x_{\eta(p^k)}).
\end{equation}
Indeed, for $g\in G\subset \Sym_{p^k}$ we have that
\[
\begin{array}{lllll}
f(g\cdot (x_1,\ldots, x_{p^k}))  & = & f(x_{g(1)},\ldots, x_{g(p^k)}) 
                                & = & (x_{\eta g(1)},\ldots, x_{\eta g(p^k)})\\
                                & = & (x_{\eta g\eta^{-1} \eta(1)},\ldots, x_{\eta g \eta^{-1} \eta(p^k)})
                                & = & \eta g\eta^{-1}\cdot (x_{\eta(1)},\ldots, x_{\eta(p^k)})\\
                                & = & \varphi_{\eta}(g)\cdot (x_{\eta(1)},\ldots, x_{\eta(p^k)})
                                & = & \varphi_{\eta}(g)\cdot f(x_1,\ldots, x_{p^k}).
\end{array}
\]
Thus, according to Lemma~\ref{lem:Fadell--Husseini_index_autos}, the indexes are $\mathrm{GL}_k(\ZZ/p)$-invariant ideals.
\end{proof}

%-----------------------------------------------------------------------------------%
\subsection{Proof of Theorem \ref{Th:IndexN_plus_3}}
\label{Sec:E-page}
%-----------------------------------------------------------------------------------%

In Theorem~\ref{Th:Estimate-IndexF(X,p_upper_k)} we estimated the partial index, up to the page $E_N$, by detecting the differentials that land in the $0$-row of
the Serre spectral sequence of the fibration
\[
F(\RR^d,p^k)\To \EE G\times_{G}F(\RR^d,p^k)\To \BB G.
\]
Thus, the next step is to consider the  $E_{N+1}$-page of the spectral sequence and to understand
\begin{compactitem}
\item the cohomology $H^N(F(\RR^d,p^k);\FF_p)$ as an $\FF_p[G]$-module,
\item the cohomology $H^*(G;H^N(F(\RR^d,p^k);\FF_p))=E_2^{*,N}$ as an $H^*(G;\FF_p)$-module 
      by identifying the $H^*(G;\FF_p)$-module generators, and finally 
\item the differential $\partial_{N+1}:E_{N+1}^{\ell,N}\To E_{N+1}^{0,N+\ell +1}$ on at least one of the previously identified $H^*(G;\FF_p)$-module generators
      in the $N$-th row of the $(N+1)$st page,  $E_{N+1}^{\ell,N}$, ($l\geq 0$).
\end{compactitem}
In this way we will show that the homomorphism
\[
 H^{N+1}(\BB G;\FF_p)\to E^{N+1,0}_{N+2}(EG\times_{G}F(\RR^d,p^k))
\]
is not injective. 

\medskip

%-----------------------------------------------------------------------------------%
\subsubsection{~}
\label{Sec:Sub-1}
%\noindent\textbf{1.}
The Equivariant Goresky--MacPherson formula, Theorem~\ref{Th:EqGM}~\eqref{Th:EqGM-Formula-2}, gives the description of the $N$-th cohomology of the configuration spaces $F(\RR^d,p^k)$ as an $\FF_p[G]$-module:
\begin{equation}
\label{eq:Decomp-1}
H^{N}(F(\RR^d,p^k);\FF_p)\cong
\bigoplus_{\substack{\pi\in (\Pi_{p^k}{\setminus}{\hat{0}})/G \\ \mathrm{size}(\pi)=p^{k-1}}}
\ind^{G}_{G_{\pi}}\tilde{H}_{-2}(\Delta(\bar{\Pi}_{1});\FF_p)^{\otimes b_1(\pi)}\otimes
\cdots\otimes \tilde{H}_{p^k-3}(\Delta(\bar{\Pi}_{p^k});\FF_p)^{\otimes b_{p^k}(\pi)}.
\end{equation}

In Section~\ref{subsec:The_special_case_G_is_(Z/p)k} we pointed out that for all partitions $\pi$ of $[p^k]$ of size $p^{k-1}$ the following equivalence holds:
\[
G_{\pi}=G \Longleftrightarrow \text{There exists a subgroup}\; H \subseteq G\; \text{of order p with}\; \pi=\pi_H.
\]
Therefore, according to the isomorphism~\eqref{eq:TensorInduced} the
decomposition~\eqref{eq:Decomp-1} has the following additional
splitting
\begin{eqnarray}
\label{eq:Decomp-2}
H^{N}(F(\RR^d,p^k);\FF_p)
&\cong&
\underset{H<G~:~|H|=p}{\bigoplus} \tilde{H}_{p-3}(\Delta(\bar{\Pi}_H);\FF_p)^{\otimes G/H}
\oplus
\underset{H_{\lambda}\neq G}{\bigoplus}\mathrm{Ind}^G_{H_{\lambda}}M_{\lambda}\nonumber\\
&\cong&
\underset{H<G~:~|H|=p}{\bigoplus} \tilde{H}_{p-3}(\Delta(\bar{\Pi}_{p});\FF_p)^{\otimes G/H}
\oplus
\underset{H_{\lambda}\neq G}{\bigoplus}\mathrm{Ind}^G_{H_{\lambda}}M_{\lambda}
\end{eqnarray}
for some $\FF_p[H_{\lambda}]$-modules $M_{\lambda}$.  Since the second
sum is over all proper subgroups of $G$, it might happen that some of
the modules $M_{\lambda}$ are zero.

Now the $E_2^{*,N}$-row of the spectral sequence as an $H^*(G;\FF_p)$-module decomposes as follows:
\begin{eqnarray}
E_2^{*,N}   &     = & H^*(G;H^N(F(\RR^d,p^k);\FF_p))\nonumber                    \\
            & \cong &
\underset{H<G~:~|H|=p}{\bigoplus}H^*\big(G;\tilde{H}_{p-3}(\Delta(\bar{\Pi}_H);\FF_p)^{\otimes G/H}\big)
\oplus\underset{H_{\lambda}\neq G}{\bigoplus}H^*(G;\mathrm{Ind}^{G}_{H_{\lambda}}M_{\lambda})\nonumber\\
\label{eq:Decomp-3-2}
            & \cong &
\underset{H<G~:~|H|=p}{\bigoplus}H^*\big(G;\tilde{H}_{p-3}(\Delta(\bar{\Pi}_{p});\FF_p)^{\otimes G/H}\big)
\oplus\underset{H_{\lambda}\neq G}{\bigoplus}H^*(G;\mathrm{Ind}^{G}_{H_{\lambda}}M_{\lambda}).
\end{eqnarray}
It is important that the $E_{N+1}^{*,N}$-row is a
sub-quotient of the $E_2^{*,N}$-row and that the $H^*(G;\FF_p)$-module
structure on $E_{N+1}^{*,N}$-row comes from the one on the
$E_2^{*,N}$-row.  Using the same argument as in the proof of
Theorem~\ref{Th:DiffSSSeq-EAb}, we conclude that the $\partial_{N+1}$
differential vanishes on the second summand of the
decomposition~\eqref{eq:Decomp-3-2}.

Therefore we need to understand the $H^*(G;\FF_p)$-module structure on the first summand.
Since the first summand is a direct sum over all subgroups $H$ of $G$ of order
$p$, it suffices to describe the generators of the $H^*(G;\FF_p)$-module
\[
H^*\big(G;\tilde{H}_{p-3}(\Delta(\bar{\Pi}_H);\FF_p)^{\otimes G/H}\big)
\cong
H^*(G;\tilde{H}_{p-3}(\Delta(\bar{\Pi}_{p});\FF_p)^{\otimes G/H})
\]
for a particular subgroup $H$.
The subgroup $H\cong\ZZ/p$ acts on $\Pi_H\cong_{\rm{poset}}\Pi_p$ by the multiplication of $H$ on the ground
set of the partition lattice $H$, i.e., by the cyclic shift when $H$
is identified with $\ZZ/p$.  Recall that
Corollary~\ref{cor:RepresPi_p} describes the homology of the partition
lattice as an $\FF_p[H]=\FF_p[\ZZ/p]$-module:
\[
\tilde{H}_{p-3}(\Delta(\bar{\Pi}_{p});\FF_p)\cong \FF_p[\ZZ/p]^{\tfrac{(p-1)!-p+1}{p}}\oplus K
\]
where $K$ is the kernel of the augmentation map.
To simplify the notation we denote the $\FF_p[\ZZ/p]$-free part by $F:=\FF_p[\ZZ/p]^{\tfrac{(p-1)!-p+1}{p}}$.

\medskip

%-----------------------------------------------------------------------------------%
\subsubsection{~}
\label{Sec:Sub-2}
Now we start with the description of
\[
H^*\big(G;\tilde{H}_{p-3}(\Delta(\bar{\Pi}_H);\FF_p)^{\otimes G/H}\big)
\cong
H^*\big(G;\tilde{H}_{p-3}(\Delta(\bar{\Pi}_{p});\FF_p)^{\otimes G/H}\big)
=
H^*\big(G;(F\oplus K)^{\otimes G/H}\big)
\]
as an $H^*(G;\FF_p)$-module.

The main tool is the graded isomorphism due to Nakaoka~\cite{Nakaoka} and Leary~\cite[Theorem~2.1, page~192]{Leary}
\begin{equation}
\label{eq:Isom-tensor-induced}
H^*\big(G;M^{\otimes G/H}\big)\cong H^*\big(G/H;H^*(H,M)^{\otimes G/H}\big)
\end{equation}
where $M$ is any $\FF_p[H]$-module.
Before proceeding further let us explain how the graded isomorphism~\eqref{eq:Isom-tensor-induced}
is a consequence of~\cite[Theorem~2.1, page~192]{Leary}.
If we take:
\[
R=\FF_p~,~X=BH~,~E=E(G/H)~,~\Omega=G/H~,~S=G/H,
\]
then the graded isomorphism of~\cite[Theorem~2.1, page~192]{Leary}
reads off as the isomorphism~\eqref{eq:Isom-tensor-induced}.  This
isomorphism is obtained from the $E_2$ collapsing
Lyndon--Hochschild--Serre (LHS) spectral sequence of the exact sequence
of groups $1\longrightarrow H\longrightarrow G\longrightarrow G/H\longrightarrow 1$ with coefficients in the
$\FF_p[G]$-module $M^{\otimes G/H}$.

The action of
\[
H^*(G;\FF_p) \cong H^*(G;\FF_p^{\otimes G/H}) \cong H^*(G/H;H^*(H,\FF_p)^{\otimes G/H})\cong H^*(G/H;H^*(H,\FF_p))
\]
on
\[
H^*(G;M^{\otimes G/H})\cong H^*(G/H;H^*(H,M)^{\otimes G/H})
\]
can be obtained from the action of the LHS spectral sequence with
coefficients in $\FF_p$ on the same LHS spectral sequence but now with
$M^{\otimes G/H}$ coefficients.

\medskip

Consider the LHS spectral sequence of the exact sequence of groups $1\longrightarrow H\longrightarrow G\longrightarrow G/H\to 1$ with coefficients in $\FF_p\cong \FF_p^{\otimes G/H}$.
The $E_2$-term and also $E_{\infty}$-term are given by
\[
E_{\infty}^{i,j}=E_2^{i,j}=H^i(G/H;H^j(H;\FF_p))\cong H^i(G/H;\FF_p)\otimes_{\FF_p} H^j(H;\FF_p)
\]
since $H^j(H;\FF_p)$ is a trivial $\FF_p[G/H]$-module and the exact sequence of groups that induces the LHS spectral sequence splits.

Let us make a choice of the generators of the cohomology in the following way:
\begin{compactitem}
\item $H^*(G/H;\FF_p)=\FF_p[t_2,\ldots ,t_k]\otimes\Lambda[e_2,\ldots ,e_k]$, $\deg(e_i)=1$ and $\deg(t_i)=2$ for $p>2$,
\item $H^*(G/H;\FF_p)=\FF_p[t_2,\ldots ,t_k]$, $\deg(t_i)=1$ for $p=2$,
\item $H^*(H;\FF_p)=\FF_p[t_1]\otimes\Lambda[e_1]$, $\deg(e_1)=1$ and $\deg(t_1)=2$ for $p>2$, and
\item $H^*(H;\FF_p)=\FF_p[t_1]$, $\deg(t_1)=1$ for $p=2$.
\end{compactitem}
The convergence of the spectral sequence $E_{\infty}^{*,*}\Longrightarrow H^*(G;\FF_p)$ gives us the following presentation of the group cohomology for $p>2$:
\begin{equation}
\label{eq:presentation}
H^*(G;\FF_p)\cong\big(\FF_p[t_2,\ldots ,t_k]\otimes\Lambda[e_2,\ldots ,e_k]\big)\otimes\big(\FF_p[t_1]\otimes\Lambda[e_1]\big).
\end{equation}
Since our choice of generators depends on the subgroup $H$ let us additionally introduce the notation
\[
e_H:=e_1\qquad\text{and}\qquad t_H:=t_1
\]
when $p>2$, and $t_H:=t_1$ for $p=2$.

The $E_2=E_{\infty}$-term of this LHS spectral sequence for $p$ an odd prime that converges to $H^*(G;\FF_p)$ is depicted in Figure~\ref{Figure-1}.
 
%-----------------------------------------------------------------------------------%
\begin{figure}[h!]
\centering
\begin{tikzpicture}
\def\hlengthA{0cm} % excess length of the normal hlines
\def\hlengthB{0.3cm} % excess length of the arrow-hline
\def\vlengthA{0cm} % excess length of the normal vlines
\def\vlengthB{0.3cm} % excess length of the arrow-vline
\matrix (A) [matrix of math nodes, nodes in empty cells] 
{ 3 & 1 \otimes e_1 t_1 & \dots & \dots \\
  2 & 1 \otimes t_1 \phantom{e_1} & e_2 \otimes t_1, \dots, e_k \otimes t_1 & t_2 \otimes t_1, \dots, t_k \otimes t_1 \\
  1 & 1 \otimes e_1 \phantom{t_1} & e_2 \otimes e_1, \dots, e_k \otimes e_1 & t_2 \otimes e_1, \dots, t_k \otimes e_1 \\
  0 & 1 & e_2, \dots, e_k & t_2, \dots, t_k \\
    & 0 & 1 & 2 \\
};
\tikzhline{A}{1}{1}{4}{shorten >=-\hlengthA};
\tikzhline{A}{2}{1}{4}{shorten >=-\hlengthA};
\tikzhline{A}{3}{1}{4}{shorten >=-\hlengthA};
\tikzhline{A}{4}{1}{4}{-triangle 45, line width=1 pt,shorten >=-\hlengthB};
\tikzvline{A}{1}{5}{1}{triangle 45-, line width=1 pt,shorten <=-\vlengthB};
\tikzvline{A}{1}{5}{2}{shorten >=-\vlengthA};
\tikzvline{A}{1}{5}{3}{shorten >=-\vlengthA};
%\tikzvline{A}{1}{5}{4}{shorten >=-\vlengthA};
\end{tikzpicture}
\caption{$E_2=E_{\infty}$-term of the LHS spectral sequence with $\FF_p$ coefficients for $p>2$}
\label{Figure-1}
\end{figure}
%-----------------------------------------------------------------------------------%

Next consider the LHS spectral sequence of the exact sequence $1\longrightarrow H\longrightarrow G\longrightarrow G/H\longrightarrow 1$, but now with coefficients in $(F\oplus K)^{\otimes G/H}$.
Again by~\cite[Theorem~2.1, page~192]{Leary}, the $E_2=E_{\infty}$-term in this case is
\[
E_{\infty}^{i,j}=E_2^{i,j}=H^i\big(G/H;H^j(H;F\oplus K)^{\otimes G/H}\big)=
H^i\big( G/H;(H^j(H;F)\oplus H^j(H;K))^{\otimes G/H}\big).
\]
According to~\cite[Proposition~3.15.2.(iii), page~97]{Benson}, the $\FF_p[G/H]$-module
of coefficients decomposes in the following way:
\begin{eqnarray*}
(H^j(H;F)\oplus H^j(H;K))^{\otimes G/H}  &   =   &
\left\{
\begin{array}{ll}
(H^0(H;F)\oplus H^0(H;K))^{\otimes G/H},   & \text{for }j=0, \\
H^j(H;K)^{\otimes G/H},                    & \text{for }j>0,
\end{array}
\right.\\
                                         & \cong &
\left\{
\begin{array}{ll}
H^0(H;F)^{\otimes G/H}\oplus H^0(H;K)^{\otimes G/H}\oplus B,   & \text{for }j=0, \\
H^j(H;K)^{\otimes G/H},                                        & \text{for }j>0,
\end{array}
\right.
\end{eqnarray*}
where $B$ is a direct sum of $\FF_p[G/H]$-modules induced by the inclusion $L/H \to G/H$ for
proper subgroups $L$ of $G$ with  $H<L<G$.
Here we use the fact that $F$ is a free $\FF_p[H]$-module and so $H^j(H;F)=0$ for all $j>0$.

Moreover, since $H\cong\ZZ/p=\langle\varepsilon\rangle$, $F=\FF_p[\ZZ/p]^m$
where $m:=\tfrac{(p-1)!-p+1}{p}$, and 
\[
K=\ker(\FF_p[\ZZ/p]\overset{\epsilon}{\To}\FF_p)
\cong \FF_p[\ZZ/p]/{(1+\varepsilon+\cdots+\varepsilon^{p-1})\FF_p}
\]
we have that
\begin{eqnarray}
(H^j(H;F)\oplus H^j(H;K))^{\otimes G/H}
&\cong&
\left\lbrace
\begin{array}{ll}
(\FF_p^m)^{\otimes G/H}\oplus\FF_p^{\otimes G/H}\oplus B,  & \text{for }j=0,   \\
\FF_p^{\otimes G/H},                                       & \text{for }j>0,
\end{array}
\right.\nonumber\\
&\cong&
\left\lbrace
\begin{array}{ll}
(\FF_p^m)^{\otimes G/H}\oplus\FF_p\oplus B,  & \text{for }j=0,   \\
\FF_p,                                       & \text{for }j>0.
\end{array}
\right.\nonumber
\end{eqnarray}
Consequently
\[
E_{\infty}^{i,j}=E_2^{i,j}=
\left\lbrace
\begin{array}{ll}
H^i(G/H;(\FF_p^m)^{\otimes G/H})\oplus H^i(G/H;\FF_p)\oplus H^i(G/H;B),     & \text{for }j=0,   \\
H^i(G/H;\FF_p),                                                             & \text{for }j>0.
\end{array}
\right.
\]
The LHS spectral sequence with $(F\oplus K)^{\otimes G/H}$ coefficients for $p>2$
is given in Figure~\ref{Figure-2} in such a way that the action of the LHS spectral
sequence with $\FF_p$ coefficients is apparent.

%%-----------------------------------------------------------------------------------%
%\begin{figure}[tbh]
%\centering
%\includegraphics[scale=0.75]{figure-2}
%\caption{$E_2=E_{\infty}$-term of the LHS spectral sequence with $(F\oplus K)^{\otimes G/H}$ coefficients for $p>2$}
%\label{Figure-2}
%\end{figure}
%%-----------------------------------------------------------------------------------%

%-----------------------------------------------------------------------------------%
\begin{figure}[tbh]
\centering
\begin{tikzpicture}%[nodes={rectangle,draw}]
\def\hlengthA{0.5cm} % excess length of the normal hlines
\def\hlengthB{0.8cm} % excess length of the arrow-hline
\def\vlengthA{0.1cm} % excess length of the normal vlines
\def\vlengthB{0.4cm} % excess length of the arrow-vline
\matrix (A) [matrix of math nodes, nodes in empty cells %,column sep=-\pgflinewidth, row sep=-\pgflinewidth
] 
{ 4 & 0 & & 0 & 0 & \dots & 0 & 0 & \dots & 0 \\
3 & 0 & T_1 \cdot t_1 & 0 & 0 & \dots & 0 & 0 & \dots & 0  \\
2 & 0 & T_1 \cdot e_1 & 0 & 0 & \dots & 0 & 0 & \dots & 0  \\
1 & 0 & T_1 & 0 & 0 & T_1\cdot e_2,\dots,T_1\cdot e_k & 0 & 0 & T_1\cdot t_2,\dots,T_1\cdot t_k & 0  \\
0 & |[fill, black!20]|\hphantom{0}\vphantom{E_1} & E_1 & |[fill, black!20]|\hphantom{0}\vphantom{E_1} & |[fill, black!20]|\hphantom{0}\vphantom{E_1} & E_1\cdot e_2,\dots,E_1\cdot e_k & |[fill, black!20]|\hphantom{0}\vphantom{E_1} & |[fill, black!20]|\hphantom{0}\vphantom{E_1} & E_1\cdot t_2,\dots,E_1\cdot t_k & |[fill, black!20]|\hphantom{0}\vphantom{E_1} \\
  &  & 0  & & & 1 & & & 2 & \\
};
\tikzvline{A}{1}{6}{1}{triangle 45-, line width=1 pt,shorten <=-\vlengthB};
\tikzvline{A}{1}{5}{2}{black!20, densely dotted,shorten <=-\vlengthA};
\tikzvline{A}{1}{5}{3}{black!20, densely dotted,shorten <=-\vlengthA};
\tikzvline{A}{1}{6}{4}{shorten <=-\vlengthA};
\tikzvline{A}{1}{5}{5}{black!20, densely dotted,shorten <=-\vlengthA};
\tikzvline{A}{1}{5}{6}{black!20, densely dotted,shorten <=-\vlengthA};
\tikzvline{A}{1}{6}{7}{shorten <=-\vlengthA};
\tikzvline{A}{1}{5}{8}{black!20, densely dotted,shorten <=-\vlengthA};
\tikzvline{A}{1}{5}{9}{black!20, densely dotted,shorten <=-\vlengthA};
\tikzvline{A}{1}{6}{10}{shorten <=-\vlengthA};
\tikzhline{A}{1}{1}{10}{shorten >=-\hlengthA};
\tikzhline{A}{2}{1}{10}{shorten >=-\hlengthA};
\tikzhline{A}{3}{1}{10}{shorten >=-\hlengthA};
\tikzhline{A}{4}{1}{10}{shorten >=-\hlengthA};
\tikzhline{A}{5}{1}{10}{-triangle 45, line width=1 pt,shorten >=-\hlengthB};

\node[fit=(A-5-3),ellipse,draw=red, inner sep=-3.5pt]{};
\node[fit=(A-4-3),ellipse,draw=red, inner sep=-3.5pt]{};
\end{tikzpicture}
\caption{$E_2=E_{\infty}$-term of the LHS spectral sequence with $(F\oplus K)^{\otimes G/H}$ coefficients for $p>2$}
\label{Figure-2}
\end{figure}
%-----------------------------------------------------------------------------------%

For $p>2$, there are only two generators $E_1$ and $T_1$ of
the $H^*(G;\FF_p)$-module structure, appearing in positions $E_1\in
E^{0,0}_{\infty}$ and $T_1\in E^{0,1}_{\infty}$, that can have
non-zero image along the differential.  These generators are subject
to the following relation
\[
E_1\cdot t_1=T_1\cdot e_1.
\]
All other generators appearing in the shaded regions of the picture are mapped
to zero by the differential  since they are all
annihilated by $e_1$ and $t_1$ in $H^*(G;\FF_p)$.  In the case $p=2$ there is
only one relevant generator $T_1$ appearing in the position $E^{0,0}_{\infty}$.

\smallskip

\noindent Thus, for every subgroup $H$ of $G$ of order $p$ there are elements $e_H$ and $t_H$ of degree $1$ and $2$ respectively in $H^*(G;\FF_p)$, and
\begin{compactitem}[$\circ$]
\item two generators of the $N$th row of the Serre spectral sequence of the Borel construction
\[
E_{H,d}\in E_2^{0,N}=E_{N+1}^{0,N}\quad \text{and} \quad T_{H,d}\in E_2^{1,N}=E_{N+1}^{1,N}
\]
that satisfy the following relation
\begin{equation}
\label{eq:RelGen}
E_{H,d}\cdot t_H=T_{H,d}\cdot e_H.
\end{equation}
and can be mapped non-trivially by the differential $\partial_{N+1}$ for $p>2$; and there is
\item only one generator of the $N$th row of the Serre spectral sequence of the Borel construction
\[
T_{H,d}\in E_2^{0,N}=E_{N+1}^{0,N}
\]
that might have non-zero image along the differential $\partial_{N+1}$ for $p=2$.
\end{compactitem}

Since we determined all the generators of the $E^{*,N}_2$-row, as an $H^*(G;\FF_p)$-module, we can describe the $(N+2)$nd partial index in the following way.
\begin{lemma}
\label{lem:RelIndexPartial}
\[
\Index{G}^{N+2}(F(\RR^d,p^k);\FF_p)=\left\{
\begin{array}{ll}
	\langle\{\partial_{N+1}(E_{H,d}),~\partial_{N+1}(T_{H,d}) : H<G, |H|=p\}\rangle , & \text{for~}p>2,\\
	\langle\{\partial_{N+1}(T_{H,d}) : H<G, |H|=p\}\rangle ,                         & \text{for~}p=2.
\end{array}
\right.
\]

\end{lemma}

\medskip

%-----------------------------------------------------------------------------------%
\subsubsection{~}
\label{Sec:Sub-5}
We now introduce ingredients that are used in the evaluation of $\partial_{N+1}$ on the generators.

Consider the following real $G$-representation
\[
W_{p^k}=\{(a_1,\ldots,a_{p^k})\in\RR^{p^k}~:~a_1 +\cdots +a_{p^k}=0\}
\]
that is a sub-representation of the regular real $G$-representation $\RR^{p^k}=\bigoplus_{g\in G}g\RR$.
The Euler class, with $\FF_p$ coefficients, $\zeta :=e(W_{p^k})\in H^{p^k}(G;\FF_p)$ of the vector bundle
$W_n\longrightarrow \mathrm{E}G\times_{G}W_n\longrightarrow \BB G$
was computed in~\cite{Mann-Milgram} and for $p>2$ is given by
\[
\zeta=\Big( \prod_{(\alpha_1,\ldots,\alpha_k)\in\FF_p^k{\setminus}\{0\} }(\alpha_1 t_1+\cdots +\alpha_k t_k)\Big)^{\tfrac{1}{2}}
\]
while for $p=2$ we have that
\[
\zeta=\prod_{(\alpha_1,\ldots,\alpha_k)\in\FF_2^k {\setminus}\{0\}} (\alpha_1 t_1+\cdots +\alpha_k t_k).
\]
Since the square root in $\FF_p[t_1,\ldots ,t_k]$ is not uniquely determined we choose for $\zeta$, when $p>2$, an arbitrary square root.
This will not matter when we later consider submodules generated by it.
Following~\cite[Proposition~3.11, page~1338]{B-Z} it can be seen that
\[
\Index{G}(S(W_{p^k});\FF_p)=\langle \zeta \rangle \subseteq H^*(G;\FF_p).
\]

\smallskip

Now let us state the following rather simple observation in a slightly more general setting.

\noindent Let $\rho \colon G\to\mathrm{O}(d)$ define an orthogonal action
of a finite group $G$ on the Euclidean space $E\cong\RR^d$.  In
addition, let $\ArA$ be a $G$-invariant arrangement in $E$ with a
$G$-invariant sub-arrangement $\ArB$.  This means that the arrangement
$\ArB$ is $G$-invariant arrangement and $\ArB\subseteq\ArA$.  Thus, we
have $G$-equivariant maps induced by inclusions
\[
D_{\ArB}\To D_{\ArA} \quad \text{and}  \quad S(E){\setminus} D_{\ArA}\To S(E){\setminus} D_{\ArB}.
\]
Consequently, by the monotonicity property of Fadell--Husseini index, we have that
\begin{equation}
\label{eq:Ineq-Index-Arr}
\Index{G}( S(E){\setminus} D_{\ArA};R)\supseteq\Index{G}( S(E){\setminus} D_{\ArB};R),
\end{equation}
for any ring $R$.

\medskip

%-----------------------------------------------------------------------------------%
\subsubsection{~}
\label{Sec:Sub-6}
As we have seen in Lemma \ref{lem:RelIndexPartial}, to determine the relevant partial index it suffices to determine the elements $\partial_{N+1}(E_{H,d})$ and $\partial_{N+1}(T_{H,d})$ when $p>2$ and the element $\partial_{N+1}(T_{H,d})$ when $p=2$.

Consider the $G=(\ZZ/p)^k$-invariant sub-arrangement
$\mathcal{B}_{H,d}$ of the arrangement $\mathcal{B}_{p^k,d}$ defined by
\[
\mathcal{B}_{H,d}:=\{L_{i,j}\in\mathcal{B}_{p^k,d}~:~L_{i,j}\supseteq V_{H,d}\}
\]
where $V_{H,d}$ was previously introduced in~\eqref{eq:V_H}.
The maximum of the intersection poset of the sub-arrangement $\mathcal{B}_{H,d}$ is the subspace $V_{H,d}=( \RR^{dp^k})^H$.
The complement $S(\RR^{dp^k}){\setminus}D_{\mathcal{B}_{H,d}}$ is $G$-invariant.
Since $G$ acts transitively on the set $[p^k]$, and every element of the group $G$ is of order $p$, we get
\[
F(\RR^d,p^k)=\underset{H\leq G : |H|=p}{\bigcap}S(\RR^{dp^k}){\setminus} D_{\mathcal{B}_{H,d}}
\quad \text{and} \quad
\mathcal{B}_{p^k,d}=\underset{H\leq G : |H|=p}{\bigcup}\mathcal{B}_{H,d}.
\]

In addition consider the subspace $V_{H,d}$ as a one-element
$G$-invariant arrangement $\mathcal{C}_d:=\{V_{H,d}\}$.  As we have
seen, there are $G$-equivariant maps induced by inclusions
\begin{equation}
\label{eq:Inc-1}
S(\RR^{dp^k}){\setminus} D_{\mathcal{B}_{p^k,d}}
\To S(\RR^{dp^k}){\setminus} D_{\mathcal{B}_{H,d}}\To S(\RR^{dp^k}){\setminus} D_{\mathcal{C}_d}.
\end{equation}
If we recall that $F(\RR^d,p^k)\simeq_{G} S(\RR^{dp^k}){\setminus} D_{\mathcal{B}_{p^k,d}}$ and observe that
\[
S(\RR^{dp^k}){\setminus} D_{\mathcal{C}_d}=S(\RR^{dp^k}){\setminus} S(V_{H,d})\simeq_{G} S(V_{H,d}^{\perp})
\]
we obtain the following index inequalities:
\begin{equation}
\label{eq:IndexIn-1}
\Index{G}(F(\RR^d,p^k);\FF_p)\supseteq \Index{G}(S(\RR^{dp^k}){\setminus} D_{\mathcal{B}_{H,d}};\FF_p)
\supseteq \Index{G}(S(V_{H,d}^{\perp});\FF_p).
\end{equation}
Here $V_{H,d}^{\perp}$ denotes the orthogonal complement of the $G$-invariant vector space $V_{H,d}$ inside $\RR^{dp^k}$.
More is true: By the Equivariant Goresky--MacPherson formula the first inclusion of~\eqref{eq:Inc-1} induces an injection of $\FF_p[G]$-modules
\begin{multline}
\tilde{H}_{p-3}(\Delta(\bar{\Pi}_{p});\FF_p)^{\otimes G/H}\cong H^{N}(S(\RR^{dp^k}){\setminus} D_{\mathcal{B}_{H,d}};\FF_p)
\\
\To H^{N}(S(\RR^{dp^k}){\setminus} D_{\mathcal{B}_{p^k,d}};\FF_p)\cong H^{N}(F(\RR^d;p^k);\FF_p).
\end{multline}
Consequently, the generators of the $H^*(G;\FF_p)$-module
\[
H^*(G;\tilde{H}_{p-3}(\Delta(\bar{\Pi}_{p});\FF_p)^{\otimes G/H})\cong H^*(G;H^{N}(S(\RR^{dp^k}){\setminus} D_{\mathcal{B}_{H,d}};\FF_p))
\]
are mapped injectively onto the generators of the $H^*(G;\FF_p)$-module
\[
H^*(G;H^{N}(S(\RR^{dp^k}){\setminus} D_{\mathcal{B}_{p^k,d}};\FF_p))\cong H^*(G;H^{N}(F(\RR^d;p^k);\FF_p))
\]
that correspond to the subgroup $H$ in the decomposition~\eqref{eq:Decomp-2}.
Thus we have computed the index of the complement $S(\RR^{dp^k}){\setminus} D_{\mathcal{B}_{H,d}}$.

\begin{lemma}
\label{lem:RelIndexComplement-02}
\begin{eqnarray}
\Index{G}(S(\RR^{dp^k}){\setminus} D_{\mathcal{B}_{H,d}};\FF_p) &=&\Index{G}^{N+2}(S(\RR^{dp^k}){\setminus} D_{\mathcal{B}_{H,d}};\FF_p)\nonumber\\
\label{lem:RelIndexComplement}
                                                                &=&
\left\lbrace
\begin{array}{ll}
\langle \partial_{N+1}(E_{H,d}),\partial_{N+1}(T_{H,d})\rangle ,   & \text{for }p>2,  \\
\langle \partial_{N+1}(T_{H,d})\rangle ,                           & \text{for }p=2.
\end{array}
\right. 
\end{eqnarray}
\end{lemma}

\medskip
%-----------------------------------------------------------------------------------%
\subsubsection{~}
\label{Sec:Sub-7}
The next step on our path to determine $\Index{G}(S(\RR^{dp^k}){\setminus} D_{\mathcal{B}_{H,d}};\FF_p)$ is to obtain the index of the sphere $S(V_{H,d}^{\perp})$.
According to~\cite[Proposition~3.11, page~1338]{B-Z} it
is enough to find the Euler class of the vector bundle
$V_{H,d}^{\perp}\To\EE G\times_{G}V_{H,d}^{\perp}\To \BB G$.
There is a $G$-equivariant map
\[
S(V_{H,d}^{\perp})\simeq_{G} \RR^{dp^k}{\setminus}V_{H,d}\To \RR^{dp^k}{\setminus}{D_d}\simeq_{G} S(W_{p^k}^{\oplus d})
\]
where $D_d$ is the diagonal $\{(x_1,\ldots,x_{p^k})\in\RR^{dp^k}:x_1=\cdots=x_{p^k}\}$.
In other words, $V_{H,d}^{\perp}$ is a direct summand of $W_{p^k}^{\oplus d}$.
Therefore,
\[
\langle e(V_{H,d}^{\perp}) \rangle=\Index{G}(S(V_{H,d}^{\perp});\FF_p)
\supseteq \Index{G}(S(W_{p^k}^{\oplus d});\FF_p)=\langle e(W_{p^k}^{\oplus d}) \rangle = \langle \zeta^d\rangle,
\]
that is, the Euler class $e(V_{H,d}^{\perp})$ divides the Euler class $e(W_{p^k}^{\oplus d})=\zeta^d$.
The group $H$ acts freely on the sphere $S(V_{H,d}^{\perp})$ via the inclusion homomorphism $H\to G$.
Therefore, $\res^G_H~e(V_{H,d}^{\perp})\neq 0$ in $H^*(H;\FF_p)$.
Using the facts that
\begin{compactitem}
\item $e(V_{H,d}^{\perp})$ divides $e(W_{p^k}^{\oplus d})=\zeta^d$,
\item $\res^G_H e(V_{H,d}^{\perp})\neq 0$ in $H^*(H;\FF_p)$, and
\item $\codim V_{H,d}=\dim V_{H,d}^{\perp}=p^{k-1}(p-1)d=\deg e(V_{H,d}^{\perp})$,
\end{compactitem}
we conclude that $e(V_{H,d}^{\perp})=\zeta_H^d$ where for $p>2$
\[
\zeta_H=\Big( \prod_{\{(\alpha_1,\ldots,\alpha_k)\in\FF_p^k \,:\, \res^G_H(\alpha_1 t_1+\cdots +\alpha_k t_k)\neq 0\}}
(\alpha_1 t_1+\cdots +\alpha_k t_k)\Big)^{\tfrac{1}{2}},
\]
and in case $p=2$
\[
\zeta_H= \prod_{\{(\alpha_1,\ldots,\alpha_k)\in\FF_2^k \,:\, \res^G_H(\alpha_1 t_1+\cdots +\alpha_k t_k)\neq 0\}}
(\alpha_1 t_1+\cdots +\alpha_k t_k).
\]

\smallskip

%-----------------------------------------------------------------------------------%
\subsubsection{~}
\label{Sec:Sub-8}

In order to prove that $\Index{G}^{N+2}(F(\RR^d,p^k);\FF_p)$ does not vanish, according to Lemmas~\ref{lem:RelIndexPartial} and \ref{lem:RelIndexComplement-02},
it suffices to prove that $\Index{G}(S(\RR^{dp^k}){\setminus} D_{\mathcal{B}_{H,d}};\FF_p)$ is not zero for at least one order $p$ subgroup $H$ of $G$.

Consider the subgroup $H$ of $G$ determined by the presentation of $H^*(G;\FF_p)$ as given in \eqref{eq:presentation}:
\begin{eqnarray}
 H^*(G;\FF_p) &\cong&\big(\FF_p[t_2,\ldots ,t_k]\otimes\Lambda[e_2,\ldots ,e_k]\big)\otimes\big(\FF_p[t_1]\otimes\Lambda[e_1]\big)\nonumber \\
              &\cong&\big(\FF_p[t_2,\ldots ,t_k]\otimes\Lambda[e_2,\ldots ,e_k]\big)\otimes H^*(H;\FF_p).\nonumber
\end{eqnarray}
The complement of the arrangement $S(\RR^{dp^k}){\setminus} D_{\mathcal{B}_{H,d}}$ is a fixed point free space with respect to the action of the group $G$.
Indeed, the link of the arrangement $D_{\mathcal{B}_{H,d}}$ contains the $H$-fixed point set $S(\RR^{dp^k})^H$ of the sphere $S(\RR^{dp^k})$.
Now the cohomological characterization of the fixed point set for the action of an elementary abelian group \cite[Proposition 3.14, page 196]{tDieck} implies that
\[
\Index{G}(S(\RR^{dp^k}){\setminus} D_{\mathcal{B}_{H,d}};\FF_p)\neq 0.
\]
Consequently, Lemma~\ref{lem:RelIndexComplement} implies that
\begin{compactitem}
\item $\partial_{N+1}(E_{H,d})\neq 0$ or $\partial_{N+1}(T_{H,d})\neq 0$ for $p>2$, and
\item $\partial_{N+1}(T_{H,d})\neq 0$ for $p=2$.
\end{compactitem}
Thus the first part of the theorem is proved for all $p$, while the second one is proved only for $p=2$.

\medskip

In order to complete the proof of the second part of the theorem we need to prove that $\partial_{N+1}(E_{H,d})\neq 0$.
In the case $p>2$ the second inclusion of~\eqref{eq:IndexIn-1} gives 
\[
\langle \partial_{N+1}(E_{H,d}),\partial_{N+1}(T_{H,d})  \rangle
= \Index{G}(S(\RR^{dp^k}){\setminus} D_{\mathcal{B}_{H,d}};\FF_p)\supseteq \Index{GS}(S(V_{H,d}^{\perp});\FF_p)=\langle \zeta_H^d\rangle .
\]
The class $\zeta_H^d\in\FF_p[t_1,\ldots,t_k]$ is a polynomial and belongs to the index $\Index{G}(S(\RR^{dp^k}){\setminus} D_{\mathcal{B}_{H,d}};\FF_p)$.
Therefore, the relation~\eqref{eq:RelGen} implies that
\begin{compactitem}
\item $\partial_{N+1}(T_{H,d})\in \FF_p[t_1,\ldots,t_k]$ is a polynomial,
\item $\partial_{N+1}(T_{H,d})$ divides $\zeta_H^d$, and
\item $\partial_{N+1}(E_{H,d})\cdot t_H=\partial_{N+1}(T_{H,d})\cdot e_H$.
\end{compactitem}
Since the multiplication by $e_H$ on the set of polynomials in $E^{*,0}_{N+1}=H^*(G;\FF_p)$ is an injection, the assumption $\partial_{N+1}(E_{H,d})=0$ yields the following contradiction:
\[
0\neq \partial_{N+1}(T_{H,d})\cdot e_H = \partial_{N+1}(E_{H,d})\cdot t_H=0.
\]
Thus, $\partial_{N+1}(E_{H,d})\neq 0$ and Theorem \ref{Th:IndexN_plus_3} is proved.

%-----------------------------------------------------------------------------------%
\section{A few applications}
\label{Sec:App}
 
In this section we present four applications of the methods we have developed.
Further applications and extensions of the results below can be found in \cite{blagojevic-luck-ziegler-2} and \cite{blagojevic-cohen-luck-ziegler}.

\subsection{The Nandakumar \& Ramana Rao conjecture}
\label{Sec:App-NRR}

The problem of Nandakumar \& Ramana Rao, posed in 2006 in~\cite{Nandakumar06}, 
can be stated as follows.

\begin{conjecture}[Nandakumar \& Ramana Rao]
  For a given planar convex body $K$ and any natural number $n>1$ there exists a
  partition of the plane into $n$ convex pieces $P_1,\ldots ,P_n$ such that
  \[
  \mathrm{area}(P_1\cap K)=\cdots =\mathrm{area}(P_n\cap K) \quad \text{and}
  \quad \mathrm{perimeter}(P_1\cap K)=\cdots =\mathrm{perimeter}(P_n\cap K).
  \]
\end{conjecture}

Nandakumar \& Ramana Rao~\cite{NandaKumarRamanaRao12} gave the answer for $n=2$,
relying on the intermediate value theorem.  The case of $n=3$ was resolved by
B\'ar\'any, Blagojevi\'c \& Sz\H{u}cs~\cite{BaranyBlagojevicSzuecs} using more
advanced topological methods.

Let $\conve(\RR^d)$ denotes the metric space of all $d$-dimensional convex bodies in $\RR^d$ with
the Hausdorff metric.  The conjecture of Nandakumar \& Ramana Rao can be
naturally generalized in the following way.

\begin{conjecture}[Generalized Nandakumar \& Ramana Rao]
  For a given convex body $K$ in $\RR^d$, an absolutely continuous
  probability measure $\mu$ on $\RR^d$, any natural number $n>1$ and any $n-1$
  continuous functions $\varphi_1,\ldots
  ,\varphi_{n-1}\colon\conve(\RR^d)\to\RR$, there exists a partition of $\RR^d$
  into $n$ convex pieces $P_1,\ldots ,P_n$ such that
  \[
  \mu(P_1\cap K)=\cdots =\mu(P_n\cap K)
  \]
  and for every $i\in\{1,\ldots ,n-1\}$
  \[
  \varphi_i(P_1\cap K)=\cdots =\varphi_i(P_n\cap K).
  \]
\end{conjecture}

The next steps in solving both the original and the generalized Nandakumar \& Ramana
Rao conjecture were done first by Karasev~\cite{Karasev:equipartition}, Hubard \& Aronov
\cite{HubardAronov} and Blagojevi\'c \& Ziegler in~\cite{B-Z-New}.  They
observed that both conjectures would hold if there is no $\Sym_n$-equivariant
map $F(\RR^d,n)\to S(W_n^{\oplus(d-1)})$.  Here $W_n$ denotes the
$\Sym_n$-representation $\{(x_1,\ldots ,x_n)\in\RR^n : x_1+\cdots +x_n=0\}$
where the action is given by permuting the coordinates.

In this section, we offer a proof that both the original and the generalized
Nandakumar \& Ramana Rao conjecture holds for $n$ a prime number by proving the
following theorem.

\begin{theorem}
  \label{th:GNRR}
  Let $n=p$ be a prime number, $d\geq 2$ and $\ZZ/p$ the subgroup of $\Sym_p$
  generated by the permutation $(12\ldots p)$.  Then there is no
  $\ZZ/p$-equivariant map
  \[
  F(\RR^d,p)\to S(W_p^{\oplus(d-1)}).
  \]
  Consequently, there is no $\Sym_p$-equivariant map $F(\RR^d,p)\to S(W_p^{\oplus(d-1)})$.
\end{theorem}
\begin{proof}
  Let us assume that there exists a $\ZZ/p$-equivariant map $f\colon
  F(\RR^d,p)\to S(W_p^{\oplus(d-1)})$.  
  Then according to the Monotonicity property of the Fadell--Husseini index, Section \ref{subsec:FH-indexDef}, we have that
  \[
  \Index{\ZZ/p}(F(\RR^d,p);\FF_p)\supseteq \Index{\ZZ/p}(S(W_p^{\oplus(d-1)})).
  \]
  The index of the configuration space $F(\RR^d,p)$ was computed in 
  Theorem~\ref{Th:IndexF(X,p)} and
  \begin{equation*}
    \Index{\ZZ/p}(F(\RR^d,p);\FF_p)=H^{\geq (d-1)(p-1)+1}(\ZZ/p;\FF_p)=
    \left\{
      \begin{array}{ll}
        \langle t^{(d-1)(p-1)+1}\rangle , & \text{~for~}p=2, \\
        \langle e\,t^{\tfrac{(d-1)(p-1)}{2}}, t^{\tfrac{(d-1)(p-1)}{2}+1}\rangle , & \text{~for~}p\text{~odd}.
      \end{array}
    \right.
  \end{equation*}
  On the other hand, the sphere $S(W_p^{\oplus(d-1)})$ is a free $\ZZ/p$-space
  and therefore
  \[
  \EE\ZZ/p\times_{\ZZ/p} S(W_p^{\oplus(d-1)}) \simeq S(W_p^{\oplus(d-1)})/(\ZZ/p)
  \]
  and consequently
  \[
  H^{\ell}(\EE\ZZ/p\times_{\ZZ/p} S(W_p^{\oplus(d-1)}))=0\; \text{for all}\; \ell
  > \dim S(W_p^{\oplus(d-1)})=(d-1)(p-1)-1.
  \]
  Therefore,
  \begin{eqnarray*}
    \Index{\ZZ/p}(S(W_p^{\oplus(d-1)});\FF_p)  &    =    & 
   \ker\left(H^*(\BB\ZZ/p;\FF_p)\to H^*(\EE\ZZ/p\times_{\ZZ/p} S(W_p^{\oplus(d-1)});\FF_p)\right)\\
    &\supseteq& H^{\geq (d-1)(p-1)}(\BB\ZZ/p;\FF_p).
  \end{eqnarray*}
  In particular, $t^{\frac{(d-1)(p-1)}{2}}\in\Index{\ZZ/p}(S(W_p^{\oplus
    d-1});\FF_p)$ for $p>2$, and $t^{(d-1)(p-1)}\in\Index{\ZZ/p}(S(W_p^{\oplus
    d-1});\FF_p)$ for $p=2$.  This is a contradiction to the Monotonicity property
  of the Fadell--Husseini index since
  \[
  t^{\frac{(d-1)(p-1)}{2}}\notin \Index{\ZZ/p}(F(\RR^d,p);\FF_p)
  \]
  in case $p>2$ and for $p=2$
  \[
  t^{(d-1)(p-1)}\notin \Index{\ZZ/p}(F(\RR^d,p);\FF_p).
  \]
  Thus, for every $d\geq 2$ there is no $\ZZ/p$-equivariant map $F(\RR^d,p)\to
  S(W_p^{\oplus(d-1)})$.
\end{proof}

\subsection{The Lusternik--Schnirelmann category of unordered configuration spaces}
\label{Sec:LS}

In this section using the results from 
Section~\ref{subsec:Lusternik-Schnirelmann_Category} and~\cite{B-Z-New} we study the
Lusternik--Schnirelmann category of unordered configuration spaces
$F(\RR^d,n)/\Sym_n$ as well as sectional category of the covering $F(\RR^d,n)\to F(\RR^d,n)/\Sym_n$.

\begin{theorem}
  \label{th:LS-ConfigSpaces-Rd}
  \qquad
  \begin{compactenum}[\rm (1)]
  \item Let $p$ be an odd prime and $n=p^k$ for some $k\geq 1$.  Then for every
    $d\geq 2$
    \[
    \cat (F(\RR^d,n)/\Sym_n)=(d-1)(n-1).
    \]
  \item Let $n=2^k$ for some $k\geq 1$.  Then for every odd $d\geq 3$
    \[
    \cat (F(\RR^d,n)/\Sym_n)=(d-1)(n-1).
    \]
  \end{compactenum}
\end{theorem}
\begin{proof}
  The Lusternik--Schnirelmann category is a homotopy invariant, therefore
  instead of the configuration space $F(\RR^d,n)$ we can consider any
  $\Sym_n$-CW model.  Here we use the $(d-1)(n-1)$-dimensional model
  $\mathcal{F}(d,n)$ derived in~\cite{B-Z-New}.  It was proved in~\cite[Theorem~1.2]{B-Z-New} 
  that there exists an $\Sym_n$-equivariant map
  $\mathcal{F}(d,n)\to S(W_n^{\oplus(d-1)})$ if and only if $n$ is not a prime
  power.  Moreover, in~\cite[Corollary~4.3]{B-Z-New} the following equivariant
  cohomology group was calculated:
  \[
  H^{(d-1)(n-1)}_{\Sym_n}\big(\mathcal{F}(d,n); \pi_{(d-1)(n-1)-1}(S(W_n^{\oplus(d-1)}))\big)\ =\
  \begin{cases}
    \ZZ/p & \text{if} \; n=p^k\; \text{is a prime power},\\
    0 & \text{otherwise.}
  \end{cases}
  \]
  Since under both sets of assumptions (1) and (2) the action of the $p$-Sylow subgroup
  $\Sym_n^{(p)}$ preserves orientation on $S(W_n^{\oplus(d-1)})$, we can
  apply Theorem~\ref{the:estimate_for_cat}(1) and conclude the
  proof.
\end{proof}

\begin{corollary}
  \label{cor:LS-ConfigSpaces-1}
  Let $p$ be a prime, $M$ be any topological space and $f:\RR^d\to M$ be an injective
  continuous map.
  \begin{compactenum}[\rm (1)]
  \item If $p$ is an odd prime and $n=p^k$ for some $k\geq 1$, then for every
    $d\geq 2$
    \[
    \cat(F(M,n)/\Sym_n)\geq (d-1)(n-1).
    \]
  \item If $n=2^k$ for some $k\geq 1$, then for every odd integer $d\geq 3$
  \end{compactenum}
  \[
  \cat(F(M,n)/\Sym_n)\geq (d-1)(n-1).
  \]
\end{corollary}
\begin{proof}
  Since $f$ is an injective map it induces an $\Sym_n$-equivariant map
  $f^n:F(\RR^d,n)\to F(M,n)$ defined by
  \[
  f^n(x_1,\ldots , x_n):=(f(x_1),\ldots , f(x_n)).
  \]
  The claim of the corollary now follows from Theorems~\ref{the:estimate_for_cat}(2) and \ref{th:LS-ConfigSpaces-Rd}.
\end{proof}

\subsection{Existence of equivariant maps}
\label{Sec:Z/n}

In~\cite[Theorem~1.2]{B-Z-New}, it was proved that there exists an
$\Sym_n$-equivariant map $F(\RR^d,n)\to S(W_n^{\oplus(d-1)})$ if and only if $n$
is not a prime power.

Let $G:=\ZZ/n$ be the cyclic subgroup of $\Sym_n$ generated by the permutation
$(1~2\ldots n)$.  In this section we consider the question when there exists a
$\ZZ/n$-equivariant map $F(\RR^d,n)\to S(W_n^{\oplus(d-1)})$ where the action is
induced via the inclusion $\ZZ/n\to\Sym_n$.

\begin{theorem}
\label{th:ExZ/n}
Let $n\geq 2$ and $d\geq 2$ be integers.
Then a $\ZZ/n$-equivariant map $F(\RR^d,n)\to S(W_n^{\oplus(d-1)})$ exists if and only if $n$ is not a prime.
\end{theorem}

\begin{proof}
We will prove the following claims: 
\begin{compactenum} %[\rm (1)]
  \item If $n=p$ is a prime, then there is no $\ZZ/n$-equivariant map $F(\RR^d,n)\to S(W_n^{\oplus(d-1)})$;
  \item if $n$ is not a prime power, then a $\ZZ/n$-equivariant map $F(\RR^d,n)\to S(W_n^{\oplus(d-1)})$ exists;
  \item if $n=p^k$ is a power of an odd prime $p$ with $k>1$, and $(n,d)\neq (9,2)$, then a $\ZZ/n$-equivariant map $F(\RR^d,n)\to S(W_n^{\oplus(d-1)})$ exists;
  \item if $n=2^k$ for $k>1$, and $d>1$ is odd, then a $\ZZ/n$-equivariant map $F(\RR^d,n)\to S(W_n^{\oplus(d-1)})$ exists;
  \item if $d,d'\,{\geq}\,2$, then a $\ZZ/n$-equivariant map $F(\RR^d,n)\to S(W_n^{\oplus(d-1)})$ exists if and only if 
			a $\ZZ/n$-equivariant map $F(\RR^{d'},n)\to S(W_n^{\oplus(d'-1)})$ exists.
\end{compactenum}
The last claim implies the existence of $\ZZ/n$-equivariant maps in the ``prime power'' cases not covered by the claims (3) and (4), and 
thus completes the proof of the theorem.
\smallskip

\noindent\textrm{(i)} This is the content of Theorem~\ref{th:GNRR}.

\noindent\textrm{(ii)} Let us assume that $n$ is not a prime power and denote $M:=(d-1)(n-1)$.
There exists an $M$-dimensional $\ZZ/n$-CW model $\mathcal{F}(d,n)$ of the configuration space $F(\RR^d,n)$ derived in~\cite{B-Z-New}. 
Therefore it suffices to prove the existence of a $\ZZ/n$-equivariant map $\mathcal{F}(d,n)\to S(W_n^{\oplus(d-1)})$.
Since
  \begin{compactitem}
  \item $\mathcal{F}(d,n)$ is an $M$-dimensional free $\ZZ/n$-CW complex,
  \item the dimension of the $\ZZ/n$-sphere $S(W_n^{\oplus(d-1)})$ is $M-1$, and
  \item the sphere $S(W_n^{\oplus(d-1)})$ is $(M-1)$-simple and
    $(M-2)$-connected,
  \end{compactitem}
the existence of a $\ZZ/n$-equivariant map $\mathcal{F}(d,n)\to S(W_n^{\oplus(d-1)})$ is equivalent to the vanishing of the primary equivariant obstruction element
  \[
  \gamma^{\ZZ/n}:=\gamma^{\ZZ/n}(\mathcal{F}(d,n),S(W_n^{\oplus(d-1)}))\in\
  H^{M}_{\ZZ/n}\big(\mathcal{F}(d,n); \pi_{M-1}(S(W_n^{\oplus(d-1)}))\big).
  \]
For each subgroup $H\leq \ZZ/n$ the restriction $\res (\gamma^{\ZZ/n})=\gamma^H$ is the primary equivariant obstruction element with respect to the group $H$.  
Since for each non-trivial subgroup $H\leq \ZZ/n$ the set of $H$ fixed points $S(W_n^{\oplus(d-1)})^H$ of the sphere $S(W_n^{\oplus(d-1)})$ is non-empty, 
we have $\res (\gamma^{\ZZ/n})=\gamma^H = 0$.
Consequently, for each non-trivial subgroup $H\leq \ZZ/n$
  \begin{equation}
    \label{eq:SubResTr}
    [\ZZ/n:H]\cdot\gamma^{\ZZ/n}= \trf\circ\res (\gamma^{\ZZ/n}) = 0.
  \end{equation}
The restriction and transfer we use here are considered in Lemma~\ref{lem:transfer}.  
Thus, the equivariant obstruction element vanishes, i.e., $\gamma^{\ZZ/n}=0$.  
Therefore, for $n$ not a prime power there exists a $\ZZ/n$-equivariant map $F(\RR^d,n)\to S(W_n^{\oplus(d-1)})$.

\noindent\textrm{(iii)-(iv)} As in the previous case, the existence of the $\ZZ/n$-equivariant map $F(\RR^d,n)\to S(W_n^{\oplus(d-1)})$ is completely 
determined by the primary obstruction
  \[
  \gamma^{\ZZ/n}:=\gamma^{\ZZ/n}(\mathcal{F}(d,n),S(W_n^{\oplus(d-1)}))\in\
  H^{M}_{\ZZ/n}\big(\mathcal{F}(d,n); \pi_{M-1}(S(W_n^{\oplus(d-1)}))\big).
  \]
In all the cases we consider $M$ is even and the coefficient $\ZZ/n$-module $\pi_{M-1}(S(W_n^{\oplus(d-1)}))\cong\ZZ$ is trivial. 
Thus, by Lemma~\ref{lem:SW_and_o}, the primary obstruction $\gamma^{\ZZ/n}$ coincides with the Euler class $e(\xi_{\ZZ/n})$ of the vector bundle $\xi_{\ZZ/n}$:
\[
W_n^{\oplus(d-1)}\rightarrow \mathcal{F}(d,n)\times_{\ZZ/n}W_n^{\oplus(d-1)}\rightarrow \mathcal{F}(d,n)/(\ZZ/n).
\]
The Euler class $e(\xi_{\ZZ/n})$ is the pullback of the Euler class $e(W_n^{\oplus(d-1)})\in H^{M}(\BB\ZZ/n;\ZZ)$ of the vector bundle
\[
W_n^{\oplus(d-1)}\rightarrow \EE\ZZ/n\times_{\ZZ/n}W_n^{\oplus(d-1)}\rightarrow \BB\ZZ/n
\]
along the classifying map $f\colon\mathcal{F}(d,n)\to\EE\ZZ/n$.
We will now show that in all cases covered by  (3) and (4) the element $e(\xi_{\ZZ/n})=f^*(e(W_n^{\oplus(d-1)}))$ vanishes, by proving that $e(W_n^{\oplus(d-1)})=0$.

For every cyclic group the reduction of coefficient ring homomorphism $\ZZ\to\ZZ/n$ induces an isomorphism
\[
H^{2i}(\ZZ/n;\ZZ)\cong H^{2i}(\ZZ/n;\ZZ/n)
\]
for every $i\ge1$.
Consequently, the Euler class $e(W_n^{\oplus(d-1)})$ vanishes if and only if its $\ZZ/n$ reduction $e_{\ZZ/n}(W_n^{\oplus(d-1)})$ vanishes.

The existence of the isomorphism of $\ZZ/n$-representations $W_n\oplus\RR\cong\RR[\ZZ/n]$ allows us to decompose $W_n$ into irreducible ones.
Here $\RR$ denotes the trivial $1$-dimensional real $\ZZ/n$-representation, while $\RR[\ZZ/n]$ denotes the regular one.
Now we can apply \cite[Theorem 3.3, page 285]{Izy-Mar} to get that $e_{\ZZ/n}(W_n^{\oplus(d-1)})$ vanishes in all the cases we consider.
Consequently, both the Euler class $e(W_n^{\oplus(d-1)})$ and the primary obstruction $\gamma^{\ZZ/n}$ vanish, so
the $\ZZ/n$-equivariant map $F(\RR^d,n)\to S(W_n^{\oplus(d-1)})$ exists.

\noindent\textrm{(v)} 
The existence of a $\ZZ/n$-equivariant map $F(\RR^d,n)\to S(W_n^{\oplus (d-1)})$ can be discussed via the equivariant obstruction theory set-up of~\cite{B-Z-New}.
For example, in the case $n=4$ the existence of a $\ZZ/4$-equivariant map $F(\RR^d,4)\to S(W_4^{\oplus (d-1)})$ is equivalent to the solvability of the following system of linear equations in integer variables $x_{[\ldots ]}$:
\begin{eqnarray*}
x_{[1|234]} + x_{[1|342]} + x_{[1|243]} + x_{[1|234]} + x_{[12|34]} + x_{[13|24]} + x_{[14|23]} +   & & \\
x_{[14|23]} + x_{[13|24]} + x_{[12|34]} + x_{[123|4]} + x_{[124|3]} + x_{[134|2]} + x_{[123|4]}     &=&1\\
x_{[1|243]} + x_{[1|432]} + x_{[1|243]} + x_{[1|234]} + x_{[12|43]} + x_{[13|24]} + x_{[14|23]} +   & & \\
x_{[14|23]} + x_{[13|24]} + x_{[12|43]} + x_{[123|4]} + x_{[124|3]} + x_{[143|2]} + x_{[124|3]}     &=&1\\
x_{[1|324]} + x_{[1|342]} + x_{[1|243]} + x_{[1|324]} + x_{[12|34]} + x_{[13|24]} + x_{[14|32]} +   & & \\
x_{[14|32]} + x_{[13|24]} + x_{[12|34]} + x_{[132|4]} + x_{[124|3]} + x_{[134|2]} + x_{[132|4]}     &=&1\\
x_{[1|342]} + x_{[1|342]} + x_{[1|423]} + x_{[1|324]} + x_{[12|34]} + x_{[13|42]} + x_{[14|32]} +   & & \\
x_{[14|32]} + x_{[13|42]} + x_{[12|34]} + x_{[132|4]} + x_{[142|3]} + x_{[134|2]} + x_{[134|2]}     &=&1\\
x_{[1|423]} + x_{[1|432]} + x_{[1|423]} + x_{[1|234]} + x_{[12|43]} + x_{[13|42]} + x_{[14|23]} +   & & \\
x_{[14|23]} + x_{[13|42]} + x_{[12|43]} + x_{[123|4]} + x_{[142|3]} + x_{[143|2]} + x_{[142|3]}     &=&1\\
x_{[1|432]} + x_{[1|432]} + x_{[1|423]} + x_{[1|324]} + x_{[12|43]} + x_{[13|42]} + x_{[14|32]} +   & & \\
x_{[14|32]} + x_{[13|42]} + x_{[12|43]} + x_{[132|4]} + x_{[142|3]} + x_{[143|2]} + x_{[143|2]}     &=&1.
\end{eqnarray*}

Generally, the existence of a $\ZZ/n$-equivariant map $F(\RR^d,n)\to S(W_n^{\oplus(d-1)})$ is equivalent to the vanishing of the equivariant obstruction cocycle calculated in~\cite[Lemma 4.1]{B-Z-New}.
Here we assume the natural restriction of the $\Sym_n$ action to the $\ZZ/n$ action.
The vanishing of the equivariant obstruction cocycle is equivalent to the integer solvability of a system of $(n-1)!$ linear equations (generated by the facets of the model $\mathcal{F}(d,n)$)
in $(n-1)!(n-1)$ integer variables (corresponding to the ridges of the model $\mathcal{F}(d,n)$).
Since the system of the equations does not depends on $d$, as proved in~\cite[Lemma 4.1]{B-Z-New}, we conclude that the answer, that is, the existence of the
$\ZZ/n$-equivariant map $F(\RR^d,n)\to S(W_n^{\oplus(d-1)})$, does not depend on $d$, for $d\ge2$.
\end{proof}

In the proof of Theorem \ref{th:ExZ/n}(2) we did not need the condition \eqref{eq:SubResTr} to be satisfied for all non-trivial subgroups of $\ZZ/n$.  
It suffices that the condition \eqref{eq:SubResTr} holds for some choice of $p$-Sylow subgroups of a group where $p$ ranges over all prime divisors of the group order. 
Having this in mind the proof of the statement (2) of the previous theorem with small modification yields the proof of the sufficiency part of~\cite[Theorem 1.2]{B-Z-New}:
\begin{theorem}
If $n$ is not a prime power then there exists an $\Sym_n$-equivariant map
\[
F(\RR^d,n)\to S(W_n^{\oplus(d-1)}).
\]
\end{theorem}

Let us point out that the first reasoning of this type was used by \"Ozaydin in his remarkable unpublished paper~\cite{Ozay}.

\subsection{On coincidence theorems of Cohen, Connett, and Lusk}

In this section we use calculations presented in Section~\ref{Sec:FH-Index-II} 
to extend/improve the results of Cohen \& Connett \cite[Theorem 1, page 218]{Cohen-Connett},
Cohen \& Lusk \cite[Theorem 1 and 2, page 245]{Cohen-Lusk}, and Karasev \& Volovikov \cite[Theorem 7.1, page 1050]{Karasev-Volovikov}.

In this section $p$ is a prime, $k>1$ is an integer, and $G$ denotes the elementary abelian group $(\ZZ/p)^k$.
Moreover, the class of $\FF_p[G]$-modules $\mathfrak{FI}_G$, that we use in theorem that follows, was introduced in Section~\ref{Sec:DiffSSSeq}.

\begin{theorem}
\label{Th:Coincidence-1}
Let $X$ be a Hausdorff free $G$-space, $d\geq 2$ be an integer, and $f\colon X\to\RR^d$ be a continuous map.
If $H^i(X;\FF_p)\in\mathfrak{FI}_G$ for every $0<i \le (d-1)(p^k-p^{k-1})$, then there exist $x\in X$ and $g\in G{\setminus}\{1\}$ such that
\[
f(x)=f(g\cdot x).
\]
\end{theorem}
\begin{proof}
Let us consider a continuous $G$-equivariant map $\psi' \colon X\to\ind^G_{\{1\}}(\RR^d)$ defined by
\[
x\mapsto \sum_{g\in G}f(g^{-1}\cdot x)g.
\]
Here we denote by $\ind^G_{\{1\}}(\RR^d)\cong(\RR^d)^{p^k}$ the induced $G$-representation from the trivial $G$-representation $\RR^d$.

\noindent If the theorem is not true, then for every $x\in X$ and every $g\in G{\setminus}\{1\}$ we have that $f(x)\neq f(g\cdot x)$.
Consequently, the map $\psi'$ factors through the configuration space $F(\RR^d,p^k)$, i.e., the following diagram of $G$-equivariant maps commutes:
\[
\xymatrix
{
X\ar[rd]^{\psi}\ar[rr]^--{\psi'} &   & \ind^G_{\{1\}}(\RR^d)\cong(\RR^d)^{p^k}\\
&F(\RR^d,p^k).\ar[ur]^{i}&
}
\]
Now using the Fadell--Husseini index we prove that the $G$-equivariant map $\psi\colon X\to F(\RR^d,p^k)$ cannot exist.

From the assumption that $H^i(X;\FF_p)\in\mathfrak{FI}_G$ for $0<i\le (d-1)(p^k-p^{k-1})$, using Theorem~\ref{Th:DiffSSSeq-EAb}, we conclude that
\[
\Index{G}(X;\FF_p)\subseteq H^{\geq (d-1)(p^k-p^{k-1})+2}(\BB G;\FF_p).
\]
On the other hand, Theorem~\ref{Th:IndexN_plus_3} implies that there exists a non-zero element
\[
v\in \Index{G}(F(\RR^d,p^k);\FF_p)\cap H^{(d-1)(p^k-p^{k-1})+1}(\BB G;\FF_p).
\]
Therefore, the existence of the element $v$ implies that
\[
\Index{G}(F(\RR^d,p^k);\FF_p)\nsubseteq \Index{G}(X;\FF_p)
\]
and consequently, by the basic monotonicity property of the Fadell--Husseini index, there cannot be a $G$-equivariant map $\psi \colon X\to F(\RR^d,p^k)$.
This concludes the proof of the theorem.
\end{proof}

Let $X$ be a Hausdorff free $G$-space, $d>2$ be an integer and $f:X\to\RR^d$ be a continuous map.
Set
\[
A(X,f):=\{x\in X\colon f(x)=f(g\cdot x)\text{~for some~}g\in G{\setminus}\{1\}\}.
\]
Along the lines of the proofs of Cohen \& Lusk \cite[Theorem 2]{Cohen-Lusk} and our Theorem~\ref{Th:Coincidence-1} the following result can be obtained.

\begin{theorem}
\label{Th:Coincidence-2}
Let $X$ be a closed $m$-dimensional manifold equipped with a free $G$-action, $d>2$ be an integer and $f\colon X\to\RR^d$ be a continuous map.
If $H^i(X;\FF_p)\in\mathfrak{FI}_G$ for $0<i\le (d-1)(p^k-p^{k-1})$, then the covering dimension of the set $A(X,f)$ can be estimated from below by
\[
\mathrm{cdim}\, A(X,f)\geq m-(d-1)(p^k-p^{k-1})-1 .
\]

\end{theorem}

\providecommand{\noopsort}[1]{}

\end{document}